\documentclass{article}
\usepackage[utf8]{inputenc}
\usepackage[margin=1in]{geometry} 
\usepackage{amsmath,amsthm,amssymb,amsfonts}
\usepackage{physics, bm, bbm, IEEEtrantools,}
\usepackage[mathscr]{eucal}
\usepackage[ruled]{algorithm2e}    % For typesetting algorithms
\usepackage{empheq} % Used for emphasising equations with braces and such
\usepackage{multirow, multicol}

\usepackage{authblk, orcidlink, graphicx}

\usepackage{caption, subcaption}
\usepackage[shortlabels]{enumitem}

\hypersetup{colorlinks = true,linkcolor = blue}

\usepackage[square,numbers]{natbib}
\newtheorem{theorem}{Theorem}[section]  % Theorems
\newtheorem{lemma}[theorem]{Lemma}  % Lemmas
\newtheorem{assum}{Assumption}
\newtheorem{prop}{Proposition}[section] % Propositions
\theoremstyle{definition}
\newtheorem{definition}{Definition}[section]    % Definitions
\theoremstyle{remark}
\newtheorem{remark}[theorem]{Remark}    % Remarks

\DeclareMathOperator{\dom}{dom}
\DeclareMathOperator{\ran}{ran}

\SetKwInput{Initialise}{Initialise}
\SetKwInput{Compute}{Compute}
\SetKwInput{Parameters}{Parameters}

\title{A decentralised forward-backward-type algorithm with network-independent heterogenous agent step sizes}
\author[1,2]{Matthew K. Tam\footnote{Email: matthew.tam@unimelb.edu.au}}
\author[1,2]{Liam Timms\footnote{Corresponding author: l.timms@student.unimelb.edu.au}}
\author[1,2]{Lele Zhang\footnote{Email: lele.zhang@unimelb.edu.au}}
\affil[1]{School of Mathematics and Statistics, University of Melbourne, Parkville 3052, Australia}
\affil[2]{ARC Training Centre in Optimisation Technologies, Integrated Methodologies, and Applications (OPTIMA)}
\date{}

\begin{document}

\maketitle

\begin{abstract}
    Consider the problem of finding a zero of a finite sum of maximally monotone operators, where some operators are Lipschitz continuous and the rest are potentially set-valued. We propose a forward-backward-type algorithm for this problem suitable for decentralised implementation. In each iteration, agents evaluate a Lipschitz continuous operator and the resolvent of a potentially set-valued operator, and then communicate with neighbouring agents. Agents choose their step sizes independently using only local information, and the step size upper bound has no dependence on the communication graph. We demonstrate the potential advantages of the proposed algorithm with numerical results for min-max problems and aggregative games.
    
    
   
\end{abstract}

\section{Introduction}  \label{sec:intro}
Let $\mathcal{H}$ be a real Hilbert space. Consider a connected graph of $N$ agents coordinating to find a zero of a finite sum of monotone operators given by
\begin{equation}    \label{eqn:finite_sum_inclusion}
    \text{find~} z\in\mathcal{H} \text{~s.t.~} 0\in \sum_{i=1}^N (A_i + B_i) (z),
\end{equation}
where $A_i:\mathcal{H} \rightrightarrows \mathcal{H}$ is maximally monotone (potentially set-valued) and $B_i:\mathcal{H} \to \mathcal{H}$ is monotone and $L_i$-Lipschitz continuous for $i=1,\dots,N$. In this work, we consider a decentralised distributed setting where agent $i$ can evaluate $B_i$ and the resolvent of $A_i$, and agents communicate with each other if connected by an edge in the graph. Problems of this kind arise, for example, when the data is partitioned into subsets and allocated to agents, such as route choice in road networks \citep{paper:paccagnan_nash_wardrop}, binary classification \citep{paper:alghunaim_decentral_proxgrad}, and generative adversarial networks~\cite{paper:beznosikov_saddlepoint}. Of particular importance to this work are \emph{min-max problems} of the form
\begin{equation}    \label{eqn:example_minmax}
    \displaystyle\min_{u\in \mathcal{H}_1} \max_{v\in \mathcal{H}_2} \sum_{i=1}^N f_i(u) + g_i(u, v) - h_i(v),
\end{equation}
where $\mathcal{H}_1$, $\mathcal{H}_2$ are real Hilbert spaces, $f_i:\mathcal{H}_1\to\mathbb{R}\cup\{+\infty\}$ and $h_i:\mathcal{H}_2\to\mathbb{R}\cup\{+\infty\}$ are proper lsc convex functions, and $g_i:\mathcal{H}_1\times\mathcal{H}_2\to\mathbb{R}$ is a convex-concave function with Lipschitz continuous gradient. Assuming~\eqref{eqn:example_minmax} has a saddle-point and the subdifferential sum rule holds for $\sum_{i=1}^N f_i$ and $\sum_{i=1}^N h_i$, solutions of~\eqref{eqn:example_minmax} can be characterised as solutions of~\eqref{eqn:finite_sum_inclusion} with
\begin{equation}   \label{eqn:example_minmax_operators}
z=\begin{pmatrix}u\\v\end{pmatrix}\in\mathcal{H} = \mathcal{H}_1\times\mathcal{H}_2,\quad A_i(z) = \begin{pmatrix} \partial f_i(u) \\ \partial h_i(v)\end{pmatrix},\quad B_i(z) = \begin{pmatrix}\grad_u g_i(u, v) \\ -\grad_v g_i(u,v)\end{pmatrix}.
\end{equation}
This follows from the first-order optimality condition for~\eqref{eqn:example_minmax} and \citep[Eqn.~2.19]{book:rockafeller_saddlepoint}. In this context, the operator $B_i$ is known as the \textit{saddle operator} associated with $g_i$.
 
In this work, we focus on solving~\eqref{eqn:finite_sum_inclusion} using decentralised algorithms of forward-backward-type with fixed step sizes. One such algorithm is \emph{PG-EXTRA}~\citep{paper:algorithm_extra, paper:algorithm_pg_extra}, applicable when $B_1,\dots,B_N$ are cocoercive with constants $1/L_1,\dots,1/L_N$ respectively. Given a \emph{mixing matrix} $W=(w_{ij})\in\mathbb{R}^{N\times N}$ (Definition \ref{defn:mixing_matrix}) and a step size $\alpha < (1 + \lambda_\text{min}(W))/(\max_i\{L_i\})$, where $\lambda_\text{min}(W)$ denotes the smallest eigenvalue of $W$, the main iteration for PG-EXTRA takes the form
\begin{equation}\label{alg:pg_extra_agent}
\left\{\begin{aligned}
    v_i^k &= B_i(x_i^k),\\
    z_i^{k+1} &= z_i^k + \sum_{j=1}^N w_{ij} x_j^{k} - \sum_{j=1}^N \overline{w}_{ij} x_j^{k-1} - \alpha (v_i^k - v_i^{k-1}),\\
    x_i^{k+1} &= J_{\alpha A_i}(z_i^{k+1}),
\end{aligned} \right.\quad\forall i=1,\dots,N,
\end{equation}
where $(\overline{w}_{ij}) = \overline{W} = (I+W)/2$ and $J_{A_i}=(I+A_i)^{-1}$ denotes the \emph{resolvent} of $A_i$. Here the variables $v^k_i$, $x^k_i$, and $z^k_i$ are local variables belonging to agent $i$, and the coefficients $w_{ij}$ are nonzero only if agents $i$ and $j$ are neighbours in the communication graph. Thus, communication between agents occurs only if the agents are neighbours. Denote 
\[
\mathbf{x}^k = (x_1^k, \dots, x^k_N),\quad \mathbf{z}^k = (z_1^k, \dots, z^k_N),\]
\begin{equation}    \label{eqn:product_operators}
A(\mathbf{x})=A_1(x_1)\times\dots\times A_N(x_N),\quad B(\mathbf{x})=(B_1(x_1),\dots, B_N(x_N)),
\end{equation}
\[\mathbf{W} = W\otimes I, \text{~and~} \overline{\mathbf{W}} = \overline{W}\otimes I,\]
where $\otimes$ denotes the \emph{Kronecker product}. Using this notation, \eqref{alg:pg_extra_agent}~can then be represented compactly as
\begin{equation}\label{alg:pg_extra}
\left\{\begin{aligned}
    \mathbf{v}^k &= B(\mathbf{x}^k),    \\
    \mathbf{z}^{k+1} &= \mathbf{z}^k + \mathbf{W} \mathbf{x}^{k} - \overline{\mathbf{W}} \mathbf{x}^{k-1} - \alpha (\mathbf{v}^k - \mathbf{v}^{k-1}),    \\
    \mathbf{x}^{k+1} &= J_{\alpha A}(\mathbf{z}^{k+1}).
\end{aligned}
\right.
\end{equation}

In the context of~\eqref{eqn:finite_sum_inclusion}, PG-EXTRA suffers a number of drawbacks. 
\begin{enumerate}
    \item \emph{The single-valued operators must be cocoercive.} Cocoercivity is often too restrictive an assumption in this context. For example, in~\eqref{eqn:example_minmax_operators} the saddle operator $B_i$ is not cocoercive even for bilinear functions.

    \item \emph{Agents must use homogeneous step sizes}. The step size, $\alpha$, is determined by the smallest cocoercivity constant across agents' single-valued operators. This means agents cannot easily exploit the curvature of their local functions.

    \item \emph{The step size upper bound is graph-dependent}. The upper bound depends on the minimum eigenvalue of $W$, which depends on the communication graph and the method of constructing $W$. Consequently, agents require knowledge beyond their local information \citep{paper:algorithm_nids}.
\end{enumerate}

Various papers in the literature have proposed algorithms to resolve one or more of these issues, although none have resolved all of three in the context of~\eqref{eqn:finite_sum_inclusion}. To resolve the cocoercivity drawback of PG-EXTRA, \citet*{paper:malitsky_tam_minmax} proposed an algorithm that only requires the operators $B_i$ to be $L_i$-Lipschitz continuous for all $i=1,\dots,N$. Given a mixing matrix $W$ and a step size $\alpha<(1+\lambda_\text{min}(W))/(4\max_i\{L_i\})$, the main iteration of \citep[Alg.~1]{paper:malitsky_tam_minmax} takes the form
\begin{subequations} \label{alg:pdtr}
\begin{empheq}[left=\empheqlbrace]{align} 
        \mathbf{v}^k &= 2B(\mathbf{x}^k) - B(\mathbf{x}^{k-1}), \label{alg:pdtr_vk} \\
        \mathbf{z}^{k+1} &= \mathbf{z}^k + \mathbf{W}\mathbf{x}^k - \overline{\mathbf{W}}\mathbf{x}^{k-1} - \alpha (\mathbf{v}^k - \mathbf{v}^{k-1}), \label{alg:pdtr_zk}  \\
        \mathbf{x}^{k+1} &= J_{\alpha A}(\mathbf{z}^{k+1}). \label{alg:pdtr_xk}
\end{empheq}
\end{subequations}
Note that~\eqref{alg:pg_extra} and~\eqref{alg:pdtr} are very similar. Indeed, the only difference is the \textit{extragradient-type} term for the forward operator in~\eqref{alg:pdtr_vk}. This term arises because convergence of~\eqref{alg:pdtr} follows from the \textit{primal-dual-twice-reflected} algorithm, whereas convergence of~\eqref{alg:pg_extra} follows from the \textit{Condat--V\~u} algorithm~\citep[\S2]{paper:malitsky_tam_minmax}. There are algorithms besides~\eqref{alg:pdtr} that can resolve the first drawback of PG-EXTRA (see, for example, \citep{paper:beznosikov_saddlepoint, paper:rogozin_decentral_mirror}). However, the convergence analysis of these follow from the algorithm presented by \citet*{paper:nemirovski_proxmethod}, which is restricted to the setting of~\eqref{eqn:example_minmax}.

To resolve the step size drawbacks of PG-EXTRA when the operators $B_1,\dots,B_N$ are cocoercive with constants $1/L_1,\dots,1/L_N$ respectively, \citet*{paper:algorithm_nids} proposed the \emph{NIDS} algorithm. In this algorithm, agent step sizes are represented as a diagonal matrix $\Lambda=\mathrm{diag}(\alpha_1,\dots,\alpha_N)$, where $\alpha_i$ is the step size for agent $i$. Given a mixing matrix $W$, step sizes satisfying $\alpha_i<2/L_i$, and parameter $\beta\leq1/\max_i\{\alpha_i\}$, denote $\widetilde{\mathbf{W}}=(I-(\beta/2) \Lambda (I-W)) \otimes I$ and $\boldsymbol{\Lambda} = \Lambda\otimes I$. The main iteration for NIDS takes the form
\begin{equation}\label{alg:nids}
\left\{\begin{aligned}
    \mathbf{v}^k &= B(\mathbf{x}^k),    \\
    \mathbf{z}^{k+1} &= \mathbf{z}^k - \mathbf{x}^{k} + \widetilde{\mathbf{W}} \big(2\mathbf{x}^k - \mathbf{x}^{k-1} - \boldsymbol{\Lambda} (\mathbf{v}^k - \mathbf{v}^{k-1}) \big),    \\
    \mathbf{x}^{k+1} &= J_{\boldsymbol{\Lambda} A}(\mathbf{z}^{k+1}).
\end{aligned}
\right.
\end{equation}
To aid comparison between the previous three algorithms, observe that the $\mathbf{z}^{k+1}$ updates in~\eqref{alg:pg_extra},~\eqref{alg:pdtr}, and~\eqref{alg:nids} are very similar. Indeed, the $\mathbf{z}^{k+1}$ updates in \eqref{alg:pg_extra} and~\eqref{alg:pdtr} can be expressed as
\[\mathbf{z}^{k+1} = \mathbf{z}^k - \mathbf{x}^{k} + \overline{\mathbf{W}} \big(2\mathbf{x}^k - \mathbf{x}^{k-1}\big) - \alpha (\mathbf{v}^k - \mathbf{v}^{k-1}),\]
and further, $\widetilde{\mathbf{W}}=\overline{\mathbf{W}}$ when $\alpha_1=\dots=\alpha_N$ and $\beta=1/\max_i\{\alpha_i\}$. Thus, except for the step sizes, the only difference is the information communicated by agents to their neighbours: in~\eqref{alg:pg_extra} and~\eqref{alg:pdtr} this information is the most recent resolvent evaluation, whereas in~\eqref{alg:nids} it is a linear combination of the current and previous resolvent and forward evaluations. The convergence analysis of~\eqref{alg:nids} follows from the \emph{Davis--Yin} algorithm \citep[\S11.3]{paper:davis_yin, book:como_ryu_yin}. There are few algorithms other than~\eqref{alg:nids} that can resolve the second and third drawbacks of PG-EXTRA  (see, for example, \citep{paper:xin_frost, paper:xin_push_pull}), and they only apply in settings more limited than~\eqref{eqn:finite_sum_inclusion}, such as convex minimisation. In this work, we focus on algorithms that are applicable to~\eqref{eqn:finite_sum_inclusion} in the general setting of monotone inclusions, as the first drawback of PG-EXTRA is not a concern in, for example, minimisation problems.

It would be tempting to try resolving the three drawbacks of PG-EXTRA by directly combining the ideas of~\eqref{alg:pdtr} and~\eqref{alg:nids}. A first attempt of this would involve replace the forward evaluation in~\eqref{alg:nids} with an extragradient evaluation as in~\eqref{alg:pdtr}. However, in the context of~\eqref{eqn:finite_sum_inclusion}, the convergence analysis of such an algorithm does not directly follow from the Davis--Yin algorithm, which relies on cocoercivity of the single-valued operators. Instead, we propose the following algorithm (Algorithm~\ref{alg:main_algorithm}). Given a mixing matrix $W$, step sizes $\alpha_1,\dots,\alpha_N$ with $\alpha_i<1/(8L_i)$, and parameter $\beta<\|\Lambda^{1/2}((I-W)/2)\Lambda^{1/2}\|$, the main iteration for our algorithm takes the form
\begin{equation}\label{alg:main_algorithm_iteration}
\left\{\begin{aligned}
        \mathbf{v}^k &= 2B(\mathbf{y}^k) - B(\mathbf{y}^{k-1}),
        \\
        \mathbf{z}^{k+1} &= \mathbf{z}^k - \mathbf{x}^k + \widetilde{\mathbf{W}} \big( 2\mathbf{x}^k - \mathbf{x}^{k-1} - \boldsymbol{\Lambda} (\mathbf{v}^k - \mathbf{v}^{k-1}) \big),
        \\
        \mathbf{x}^{k+1} &= J_{\boldsymbol{\Lambda} A} (\mathbf{z}^{k+1}),
        \\
        \mathbf{y}^{k+1} &= \mathbf{x}^k + \mathbf{z}^{k+1} - \mathbf{z}^k.
 \end{aligned} \right.
 \end{equation}
This algorithm resolves all three drawbacks of PG-EXTRA in the context of~\eqref{eqn:finite_sum_inclusion} without introducing additional assumptions on the operators. We wish to emphasise an important difference between this iteration and~\eqref{alg:pdtr} and~\eqref{alg:nids}: the forward operator is not evaluated at $\mathbf{x}^k$ as with the other methods, but rather it is evaluated at another point $\mathbf{y}^k$.

The remainder of this paper is organised as follows. Section \ref{sec:preliminaries} covers background information on distributed optimisation as well as nonsmooth analysis. Section~\ref{sec:main_algorithm} establishes the convergence of~\eqref{alg:main_algorithm_iteration}, where the analysis is based on the \textit{backward-forward-reflected-backward} algorithm \citep{paper:rieger_tam_noncocoercive}. Section \ref{sec:numerics_part1} provides numerical results that validate the convergence analysis and illustrate the advantages of graph-independent heterogenous agent step sizes. Section \ref{sec:boosted_algorithm} presents an alternative implementation of~\eqref{alg:main_algorithm_iteration} for \emph{aggregative games} which can use fewer variables, and thus less memory, than a na\"ive implementation. Section~\ref{sec:vpp_model} presents an aggregative game for the coordination of a virtual power plant and provides numerical results on the performance of~\eqref{alg:main_algorithm_iteration} and its alternative implementation. Section~\ref{sec:conclusion} summarises the contribution of the paper and discusses future directions.

\section{Notation and preliminaries}   \label{sec:preliminaries}
Throughout this paper, let $\mathcal{H}$ be a real Hilbert space with scalar product $\langle \cdot, \cdot \rangle$ and induced norm $\norm{\cdot}$. The product space $\mathcal{H}\times\dots\times\mathcal{H}$ is denoted as $\mathcal{H}^N$. As a rule of thumb, we use normal letters to denote elements of $\mathcal{H}$ and bold letters for elements of the product space $\mathcal{H}^N$ (that is, $\mathbf{x}=(x_1,\dots,x_N)\in\mathcal{H}^N)$. The \textit{diagonal subspace} of $\mathcal{H}^N$ is the subspace $\mathbb{R}\mathbf{1} = \{\mathbf{z}\in\mathcal{H}^N \,\rvert\, z_1=\dots=z_N\}$. The identity operator on $\mathcal{H}$ is denoted by $I:\mathcal{H}\to\mathcal{H}$.

\subsection{Graphs and mixing matrices} \label{sec:subsec_distributed}
Let $\preceq$ represent the \emph{Loewner order}, $\otimes$ represent the \emph{Kronecker product}, $\lambda_\text{max}(M)$ denote the largest eigenvalue of a real symmetric matrix $M$, $\|\cdot\|_2$ denote the spectral norm, and $\mathbf{1}=(1, \dots, 1)^\top\in\mathbb{R}^N$.

\begin{prop} \label{prop:kronecker}
Given real symmetric matrices $M_1, M_2$, the Kronecker product has the following properties.
\begin{enumerate}[(i)]
    \item $(M_1 M_2)\otimes I = (M_1\otimes I) (M_2\otimes I)$,
    
    \item $M_1\succ0$ (respectively $\succeq0$) if and only if $M_1\otimes I\succ0$ (respectively $\succeq0$), and

    \item The singular values of $M_1$ are the singular values of $M_1\otimes I$, thus $\|M_1\|_2=\|M_1\otimes I\|_2$.
\end{enumerate}
\end{prop}
\begin{proof}
(i)~See, for example, \citep[Lem.~4.2.10]{book:horn_johnson_matrix}.
(ii)~See, for example, \citep[Cor.~4.2.13]{book:horn_johnson_matrix}.
(iii)~See, for example, \citep[Thm.~4.2.15]{book:horn_johnson_matrix}.
\end{proof}

In the context of distributed algorithms like those discussed in Section~\ref{sec:intro}, communication amongst agents is modelled by a connected graph $(\mathcal{V}, \mathcal{E})$, where $\mathcal{V} = \{1, \dots, N\}$ is the set of nodes and $\mathcal{E}$ is the set of edges. Each agent is represented by a node, and communication is represented by edges: two agents $i,j\in \mathcal{V}$ can communicate if $(i,j)\in \mathcal{E}$. For a given node $i\in\mathcal{V}$, $\mathcal{N}(i)$ denotes the set of \emph{neighbours} of node $i$ and $\mathcal{N}^2(i)$ denotes the set of nodes whose distance from node $i$ is two or less, defined by
\[\mathcal{N}(i) = \{i\} \cup \{j\in\mathcal{V} \,\rvert\, (i,j)\in\mathcal{E}\},\quad \mathcal{N}^2(i) = \bigcup_{j\in\mathcal{N}(i)} \mathcal{N}(j).\]
Observe that $i\in\mathcal{N}(j)$ if and only if $j\in\mathcal{N}(i)$. The definitions of the set of neighbours and distance-two neighbours both have $i\in\mathcal{N}(i)$ and $i\in\mathcal{N}^2(i)$ for notational convenience. 

A round of communication can be modelled by agents forming a linear combination of variables communicated by neighbouring agents, for example $\sum_{j=1}^N w_{ij} x^k_j$ in~\eqref{alg:pg_extra_agent}, when the matrix $W=(w_{ij})$ is a \textit{mixing matrix}.
\begin{definition}[Mixing matrix {\citep{paper:algorithm_extra}}]
    \label{defn:mixing_matrix}
    Given a connected graph $(\mathcal{V}, \mathcal{E})$ with $\mathcal{V}=\{1,\dots,N\}$, a matrix $W = (w_{ij})\in\mathbb{R}^{N\times N}$ is a \textit{mixing matrix} if it satisfies the following properties:
    \begin{enumerate}[(i)]
        \item $w_{ij} = 0$ if $i\neq j$ and $(i,j)\notin\mathcal{E}$,
        \item $W = W^\top$,
        \item $\mathrm{ker}(I-W) = \mathbb{R}\mathbf{1}$,
        \item $-I \prec W \preceq I$.
    \end{enumerate}
\end{definition}
\noindent
One method of constructing a mixing matrix is the \textit{Laplacian-based constant edge weight matrix}
\begin{equation}    \label{defn:laplacian_mixing}
    W = I - \frac{1}{\tau} \mathcal{L},
\end{equation}
where $\mathcal{L}$ is the graph Laplacian of $(\mathcal{V}, \mathcal{E})$ and $\tau > (1/2)\lambda_\text{max}(\mathcal{L})$. For a discussion of other choices for the mixing matrix, see \citep{paper:lin_boyd_flda, paper:algorithm_extra}.

In the remainder of this paper, we will abuse notation in the following way: for a real symmetric matrix $M$, we will denote the product $M\otimes I$ by $M$. This notation will not be problematic for results involving the norm and positive (semi-)definiteness of $M$, as Proposition~\ref{prop:kronecker} establishes that the operation $M\mapsto M\otimes I$ preserves singular values and definiteness.

\subsection{Operators}
Set-valued operators are denoted by $A:\mathcal{H}\rightrightarrows\mathcal{H}$ and single-valued operators by $B:\mathcal{H}\to\mathcal{H}$. For an operator $A:\mathcal{H}\rightrightarrows\mathcal{H}$, its \textit{domain} is $\dom A = \{x\in\mathcal{H} \,\rvert\, A(x)\neq\varnothing\}$ and its \textit{range} is $\ran A = A(\mathcal{H})$. The \textit{graph} of $A$ is
\[\mathrm{gra}A = \{(x, x')\in \mathcal{H}\times\mathcal{H} \,\rvert\, x'\in A(x)\},\]
and the \textit{inverse} of $A$ is defined through its graph,
\[\mathrm{gra}(A^{-1}) = \{(x', x)\in \mathcal{H}\times\mathcal{H} \,\rvert\, (x, x')\in \mathrm{gra}A\}.\]
An operator $A:\mathcal{H}\rightrightarrows\mathcal{H}$ is \textit{monotone} if
\[\langle x-y, x' - y'\rangle \geq 0 \quad \forall (x, x'), (y, y')\in\mathrm{gra}A,\]
and a monotone operator $A:\mathcal{H}\rightrightarrows\mathcal{H}$ is \textit{maximally monotone} if there does not exist a monotone operator $A':\mathcal{H}\rightrightarrows\mathcal{H}$ such that $\mathrm{gra}A \subset \mathrm{gra}A'$. Given a maximally monotone operator $A:\mathcal{H}\rightrightarrows\mathcal{H}$, its \textit{resolvent}, $J_A = (I+A)^{-1}$, is a single-valued operator with $\dom J_A = \mathcal{H}$ \citep{paper:minty_monotone}. An operator $B:\mathcal{H}\to\mathcal{H}$ is $L$-\textit{Lipschitz continuous} if $L\geq0$ and
\[\|B(x) - B(y)\| \leq L\norm{x-y} \quad \forall x,y\in\mathcal{H},\]
and $\beta$-\textit{cocoercive} if $\beta>0$ and
\[\langle x-y, B(x) - B(y) \rangle \geq \beta \|B(x)-B(y)\|^2 \quad \forall x,y\in\mathcal{H}.\]
Cocoercive operators are Lipschitz continuous, following from the Cauchy--Schwarz inequality. However, the converse is not always true. For example, skew-symmetric operators are Lipschitz continuous and never cocoercive \citep{paper:malitsky_tam_forb}. Given a differentiable convex function $f:\mathcal{H}\to\mathbb{R}$, the operator $B=\grad f:\mathcal{H}\to\mathcal{H}$ is $1/L$-cocoercive if and only if it is $L$-Lipschitz continuous \citep{paper:baillon_haddad}. If, however, an operator $B:\mathcal{H}\to\mathcal{H}$ is not the gradient of a convex function, then cocoercivity can only be guaranteed with additional restrictive assumptions such as strong monotonicity \citep[Remark 5]{paper:belgioioso_ev_charge}.

\subsection{Convex functions}
For a function $f:\mathcal{H}\to\mathbb{R}\cup\{\pm\infty\}$, its \textit{domain} is $\dom f = \{x\in\mathcal{H} \,\rvert\, f(x) < +\infty\}$ and it is \textit{proper} if $\dom f\neq\varnothing$ and $f(x)>-\infty$ for all $x\in\dom f$. Given a proper lsc convex function $f:\mathcal{H}\to\mathbb{R}\cup\{+\infty\}$, its (convex) \textit{subdifferential} at a point $x\in\mathcal{H}$ is the set-valued operator
\[
\partial f(x) = 
\begin{cases}
\{u\in\mathcal{H} \,\rvert\, \langle y-x, u \rangle + f(x) - f(y)\leq0\,, \forall y\in \mathcal{H}\}, & x\in\dom f,  \\
\varnothing, & x\notin\dom f.
\end{cases}\]
The operator $\partial f:\mathcal{H}\rightrightarrows\mathcal{H}$ is maximally monotone \citep{paper:rockafellar_subdifferential}, and if a function $f:\mathcal{H}\to\mathbb{R}$ is differentiable then $\partial f(x)=\{\grad f(x)\}$. Given a nonempty closed convex set $C\subseteq\mathcal{H}$, the \textit{indicator function} of $C$ is the proper lsc convex function $\iota_C:\mathcal{H}\to\mathbb{R}\cup\{+\infty\}$ defined by 
\[\iota_C(x) = \begin{cases} 0 & x\in C, \\ +\infty & x\notin C. \end{cases}\]
The \textit{normal cone} of $C$, $N_C:\mathcal{H}\rightrightarrows\mathcal{H}$, is the subdifferential of $\iota_C$.

\subsection{Properties of monotone operators and inclusions}
We now state some general results on monotone operators and monotone inclusions that will be useful in the convergence analysis of Algorithm~\ref{alg:main_algorithm}, detailed in Section \ref{sec:main_algorithm}. The first result establishes a sufficient condition for positive definiteness of a certain operator.

\begin{lemma}   \label{lem:M_positive_def}
    Let $S,T:\mathcal{H}\to\mathcal{H}$ be self-adjoint operators. If $c > \|T S^2 T\|$, then $c I - S T^2 S \succ 0$.
\end{lemma}
\begin{proof}
    Suppose $c > \|T S^2 T\| = \|T S\|^2,$ where equality follows from $S$ and $T$ being self-adjoint. Using the definition of the operator norm, we then have
    \[c\|x\|^2 > \|TS\|^2\|x\|^2 \geq \|T Sx\|^2 \iff \big\langle \big(c I - S T^2 S\big) x, x\big\rangle >0, \quad \forall x\in\mathcal{H}\setminus\{0\},\]
    which establishes positive definiteness.
\end{proof}

The second result establishes a sufficient condition for the operators in~\eqref{eqn:product_operators} defined on the product space to be maximally monotone, as well as a specific case where operator composition preserves maximal monotonicity and Lipschitz continuity.
\begin{lemma} \label{lem:diagonal_composition}
    Let $T_1, \dots, T_N:\mathcal{H}\rightrightarrows\mathcal{H}$ be operators and $D=\mathrm{diag}(d_1, \dots, d_N) \otimes I$  with $d_1, \dots, d_N > 0$. Define the operator $T:\mathcal{H}^N\rightrightarrows\mathcal{H}^N$ by
    \[T(\mathbf{x}) = T_1 (x_1) \times \dots \times T_N(x_N).\]
    Then the following statements hold.
    \begin{enumerate}[(i)]
        \item \label{lem:diagonal_composition1} If $T_1, \dots, T_N:\mathcal{H}\rightrightarrows\mathcal{H}$ are maximally monotone, then $DTD$ is maximally monotone.

        \item \label{lem:diagonal_composition2} If $T_i:\mathcal{H}\to\mathcal{H}$ is $L_i$-Lipschitz continuous for all $i=1,\dots,N$, then $DTD$ is $\max_i\{d_i^2L_i\}$-Lipschitz continuous.
    \end{enumerate}
     
\end{lemma}

\begin{proof}
    \ref{lem:diagonal_composition1} Suppose $T_1, \dots, T_N$ are maximally monotone. Then $T$ is maximally monotone \citep[Prop.~20.23]{book:camo_bauschke_combettes}. Since $D$ is self-adjoint and $\ran D = \mathcal{H}^N$, it follows that $D T D$ is maximally monotone \citep[Cor.~ 25.6]{book:camo_bauschke_combettes}. 
    
    \ref{lem:diagonal_composition2} Suppose that $T_i$ is $L_i$-Lipschitz continuous for all $i=1,\dots,N$. Then, for all $\mathbf{x} = (x_1, \dots, x_N)$ and $\mathbf{x}' = (x_1', \dots, x_N')\in\mathcal{H}^N$, we have
    \begin{align*}
    \|(D T D)(\mathbf{x}) -(D T D)(\mathbf{x}')\|^2 = \sum_{i=1}^N \|(d_i T_i)(d_i x_i) - (d_i T_i)(d_i x_i')\|^2 & \leq \sum_{i=1}^N d_i^2 L_i^2 \|d_i x_i - d_ix_i'\|^2 \\
    & \leq \max_i \{d_i^4 L_i^2\} \|\mathbf{x} - \mathbf{x}'\|^2.
    \end{align*}
    Thus, $D T D$ is $\max_i\{d_i^2L_i\}$-Lipschitz continuous,
    establishing the result.
\end{proof}

The third result is an extension of \citep[Prop.~4.1,]{paper:malitsky_tam_minmax}, where an inclusion in $\mathcal{H}$ of the form~\eqref{eqn:finite_sum_inclusion} is reformulated as an inclusion in the product space $\mathcal{H}^N$.

\begin{prop}  \label{prop:product_space}
    Let $T_1,\dots, T_N:\mathcal{H}\rightrightarrows\mathcal{H}$ be operators, $S:\mathcal{H}^N\to\mathcal{H}^N$ be an invertible self-adjoint operator, and $K:\mathcal{H}^N\to\mathcal{H}^N$ be a bounded linear operator with $\mathrm{ker}K=\mathbb{R}\mathbf{1}$. Define the operator $T:\mathcal{H}^N\rightrightarrows\mathcal{H}^N$ by
    \[T(\mathbf{x}) = T_1 (x_1) \times \dots \times T_N(x_N).\]
    Then $\mathbf{y}\in\mathcal{H}^N$ solves
    \begin{equation}    \label{eqn:proof_product_space1}
        0 \in ST(S \mathbf{y}) + (K S)^* N_{\{0\}}(K S\mathbf{y}) \subset \mathcal{H}^N,
    \end{equation}
    if and only if $\mathbf{y} = S^{-1}(z, \dots, z) \in\mathcal{H}^N$ for some $z\in\mathcal{H}$ and $z$ solves
    \begin{equation}    \label{eqn:proof_product_space2}
        0 \in\sum_{i=1}^N T_i(z) \subset \mathcal{H}.
    \end{equation}
\end{prop}

\begin{proof}
    Since $S$ is invertible and self-adjoint, $\mathbf{y}\in\mathcal{H}^N$ solves~\eqref{eqn:proof_product_space1} if and only if $\mathbf{z} = S\mathbf{y}\in\mathcal{H}^N$ solves
    \begin{align}   \label{eqn:proof_product_space_prop4.1}
    0 \in T(\mathbf{z}) + K^* N_{\{0\}}(K \mathbf{z}).
    \end{align}
    Applying \citep[Prop.~4.1]{paper:malitsky_tam_minmax} to~\eqref{eqn:proof_product_space_prop4.1}, we have that $\mathbf{z}=(z_1,\dots,z_N)$ solves~\eqref{eqn:proof_product_space_prop4.1} if and only if $z_1=\dots=z_N\in\mathcal{H}$ and $z=z_1$ solves~\eqref{eqn:proof_product_space2}, completing the proof.
\end{proof}
The fourth result extends \citep[\S3.5.1]{book:como_ryu_yin} to real Hilbert spaces, establishing that auxiliary variables can be included in a monotone inclusions without changing the solution set.
\begin{prop}    \label{prop:augmented_inclusion}
    Let $\mathcal{H}_1, \mathcal{H}_2$ be real Hilbert spaces, $T:\mathcal{H}_1 \rightrightarrows \mathcal{H}_2$, $S:\mathcal{H}_2\rightrightarrows\mathcal{H}_2$ be operators, and $P:\mathcal{H}_1 \to \mathcal{H}_2$, $Q: \mathcal{H}_2 \to \mathcal{H}_2$ be bounded linear operators. Then $z \in \mathcal{H}_1$ solves 
    \begin{equation}    \label{eqn:proof_base_inclusion}
        0 \in Tz + P^* S(Pz)
    \end{equation}
    if and only if there exists $\tilde z \in \mathcal{H}_2$ such that $(z, \tilde z) \in\mathcal{H}_1\times\mathcal{H}_2$ solves
    \begin{subequations} \label{eqn:proof_augment_inclusion}
    \begin{empheq}[left=\empheqlbrace]{align}
        0 & \in Tz + P^* S \left(Pz + Q\tilde z \right),\label{eqn:proof_augment_inclusion_a}
        \\
        0 & \in N_{\{0\}} (\tilde z) + Q^* S \left(Pz + Q\tilde z\right). \label{eqn:proof_augment_inclusion_b}
    \end{empheq}
    \end{subequations}
\end{prop}

\begin{proof}
    Suppose that $z$ solves~\eqref{eqn:proof_base_inclusion} and set $\tilde z = 0$. Since $Q\tilde z = 0$ by linearity,~\eqref{eqn:proof_base_inclusion} implies that $(z, \tilde z)$ solves

    \[0 \in Tz + P^*S(Pz) = Tz + P^*S(Pz + Q\tilde z),\]
    that is, inclusion~\eqref{eqn:proof_augment_inclusion_a} holds. Additionally,~\eqref{eqn:proof_base_inclusion} implies $P^*S(Pz) \neq \varnothing$ ,which implies $S(Pz) \neq \varnothing$, which implies $Q^* S \left( Pz \right) = Q^* S \left( Pz + Q \tilde z\right) \neq \varnothing$. Since $N_{\{0\}}(\tilde{z})=N_{\{0\}} (0) = \mathcal{H}_2$, we also have 
    \[0 \in N_{\{0\}} (\tilde z) = N_{\{0\}} (\tilde z) + Q^* S \left( Pz + Q \tilde z \right),\]
    which establishes~\eqref{eqn:proof_augment_inclusion_b}.
    
    Conversely, suppose that $(z, \tilde z)$ solves~\eqref{eqn:proof_augment_inclusion}. Then $N_{\{0\}} (\tilde z)\neq\varnothing$ due to~\eqref{eqn:proof_augment_inclusion_b}, and hence $\tilde z = 0$. Then $Q\tilde z = 0$ by linearity, and so~\eqref{eqn:proof_augment_inclusion_a} implies that
   \[0 \in Tz + P^* S \left(Pz + 0 \right) = Tz + P^* S (Pz).\]
   Then $z$ solves~\eqref{eqn:proof_base_inclusion}, completing the proof.
\end{proof}
The final result decomposes the resolvent of a composition of operators into an equation and an inclusion that does not involve a composition. This generalises the result in \citep[Eqn.~2.6]{book:como_ryu_yin}, which considers the finite dimensional setting where one of the operators is a subdifferential.
\begin{prop}   \label{prop:resolvent_decompose}
    Let $\mathcal{H}_1$ and $\mathcal{H}_2$ be real Hilbert spaces, $S:\mathcal{H}_2\rightrightarrows\mathcal{H}_2$ be a maximally monotone operator, and $P:\mathcal{H}_1 \times \mathcal{H}_2 \to \mathcal{H}_2$ be a bounded linear operator. Define the operator $T: \mathcal{H}_1 \times \mathcal{H}_2 \rightrightarrows \mathcal{H}_1 \times \mathcal{H}_2$ by $T(t) = P^* S (P t)$, and suppose that $T$ is maximally monotone. Then, for all $r \in \mathcal{H}_1 \times \mathcal{H}_2$, $t \in \mathcal{H}_1 \times \mathcal{H}_2$ and $\alpha>0$ we have

    \begin{equation}    \label{eqn:proof_decompose_resolvent}
    r = J_{\alpha T}(t)
    \end{equation}
    if and only if there exists $u\in\mathcal{H}_2$ such that
   \begin{subequations}    
    \begin{empheq}[left=\empheqlbrace]{align}
            r & = t - \alpha P^*u,\label{eqn:proof_decompose_equality1}
            \\
            S^{-1}(u) & \ni Pt - \alpha PP^*u. \label{eqn:proof_decompose_equality2}
    \end{empheq}
    \end{subequations}
    
\end{prop}

\begin{proof}
    Suppose that $r = J_{\alpha T}(t)$. By the definitions of $T$ and the resolvent, we have
    \[
    t \in (I + \alpha T)(r) = r + \alpha P^* S (P r).
    \]
    Hence, $S(Pr) \neq \varnothing$ and there exists $u \in S(Pr) \subseteq \mathcal{H}_2$ such that
    $t = r + \alpha P^*u,$ which establishes~\eqref{eqn:proof_decompose_equality1}.
    Since $u \in S(Pr)$ and $P$ is linear, we have
    \[u \in S(Pr) \iff S^{-1}(u) \ni Pr = P\left( t - \alpha P^*u \right) = Pt -  \alpha PP^*u,\]
    which establishes~\eqref{eqn:proof_decompose_equality2}.
    
    Conversely, let $r\in\mathcal{H}_1\times\mathcal{H}_2, t\in\mathcal{H}_1\times\mathcal{H}_2$, and suppose there exists $u\in\mathcal{H}_2$ such that~\eqref{eqn:proof_decompose_equality1} and~\eqref{eqn:proof_decompose_equality2} hold. Combining~\eqref{eqn:proof_decompose_equality2},~\eqref{eqn:proof_decompose_equality2}, and linearity of $P$, we have
    \[S^{-1}(u) \ni Pt - \alpha PP^*u = P \left( t - \alpha P^*u \right) = Pr \iff u \in S(Pr).\]
    Using~\eqref{eqn:proof_decompose_equality1} and the definition of $T$, we deduce that
    \[t = r + \alpha P^*u \in r + \alpha P^* S(Pr) = (I + \alpha T)(r).\]
    Then by definition of the resolvent, (\ref{eqn:proof_decompose_resolvent}) holds, which establishes the result.
\end{proof}
\begin{remark} \label{rem:resolvent_decompose}
If $S = N_{\{0\}}$, then $S^{-1}=0$  and~\eqref{eqn:proof_decompose_equality2} reduces to $\alpha PP^*u = Pt$.
\end{remark}

\section{A distributed forward-backward-type algorithm} \label{sec:main_algorithm}
In this section, we present an algorithm for solving~\eqref{eqn:finite_sum_inclusion} which allows heterogenous agent step sizes chosen independently of the communication graph. The derivation of this algorithm builds on the derivation of~\eqref{alg:nids} but is based on the backward-forward-reflected-backward algorithm \citep{paper:rieger_tam_noncocoercive} rather than the Davis--Yin algorithm. Its updates combine the ideas of~\eqref{alg:pdtr} and~\eqref{alg:nids} mentioned in Section \ref{sec:intro}, namely using an extragradient evaluation for the forward operator and agents communicating resolvent and forward evaluations. The proposed algorithm for finding zeros of inclusion~\eqref{eqn:finite_sum_inclusion} is given in Algorithm~\ref{alg:main_algorithm}.

\begin{algorithm}[!ht]
\caption{An algorithm with graph-independent heterogeneous step sizes for inclusion~\eqref{eqn:finite_sum_inclusion}.} 
\label{alg:main_algorithm}
Choose $\mathbf{y}^0, \mathbf{z}^0\in\mathcal{H}^N$ and set $\mathbf{v}^{0} = B(\mathbf{y}^0)$;

Choose step sizes $\alpha_1,\dots,\alpha_N$ and set $\Lambda=\mathrm{diag}(\alpha_1, \dots, \alpha_N)$;

Choose a  mixing matrix $W$ and parameter $\beta < \|\Lambda^{1/2} ((I-W)/2) \Lambda^{1/2}\|^{-1}$;

Compute $\widetilde{W}=I - (\beta/2) \Lambda (I-W)$;

Initialise variables:
\begin{subequations} \label{eqn:alg1_init}
\begin{empheq}[left=\empheqlbrace]{align}
\mathbf{x}^0 & = J_{\Lambda A} (\mathbf{z}^0) \label{eqn:alg1_x0}
\\
\mathbf{y}^1 &= 2\mathbf{x}^0 - \mathbf{z}^0 - \Lambda \mathbf{v}^0 \label{eqn:alg1_y1}
\\
\mathbf{z}^1 & = \mathbf{y}^1 + \mathbf{z}^0 - \mathbf{x}^0 \label{eqn:alg1_z1}
\\
\mathbf{x}^1 & = J_{\Lambda A}(\mathbf{z}^1) \label{eqn:alg1_x1}
\end{empheq}
\end{subequations}

\For{$k\geq1$}{
Update variables:
\begin{subequations}
    \begin{empheq}[left=\empheqlbrace]{align}
    \mathbf{v}^k &= 2B(\mathbf{y}^k) - B(\mathbf{y}^{k-1})  \label{eqn:alg1_vk}
    \\
    \mathbf{z}^{k+1} &= \mathbf{z}^k - \mathbf{x}^k + \widetilde{W} \big( 2\mathbf{x}^k - \mathbf{x}^{k-1} - \Lambda (\mathbf{v}^k - \mathbf{v}^{k-1}) \big) \label{eqn:alg1_zk}
    \\
    \mathbf{x}^{k+1} &= J_{\Lambda A} (\mathbf{z}^{k+1})    \label{eqn:alg1_xk}
    \\
    \mathbf{y}^{k+1} &= \mathbf{x}^k + \mathbf{z}^{k+1} - \mathbf{z}^k  \label{eqn:alg1_yk}
    \end{empheq}
\end{subequations}
}
\end{algorithm}
Before turning to the convergence analysis of our proposed algorithm, a few comments on the implentation and computation of parameters in Algorithm~\ref{alg:main_algorithm} are in order.
\begin{remark} \label{rem:mixing_matrix_local}
    In Algorithm~\ref{alg:main_algorithm}, communication amongst agents in \eqref{eqn:alg1_zk} is modelled by matrix multiplication with $\widetilde{W}=(\widetilde{w}_{ij})\in\mathbb{R}^{N \times N}$. However, to implement~\eqref{eqn:alg1_zk}, agent $i$ requires only their step size $\alpha_i$, row $i$ of the mixing matrix $W=(w_{ij})$, and communication with their neighbours $j\in\mathcal{N}(i)$. To see this, first note that from the definition of $\widetilde{W}$, the coefficients $\widetilde{w}_{ij}$ are
    \[\widetilde{w}_{ij}=\begin{cases}
        1-\frac{\beta}{2} \alpha_i (1-w_{ii})  & j=i, \\
        \frac{\beta}{2} \alpha_i w_{ij}  & j\in\mathcal{N}(i)\setminus\{i\}, \\
        0   & j\notin\mathcal{N}(i).
    \end{cases}\]
    Then observe that~\eqref{eqn:alg1_zk} can be expressed as
    \[z^{k+1}_i = z^k_i - x^k_i + \sum_{j=1}^N \widetilde{w}_{ij} \big(2x_j^k - x_j^{k-1} - \alpha_j (v_j^k - v_j^{k-1})\big), \quad \forall i=1,\dots,N,\]
    as outlined in \citep[Alg.~1]{paper:algorithm_nids}. Moreover, agent $i$ only needs to sum over $j\in\mathcal{N}(i)$ since $\widetilde{w}_{ij}=0$ otherwise.
    
\end{remark}

\begin{remark}  \label{rem:beta_condition}
    It may be difficult to directly choose the parameter $\beta$ in Algorithm~\ref{alg:main_algorithm} using the condition $\beta~<~\|\Lambda^{1/2}((I-W)/2)\Lambda^{1/2}\|^{-1}$. In practice, it suffices to choose $\beta \leq (\max_i\{\alpha_i\})^{-1}$. This choice is straightforward to compute by having agents repeatedly take the maximum over their neighbours' step sizes, as observed in \citep[\S3]{paper:algorithm_nids}. To see that this choice of $\beta$ is valid, observe that
    \[\|\Lambda^{1/2} ((I-W)/2) \Lambda^{1/2}\| \leq \|\Lambda^{1/2}\| \|(I-W)/2\| \|\Lambda^{1/2}\| < \|\Lambda\| = \max_i\{\alpha_i\},
    \]
    where the final inequality follows from $(I-W)/2 \prec I$ due to Definition~\ref{defn:mixing_matrix}(iv). Hence $\beta\leq (\max_i\{\alpha_i\})^{-1}$ implies $\beta<\|\Lambda^{1/2}((I-W)/2)\Lambda^{1/2}\|^{-1}$, satisfying the condition in Algorithm \ref{alg:main_algorithm}.
\end{remark}
To simplify the presentation of the convergence proof of Algorithm~\ref{alg:main_algorithm}, we first formulate inclusion~\eqref{eqn:finite_sum_inclusion} as an inclusion in the product space $\mathcal{H}^N \times \mathcal{H}^N$.
\begin{lemma}   \label{lem:transformed_inclusion}
    Let $W$ be a mixing matrix, $\gamma_1, \dots, \gamma_N>0$, and denote $\Gamma=\mathrm{diag}(\gamma_1,
    \dots,\gamma_N)$. Define the operators $A:\mathcal{H}^N\rightrightarrows\mathcal{H}^N$ and $B:\mathcal{H}^N\to\mathcal{H}^N$ as in~\eqref{eqn:product_operators}, and denote $K=((I-W)/2)^{1/2}\succeq0$ and $M~=~(cI~-~K\Gamma^2K)^{1/2}\succ0$ where $c > \|\Gamma K^2 \Gamma\|$. Define the operators $\widetilde{A}, \widetilde{C}:\mathcal{H}^N\times\mathcal{H}^N\rightrightarrows\mathcal{H}^N\times\mathcal{H}^N$ and $\widetilde{B}:\mathcal{H}^N\times\mathcal{H}^N\to\mathcal{H}^N\times\mathcal{H}^N$ by
    \begin{equation}    \label{eqn:transformed_operators}
    \begin{aligned}
    \widetilde{A}(\mathbf{z}, \tilde{\mathbf{z}}) &= \big(\Gamma A(\Gamma \mathbf{z}), N_{\{0\}}(\tilde{\mathbf{z}})\big), \quad \widetilde{B}(\mathbf{z}, \tilde{\mathbf{z}}) = \big(\Gamma B(\Gamma \mathbf{z}), 0 \big), \\
    \widetilde{C}(\mathbf{z}, \tilde{\mathbf{z}}) &= \big(\Gamma K N_{\{0\}}(K\Gamma \mathbf{z} + M \tilde{\mathbf{z}}), M N_{\{0\}}(K\Gamma \mathbf{z} + M \tilde{\mathbf{z}}) \big).
    \end{aligned}
    \end{equation}
    Then the following statements hold.
    \begin{enumerate}[(i)]
        \item \label{lem:transformed_inclusion_i} The point $(\mathbf{a}, \tilde{\mathbf{a}}) \in \mathcal{H}^N\times\mathcal{H}^N$ satisfies
        \begin{equation}    \label{eqn:proof_transformed_inclusion}
        0 \in (\widetilde{A} + \widetilde{B} + \widetilde{C} ) (\mathbf{a}, \tilde{\mathbf{a}}) \subset \mathcal{H}^N\times\mathcal{H}^N.
        \end{equation}
        if and only if $\mathbf{a} = \Gamma^{-1}(z, \dots,z)$ for some $z\in\mathcal{H}$ and $z$ satisfies\eqref{eqn:finite_sum_inclusion}.
        
        \item Suppose $A_1,\dots,A_N$ are maximally monotone, and $B_i$ are monotone and $L_i$-Lipschitz continuous for all $i=1,\dots,N$. Then $\widetilde{A}$ and $\widetilde{C}$ are maximally monotone, and $\widetilde{B}$ is monotone and $L$-Lipschitz continuous with $L~=~\max_i\{\gamma_i^2 L_i\}$. \label{lem:transformed_inclusion_ii}
    \end{enumerate}
\end{lemma}

\begin{proof}
\ref{lem:transformed_inclusion_i}~We first note that $M$ is well-defined by Lemma~\ref{lem:M_positive_def}. Applying Proposition~\ref{prop:augmented_inclusion} to inclusion~\eqref{eqn:proof_transformed_inclusion} with $T = \Gamma A \Gamma + \Gamma B \Gamma$, $S = N_{\{0\}}$, $P = K\Gamma$, and $Q = M$, we deduce that $(\mathbf{a}, \tilde{\mathbf{a}}) \in \mathcal{H}^N\times\mathcal{H}^N$ satisfies~\eqref{eqn:proof_transformed_inclusion} if and only if $\mathbf{a}\in\mathcal{H}^N$ satisfies
    \begin{equation}    \label{eqn:proof_product_gamma_inclusion}
        0 \in \Gamma A(\Gamma \mathbf{a}) + \Gamma B(\Gamma \mathbf{a}) + \Gamma K N_{\{0\}}(K \Gamma \mathbf{a}).
    \end{equation}
    Applying Proposition~\ref{prop:product_space} to~\eqref{eqn:proof_product_gamma_inclusion} with $T_i = A_i + B_i$ and $S=\Gamma$, we further deduce that $\mathbf{a}\in\mathcal{H}^N$ solves~\eqref{eqn:proof_product_gamma_inclusion} if and only if $\mathbf{a}=\Gamma^{-1}(z,\dots,z)$ for some $z\in\mathcal{H}$ and $z$ satisfies~\eqref{eqn:finite_sum_inclusion}.

\ref{lem:transformed_inclusion_ii}~From Lemma~\ref{lem:diagonal_composition}, we deduce that $\Gamma A \Gamma$ is maximally monotone and $\Gamma B \Gamma$ is monotone and $L$-Lipschitz continuous with $L = \max_i \{\gamma_i^2 L_i\}$. As $N_{\{0\}}$ and $0$ are also maximally monotone, we have that $\widetilde{A}$ is maximally monotone and $\widetilde{B}$ is monotone \citep[Prop.~20.23]{book:camo_bauschke_combettes}. Furthermore, $\widetilde{B}$ is $L$-Lipschitz continuous. Finally, observe that $\widetilde{C}$ can be written as
    \[\widetilde{C} = \begin{bmatrix} (K \Gamma)& M\end{bmatrix}^* \circ N_{\{0\}} \circ \begin{bmatrix} (K \Gamma)& M\end{bmatrix}.\]
    Since $\begin{bmatrix} (K \Gamma)& M\end{bmatrix} \begin{bmatrix} (K \Gamma)& M\end{bmatrix}^* = (1/\beta)I$ which is invertible, $\widetilde{C}$ is maximally monotone \citep[Prop.~23.25]{book:camo_bauschke_combettes}, establishing the result.
\end{proof}

We now state the main result on the convergence of Algorithm~\ref{alg:main_algorithm} for an inclusion of the form~\eqref{eqn:finite_sum_inclusion}.
\begin{theorem} \label{thm:main_algorithm}
    Suppose $A_i$ is maximally monotone, $B_i$ is monotone and $L_i$-Lipschitz continuous, and $(\sum_{i=1}^N A_i + B_i )^{-1}(0) \neq \varnothing$. Let $\alpha_1, \dots, \alpha_N$ satisfy $0 < \alpha_i < 1/(8L_i)$. Then the sequences $\{\mathbf{x}^k\}_{k=0}^\infty$, $\{\mathbf{y}^k\}_{k=0}^\infty$, $\{\mathbf{z}^k\}_{k=0}^\infty$ defined by Algorithm~\ref{alg:main_algorithm} have the following properties.

    \begin{enumerate}[(i)]
        \item $\mathbf{z}^k \rightharpoonup \mathbf{z}^* \in\mathcal{H}^N$, \label{eqn:thm3.3i}

        \item $\mathbf{x}^k\rightharpoonup \mathbf{x}^*$, $\mathbf{y}^k \rightharpoonup \mathbf{x}^* = J_{\Lambda A}(\mathbf{z}^*)$, and \label{eqn:thm3.3ii}

        \item $\mathbf{x}^* = (x^*, \dots, x^*) \in\mathcal{H}^N$ where $x^* \in ( \sum_{i=1}^N A_i + B_i )^{-1}(0)\subset\mathcal{H}$. \label{eqn:thm3.3iii}
    \end{enumerate}
\end{theorem}

\begin{proof}
    Denote $\Gamma=\Lambda^{1/2}\succ0$, $K=((I-W)/2)^{1/2}\succeq0$, and $M=((1/\beta)I-K\Gamma^2 K)^{1/2}\succ0$, noting that $M$ is well-defined from Lemma~\ref{lem:M_positive_def}. Define the operators $\widetilde{A}$, $\widetilde{B}$, and $\widetilde{C}$ as in~\eqref{eqn:transformed_operators}. From Lemma~\ref{lem:transformed_inclusion}\ref{lem:transformed_inclusion_i}, we deduce that $(\mathbf{a}^*, \tilde{\mathbf{a}}^*) \in \mathcal{H}^N\times\mathcal{H}^N$ satisfies
    \begin{equation}    \label{eqn:proof_product_space_inclusion}
    0 \in (\widetilde{A} + \widetilde{B} + \widetilde{C} ) (\mathbf{a}^*, \tilde{\mathbf{a}}^*) \subset \mathcal{H}^N\times\mathcal{H}^N,    
    \end{equation}
    if and only if $\mathbf{a}^*=\Gamma^{-1}(x^*, \dots, x^*)$ and $x^*\in\mathcal{H}$ satisfies~\eqref{eqn:finite_sum_inclusion}. From Lemma~\ref{lem:transformed_inclusion}\ref{lem:transformed_inclusion_ii}, we deduce that $\widetilde{A}$ and $\widetilde{C}$ are maximally monotone, and $\widetilde{B}$ is monotone and $L$-Lipschitz continuous with $L=\max_i\{\alpha_iL_i\}$.
    
    We now apply the \emph{backward-forward-reflected-backward algorithm}~\citep{paper:rieger_tam_noncocoercive} with unit step size to inclusion~\eqref{eqn:proof_product_space_inclusion}. Given arbitrary initial points $\mathbf{c}^0, \tilde{\mathbf{c}}^0, \mathbf{b}^0, \tilde{\mathbf{b}}^0, \mathbf{b}^{-1}, \tilde{\mathbf{b}}^{-1} \in \mathcal{H}^N$, this takes the form
    \begin{subequations}
    \begin{empheq}[left=\empheqlbrace,right=\quad\forall k\geq 0]{align}
        (\mathbf{a}^k, \tilde{\mathbf{a}}^k)   & =   J_{ \widetilde{A}} ((\mathbf{c}^k, \tilde{\mathbf{c}}^k)),   \label{eqn:BFRB_1}
        \\
        (\mathbf{b}^{k+1}, \tilde{\mathbf{b}}^{k+1})   & =   J_{\widetilde{C}} \Big(2 (\mathbf{a}^k, \tilde{\mathbf{a}}^k) - (\mathbf{c}^k, \tilde{\mathbf{c}}^k) - \big(2\widetilde{B}((\mathbf{b}^k, \tilde{\mathbf{b}}^k)) - \widetilde{B}((\mathbf{b}^{k-1}, \tilde{\mathbf{b}}^{k-1}))\big) \Big),   \label{eqn:BFRB_2}
        \\
        (\mathbf{c}^{k+1}, \tilde{\mathbf{c}}^{k+1})   & =   (\mathbf{c}^k, \tilde{\mathbf{c}}^k) + (\mathbf{b}^{k+1}, \tilde{\mathbf{b}}^{k+1}) - (\mathbf{a}^k, \tilde{\mathbf{a}}^k). \label{eqn:BFRB_3}
    \end{empheq}
    \end{subequations}
    Applying \citep[Thm.~3]{paper:rieger_tam_noncocoercive} and noting that $1/(8\max_i\{\alpha_iL_i\})>1$, we deduce that $(\mathbf{c}^k, \tilde{\mathbf{c}}^k)\rightharpoonup(\mathbf{c}^*, \tilde{\mathbf{c}}^*)\in\mathcal{H}^N\times\mathcal{H}^N$, $(\mathbf{a}^k, \tilde{\mathbf{a}}^k)\rightharpoonup(\mathbf{a}^*, \tilde{\mathbf{a}}^*)$, and $(\mathbf{b}^k, \tilde{\mathbf{b}}^k) \rightharpoonup (\mathbf{a}^*, \tilde{\mathbf{a}}^*) = J_{\widetilde{A}}((\mathbf{c}^*, \tilde{\mathbf{c}}^*)) \in (\widetilde{A} + \widetilde{B} + \widetilde{C})^{-1}(0)$. From the definition of $\widetilde{A}$, we have $J_{\widetilde{A}} = (J_{\Gamma A \Gamma}, 0)$, and hence $\mathbf{a}^* = J_{\Gamma A\Gamma}(\mathbf{c}^*)$, $\tilde{\mathbf{a}}^* = 0$, and
    \begin{subequations}
    \begin{empheq}[left=\empheqlbrace, right=\quad\forall k\geq0]{align}
    \mathbf{a}^k &= J_{\Gamma A\Gamma}(\mathbf{c}^k), \label{eqn:proof_ak_final1} \\
    \tilde{\mathbf{a}}^k &= 0. \label{eqn:proof_ak_final2}
    \end{empheq}
    \end{subequations}
    Denote $\mathbf{x}^k = \Gamma \mathbf{a}^k$, $\mathbf{y}^k = \Gamma \mathbf{b}^k$, $\mathbf{z}^k = \Gamma \mathbf{c}^k$, $\mathbf{x}^*=\Gamma \mathbf{a}^* = (x^*,\dots,x^*)$, and $\mathbf{z}^*=\Gamma\mathbf{c}^*$. Since bounded linear operators are weakly continuous~\citep[Lem.~2.41]{book:camo_bauschke_combettes}, we have that $\mathbf{x}^k\rightharpoonup \mathbf{x}^*$, $\mathbf{y}^k\rightharpoonup \mathbf{x}^*$, and $\mathbf{z}^k\rightharpoonup\mathbf{z}^*\in\mathcal{H}$.
    
    Using~\eqref{eqn:proof_ak_final1} and the definition of the resolvent, noting that $\Lambda A$ is maximally monotone since $\Lambda$ is self-adjoint and strongly monotone \citep[Prop.~20.24]{book:camo_bauschke_combettes}, we deduce that 
    \begin{equation}    \label{eqn:proof_xk_final}
    \mathbf{a}^k = J_{ \Gamma A \Gamma}(\mathbf{c}^k) \iff \mathbf{a}^k + \Gamma A \Gamma \mathbf{a}^k \ni \mathbf{c}^k \iff \mathbf{x}^k + \Gamma^2 A \mathbf{x}^k \ni \mathbf{z}^k \iff \mathbf{x}^k = J_{\Lambda A}(\mathbf{z}^k), \quad \forall k\geq0,
    \end{equation}
    establishing~\eqref{eqn:alg1_xk} in Algorithm~\ref{alg:main_algorithm}. Repeating the same argument, we also see that $\mathbf{x}^* = J_{\Lambda A}(\mathbf{z}^*)$.

    Denote $\mathbf{v}^k=2B(\mathbf{y}^k)-B (\mathbf{y}^{k-1})$, for all $k\geq0$. Applying Proposition~\ref{prop:resolvent_decompose} to~\eqref{eqn:BFRB_2} with $S=\widetilde{C}$ (\emph{i.e.,} with $C = N_{\{0\}}$ and $P = \begin{bmatrix} K\Gamma & M \end{bmatrix}$) followed by simplifying using Remark~\ref{rem:resolvent_decompose},~\eqref{eqn:proof_ak_final2}, and the definition of $\widetilde{B}$ yields
    \begin{subequations}    \label{eqn:proof_bk_simplified}
    \begin{empheq}[left=\empheqlbrace,right=\quad\forall k\geq0]{align} 
        \mathbf{u}^k &= \beta K\Gamma (2 \mathbf{a}^k - \mathbf{c}^k - \Gamma\mathbf{v}^k) - \beta M \tilde{\mathbf{c}}^k. \label{eqn:proof_bk_simplified3}
        \\
        \mathbf{b}^{k+1} &= 2\mathbf{a}^k - \mathbf{c}^k - \Gamma \mathbf{v}^k - \Gamma K \mathbf{u}^k, \label{eqn:proof_bk_simplified1}
        \\
        \tilde{\mathbf{b}}^{k+1} &= - \tilde{\mathbf{c}}^k - M \mathbf{u}^k, \label{eqn:proof_bk_simplified2}
    \end{empheq}
    \end{subequations}
    By substituting~\eqref{eqn:proof_ak_final2},~\eqref{eqn:proof_bk_simplified1}, and~\eqref{eqn:proof_bk_simplified2} into~\eqref{eqn:BFRB_3}, we obtain
    \begin{subequations}    \label{eqn:proof_ck_update}
    \begin{empheq}[left=\empheqlbrace,right=\quad\forall k\geq0]{align}     
        \mathbf{c}^{k+1} & = \mathbf{a}^k - \Gamma \mathbf{v}^k - \Gamma K \mathbf{u}^k, \label{eqn:proof_ck_update1}
        \\
        \tilde{\mathbf{c}}^{k+1} & = - M \mathbf{u}^k.
        \label{eqn:proof_ck_update2}
    \end{empheq}
    \end{subequations}
    By substituting~\eqref{eqn:proof_ck_update1} into~\eqref{eqn:proof_bk_simplified1} and using the definitions of $\mathbf{x}^k$, $\mathbf{y}^k$, and $\mathbf{z}^k$, we obtain~\eqref{eqn:alg1_yk}:
    \begin{equation}    \label{eqn:proof_yk_final}
    \mathbf{y}^{k+1} = \Gamma\mathbf{b}^{k+1} = \Gamma (\mathbf{a}^k + \mathbf{c}^{k+1} - \mathbf{c}^k) = \mathbf{x}^k + \mathbf{z}^{k+1} - \mathbf{z}^k,\quad \forall k\geq0.
    \end{equation}
    By substituting~\eqref{eqn:proof_ck_update} into~\eqref{eqn:proof_bk_simplified3}, we obtain
    \begin{empheq}[right=\quad\forall k\geq1]{align}
        \mathbf{u}^k & = \beta K \Gamma \big(2\mathbf{a}^k - (\mathbf{a}^{k-1} - \Gamma \mathbf{v}^{k-1} - \Gamma K \mathbf{u}^{k-1}) - \Gamma \mathbf{v}^k \big) + \beta M^2 \mathbf{u}^{k-1},  \nonumber
        \\
        & = \beta\begin{bmatrix} K\Gamma & M \end{bmatrix} \begin{bmatrix} K\Gamma & M \end{bmatrix}^* \mathbf{u}^{k-1} + \beta K \Gamma \big(2\mathbf{a}^k - \mathbf{a}^{k-1} - \Gamma (\mathbf{v}^k - \mathbf{v}^{k-1}) \big), \nonumber
        \\
        & = \mathbf{u}^{k-1} + \beta K \Gamma \big(2\mathbf{a}^k - \mathbf{a}^{k-1} - \Gamma (\mathbf{v}^k - \mathbf{v}^{k-1}) \big), \label{eqn:proof_uk_simplified}
    \end{empheq}
    where the final equality follows from $\begin{bmatrix} K\Gamma & M \end{bmatrix}\begin{bmatrix} K\Gamma & M \end{bmatrix}^*=(1/\beta)I$. Substituting~\eqref{eqn:proof_uk_simplified} into~\eqref{eqn:proof_ck_update1} and observing that~\eqref{eqn:proof_ck_update1} implies $\Gamma K\mathbf{u}^k=\mathbf{a}^k-\mathbf{c}^{k+1}-\Gamma\mathbf{v}^k$, we deduce that
    \begin{empheq}[right=\quad\forall k\geq1]{align}
        \mathbf{c}^{k+1} & = \mathbf{a}^k - \Gamma \mathbf{v}^k - \Gamma K \big(\mathbf{u}^{k-1} + \beta K \Gamma \big(2\mathbf{a}^k - \mathbf{a}^{k-1} - \Gamma (\mathbf{v}^k - \mathbf{v}^{k-1}) \big)\big), \nonumber
        \\
        & = \mathbf{c}^{k} + \mathbf{a}^k - \mathbf{a}^{k-1} - \Gamma (\mathbf{v}^k - \mathbf{v}^{k-1}) - \beta \Gamma K^2 \Gamma \big(2\mathbf{a}^k - \mathbf{a}^{k-1} - \Gamma (\mathbf{v}^k - \mathbf{v}^{k-1}) \big), \nonumber
        \\
        & = \mathbf{c}^{k} - \mathbf{a}^k + (I - \beta \Gamma K^2\Gamma) \big(2\mathbf{a}^k - \mathbf{a}^{k-1} - \Gamma (\mathbf{v}^k - \mathbf{v}^{k-1}) \big). \label{eqn:proof_ck_final}
    \end{empheq}
    Applying $\Gamma$ to both sides of~\eqref{eqn:proof_ck_final} and using the definitions of $\mathbf{x}^k$, $\mathbf{y}^k$, $\mathbf{z}^k$, and $K$, we obtain~\eqref{eqn:alg1_zk}:
    \[\mathbf{z}^{k+1} = \mathbf{z}^k - \mathbf{x}^k +  \left(I - \beta \Lambda \frac{I-W}{2}\right) \big(2\mathbf{x}^k - \mathbf{x}^{k-1} - \Lambda (\mathbf{v}^k - \mathbf{v}^{k-1}) \big), \quad \forall k\geq1.
    \]
    Thus, we have established~\eqref{eqn:alg1_vk}-\eqref{eqn:alg1_yk}. It remains to show the initialisation~\eqref{eqn:alg1_x0}-\eqref{eqn:alg1_x1}.

    To this end, first note that by definition $\mathbf{z}^0=\Gamma\mathbf{c}^0$, $\mathbf{y}^{-1}=\Gamma\mathbf{b}^{-1}$, and $\mathbf{y}^0=\Gamma\mathbf{b}^0$. From~\eqref{eqn:proof_xk_final}, we have $\mathbf{x}^{0}=J_{\Lambda A}(\mathbf{z}^0)$, and $\mathbf{x}^{1}=J_{\Lambda A}(\mathbf{z}^1)$, which establishes~\eqref{eqn:alg1_x0} and~\eqref{eqn:alg1_x1}. Since $\mathbf{y}^{-1},\mathbf{y}^0,\tilde{\mathbf{c}}^0$ were arbitrary, we may, in particular, choose $\mathbf{y}^{-1}=\mathbf{y}^0$, which implies $\mathbf{v}^0=B(\mathbf{y}^0)$, and  
    \begin{equation} \label{eqn:c0_initialisation}
    \tilde{\mathbf{c}}^0 = M^{-1} K(2\mathbf{x}^0 - \mathbf{z}^0 - \Gamma^2 \mathbf{v}^0) = M^{-1}K\Gamma (2\mathbf{a}^0 - \mathbf{c}^0 - \Gamma \mathbf{v}^0),
    \end{equation}
    which is well-defined since $M\succ0$. Substituting~\eqref{eqn:c0_initialisation} into ~\eqref{eqn:proof_bk_simplified3} gives $\mathbf{u}^0=0$. Substituting this into~\eqref{eqn:proof_bk_simplified1} gives
    $\mathbf{y}^1 =\Gamma\mathbf{b}^1=\Gamma(2\mathbf{a}^1-\mathbf{c}^1-\Gamma\mathbf{v}^0) = 2\mathbf{x}^0 - \mathbf{z}^0 -\Lambda\mathbf{v}^0$, which establishes~\eqref{eqn:alg1_y1}. Finally,~\eqref{eqn:alg1_z1} follows from~\eqref{eqn:proof_yk_final}, which completes the derivation of Algorithm \ref{alg:main_algorithm}.
\end{proof}

\begin{remark} \label{rem:stepsize_graph}
The convergence of Algorithm~\ref{alg:main_algorithm} is independent of the communication graph, and agents are allowed to choose the largest step sizes permitted by their single-valued operators according to $\alpha_i< 1/(8L_i)$. In fact, even when agents use homogeneous step sizes, graph-independent step sizes can allow agents to use a larger step size than in~\eqref{alg:pdtr}. To see this, suppose $\alpha=\alpha_1=
\dots=\alpha_N$ and $W = I - (1/\tau)\mathcal{L}$, where $\tau>(1/2)  \lambda_\text{max}(\mathcal{L})$. Additionally, let
\begin{equation}
\tau < \frac{2}{3} \lambda_\text{max}(\mathcal{L}).
\label{eqn:mixing_matrix_condition}
\end{equation}
We then have
\[
\tau < \frac{2}{3} \lambda_\text{max}(\mathcal{L}) \iff \lambda_\text{min}(W) = 1-\frac{1}{\tau}\lambda_\text{max}(\mathcal{L}) < -\frac{1}{2} \iff \frac{1 + \lambda_\text{min}(W)}{4L} < \frac{1}{8L}.
\]
Thus, the step size upper-bound for Algorithm~\ref{alg:main_algorithm} is greater than that of~\eqref{alg:pdtr} when the mixing matrix is chosen according to~\eqref{defn:laplacian_mixing} and~\eqref{eqn:mixing_matrix_condition}. Note that it is always possible to choose $\tau$ in this manner since $\lambda_\text{max}(\mathcal{L})>0$ for connected graphs \citep[Lem.~1.7]{textbook:chung_spectral_graph}. Furthermore, with this choice of $\tau$, $W$ is a special case of the Laplacian-based constant edge weight matrix~\eqref{defn:laplacian_mixing} where that parameter $\tau$ has an upper bound.
\end{remark}

\section{Numerical experiments} \label{sec:numerics_part1}
In this section, we numerically compare \citep[Alg.~1]{paper:malitsky_tam_minmax}, as presented in~\eqref{alg:pdtr}, and our proposed Algorithm~\ref{alg:main_algorithm} across two different problems: \emph{robust least squares} and a \emph{zero-sum matrix game}. In the instances we consider, both problems have a unique solution which will be denoted by $x^*\in\mathcal{H}$.

All experiments use $N=10$ agents and several different communication graphs (barbell, cycle, and 2D grid graphs). For brevity, we display the results using a cycle graph in this section and include the remainder in Appendix~\ref{sec:appendix}. For the mixing matrix $W$, we used the Laplacian-based constant edge-weight matrix~\eqref{defn:laplacian_mixing} with $\tau=0.505\lambda_\text{max}(\mathcal{L})$, so as to satisfy~\eqref{eqn:mixing_matrix_condition}. All experiments were implemented in Python 3.12.7 using \texttt{NumPy}~v1.26.4 on a virtual machine running Windows 10 with 16GB RAM and one core of an AMD EPYC 7763 CPU @ 2.45 GHz.

Algorithm performance is evaluated across two metrics:
\begin{enumerate}[(i)] 
    \item The \emph{relative error} given by 
 \[\text{relative error} = \frac{\|\mathbf{x}^k - \mathbf{x}^*\|}{\|\mathbf{x}^*\|},\]
  where $\mathbf{x}^k$ is defined in~\eqref{eqn:alg1_xk} for Algorithm~\ref{alg:main_algorithm} and~\eqref{alg:pdtr_zk} for~\eqref{alg:pdtr}, and $\mathbf{x}^*=(x^*, \dots, x^*)\in\mathcal{H}^N$.
  
    \item The \emph{normalised residuals} (\emph{i.e.,} normalised by step size) given by
 \[\text{normalised residual} = \|\mathbf{z}^k - \mathbf{z}^{k-1}\|_{\Lambda^{-1}},\]
 where $\mathbf{z}^k$ is defined in~\eqref{eqn:alg1_zk} for Algorithm \ref{alg:main_algorithm} and~\eqref{alg:pdtr_zk} for~\eqref{alg:pdtr}. The stepsize matrix $\Lambda$ is taken as $\Lambda=\mathrm{diag}(\alpha_1, \dots, \alpha_N)$ for Algorithm \ref{alg:main_algorithm} and $\Lambda=\alpha I$ for~\eqref{alg:pdtr} with $\alpha=\max_i\{\alpha_i\}$. 
\end{enumerate}
Recall that $L_i$ denotes the Lipschitz constant of the operator $B_i$. Algorithm~\ref{alg:main_algorithm} is tested using three step size and parameter combinations:
\begin{enumerate}[(i)] 
\item heterogeneous step sizes $\alpha_i = 0.9/(8L_i)$ and parameter $\beta=\beta_\text{max}=0.9(\max_i\{\alpha_i\})^{-1}$,
\item heterogeneous step sizes $\alpha_i = 0.9/(8L_i)$ and parameter $\beta = \beta_\text{norm} = 0.9\|\Lambda^{1/2} ((I-W)/2) \Lambda^{1/2}\|^{-1}$, and
\item homogenous step sizes $\alpha_i~=~0.9/(8\max_i\{L_i\})$ and parameter $\beta=\beta_\text{max}=0.9(\max_i\{\alpha_i\})^{-1}$.
\end{enumerate}
Algorithm~\eqref{alg:pdtr} is tested using the step size $\alpha~=~0.9(1+\lambda_{\text{min}}(W))/(4\max_i\{L_i\})$.

\subsection{Robust least squares} \label{sec:numerics:rls}
In the first numerical experiment, we consider the \emph{robust least squares problem} \cite{paper_elghaoui_robust_ls}. Given a data matrix $M\in\mathbb{R}^{d\times p}$ and noisy vector $\tilde{v}\in\mathbb{R}^d$, we seek to recover a pair of vectors $(u,v)\in\mathbb{R}^p\times\mathbb{R}^d$ such that $u$ minimises the worst-case error in the least-square sense. We assume the noise in $\tilde v$ is bounded in norm by some constant $\delta>0$. Consider a population of $N$ agents such that the data matrix and the noisy vector are partitioned across the agents. In other words, suppose $d=\sum_i d_i$ and denote
\[M = \left[\begin{array}{c|c|c} M_1 & \dots & M_N \end{array}\right],\quad \tilde{v}=\begin{pmatrix}\tilde{v}_1 \\\vdots \\ \tilde{v}_N\end{pmatrix}\in\mathbb{R}^d,\quad v=\begin{pmatrix}v_1 \\ \vdots\\ v_N\end{pmatrix}\in\mathbb{R}^d,\]
where $M_i\in\mathbb{R}^{d_i\times p}$ and $\tilde{v}_i,v_i\in\mathbb{R}^{d_i}$ are local to agent $i$. 
Under this notation, the distributed robust least squares problem can then be formulated as 
\begin{subequations}    \label{eqn:robust_ls_full}
\begin{empheq}[left=\empheqlbrace]{align}
    \min_{u\in\mathbb{R}^p} \max_{v\in\mathbb{R}^d} & \,\sum_{i=1}^N \|M_i u - v_i\|^2,
    \\
    \text{s.t.} \quad & \sum_{i=1}^N\|v_i-\tilde v_i\| \leq \delta.
\end{empheq}    
\end{subequations}
Since~\eqref{eqn:robust_ls_full} is not concave in $v$, we follow the approach of \citep{paper:thekumparampil_minmax} and solve a penalty formulation. Given a parameter $\lambda>0$, the penalty formulation of~\eqref{eqn:robust_ls_full} is
\begin{equation}    \label{eqn:robust_ls}
    \min_{u\in\mathbb{R}^p} \max_{v\in\mathbb{R}^d} g(u,v) = \sum_{i=1}^N \|M_i u - v_i\|^2 - \lambda \|v_i - \tilde v_i\|^2.
\end{equation}
Note that when $\lambda\geq1$, the objective function $g$ is convex-concave. Moreover, when $\lambda>1$ and $M$ has full column rank, $g$ is strongly convex-strongly concave, and hence~\eqref{eqn:robust_ls} has a unique solution. Define the operators $E_1,\dots,E_N:\mathbb{R}^{d}\to\mathbb{R}^{d_i}$ by $E_i(v)=v_i$. Then, following~\eqref{eqn:example_minmax_operators}, solutions to~\eqref{eqn:robust_ls} can then be characterised as solutions of the monotone inclusion~\eqref{eqn:finite_sum_inclusion} with
\[
z=\begin{pmatrix}u\\v\end{pmatrix} \in\mathcal{H}=\mathbb{R}^p\times\mathbb{R}^d,\quad A_i(z) = \binom{0}{0},\quad  B_i(z) = 2 \begin{bmatrix} M_i^\top M_i & -M_i^\top E_i \\ -E_i^\top M_i & (\lambda-1)E_i^\top E_i \end{bmatrix} z - \begin{pmatrix} 0 \\ 2\lambda E_i^\top\tilde{v}_i\end{pmatrix}.
\]
 
The performance of Algorithm~\ref{alg:main_algorithm} and~\eqref{alg:pdtr} in solving~\eqref{eqn:robust_ls} was compared across two datasets. The first is the synthetic dataset used by \citet*{paper:yang_robust_ls}: $\lambda=3$, $M\in\mathbb{R}^{1000\times500}$ with entries sampled from $\mathcal{N}(0,1)$, and $\tilde{v}=M\tilde{u}+\epsilon$, where $\tilde{u}\in\mathbb{R}^{500}$ has entries sampled from $\mathcal{N}(0,0.1)$ and $\epsilon\in\mathbb{R}^{1000}$ has entries sampled from $\mathcal{N}(0,0.01)$. The second dataset is a subset of the California Housing dataset~\citep{data:california_housing}, containing $d=200$ data points regarding the median house value in California based on $p=8$ attributes, preprocessed using \texttt{scikit-learn}'s \texttt{StandardScaler} function. We follow \citet*{paper:zhang_robust_ls} in using this dataset and set $\lambda=50$. In both datasets, $M$ has full column rank, and the unique solution of the problem is computed by solving the first-order optimality conditions for~\eqref{eqn:robust_ls} as a linear system.

The results for the synthetic dataset are averaged over five realisations of $M$ and $\tilde{v}$, and use the initial points~$\mathbf{z}^0 = 0$ and $\mathbf{y}^0=0$. The results for the California Housing results are averaged over five runs using $\mathbf{y}^0=0$ and a random initial point~$\mathbf{z}^0$, with entries sampled from $U(0,1)$ and  scaled such that $\|\mathbf{z}^0\|=10$.

\begin{figure}[ht!]
    \begin{subfigure}[t]{0.5\textwidth}
        \centering
        \includegraphics[width=\textwidth]{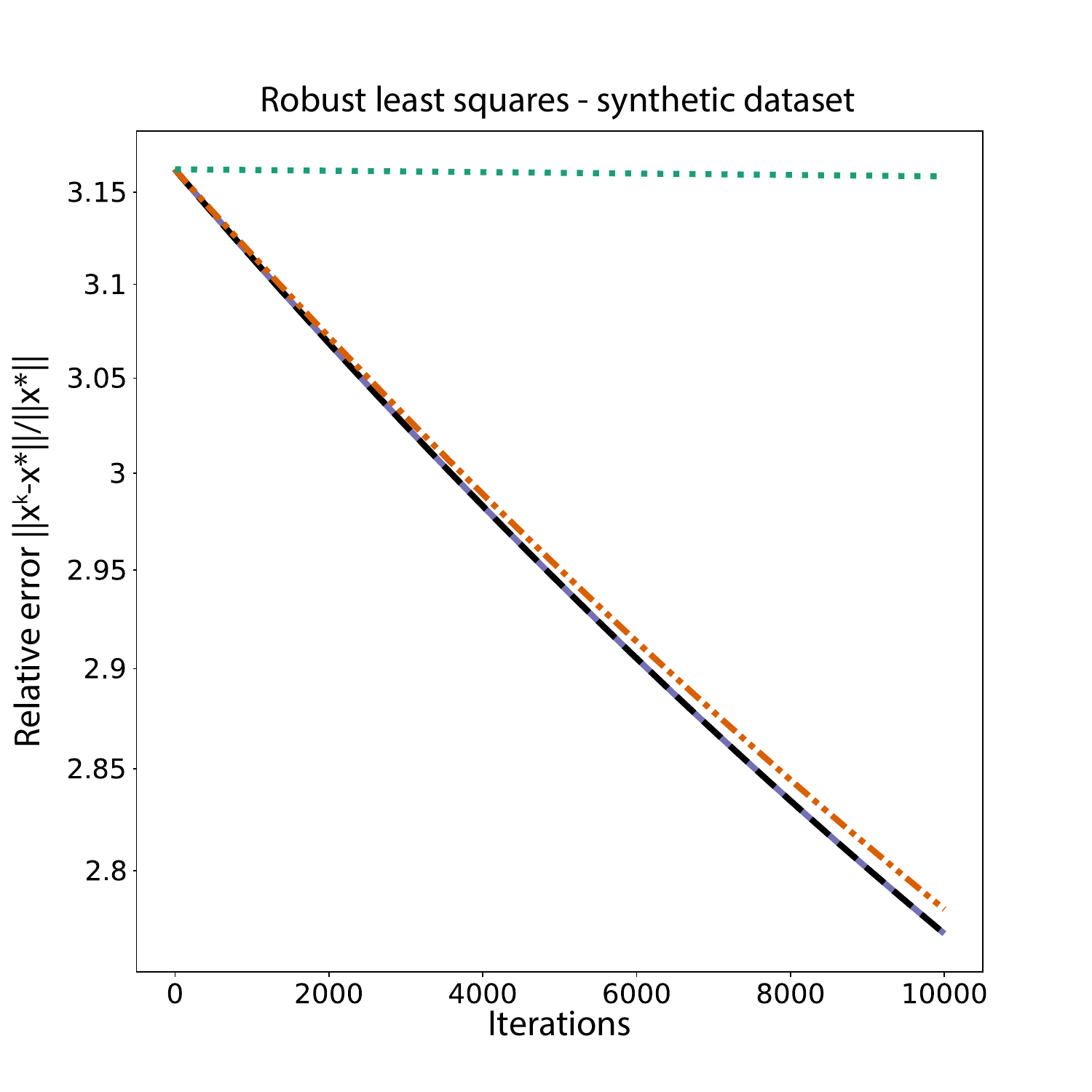}
        \caption{Relative errors vs. iterations.}
        \label{fig:rls_synthetic_relerror}
    \end{subfigure}
    % \hfill
    \begin{subfigure}[t]{0.5\textwidth}
        \centering
        \includegraphics[width=\textwidth]{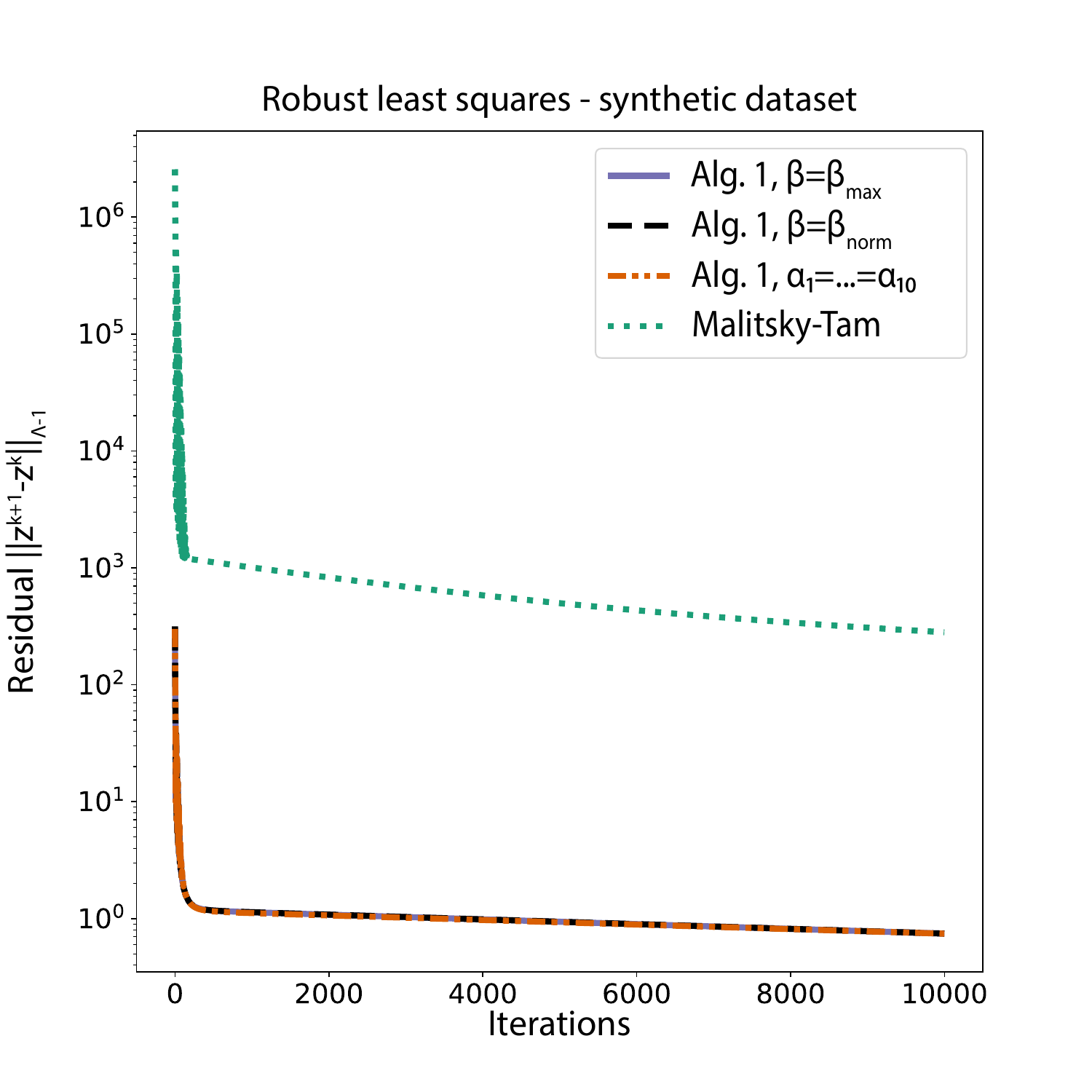}
        \caption{Residuals vs. iterations.}
        \label{fig:rls_synthetic_resid}
    \end{subfigure}
    
    \caption{Results for robust least squares using the synthetic dataset and a cycle communication graph averaged over five realisations of $M$ and $\tilde{v}$. Step sizes are given at the start of Section \ref{sec:numerics_part1}.}
    \label{fig:rls_synthetic}
\end{figure}
 
\begin{figure}[ht!]
    \begin{subfigure}{0.5\textwidth}
    \centering
    \includegraphics[width=\textwidth]{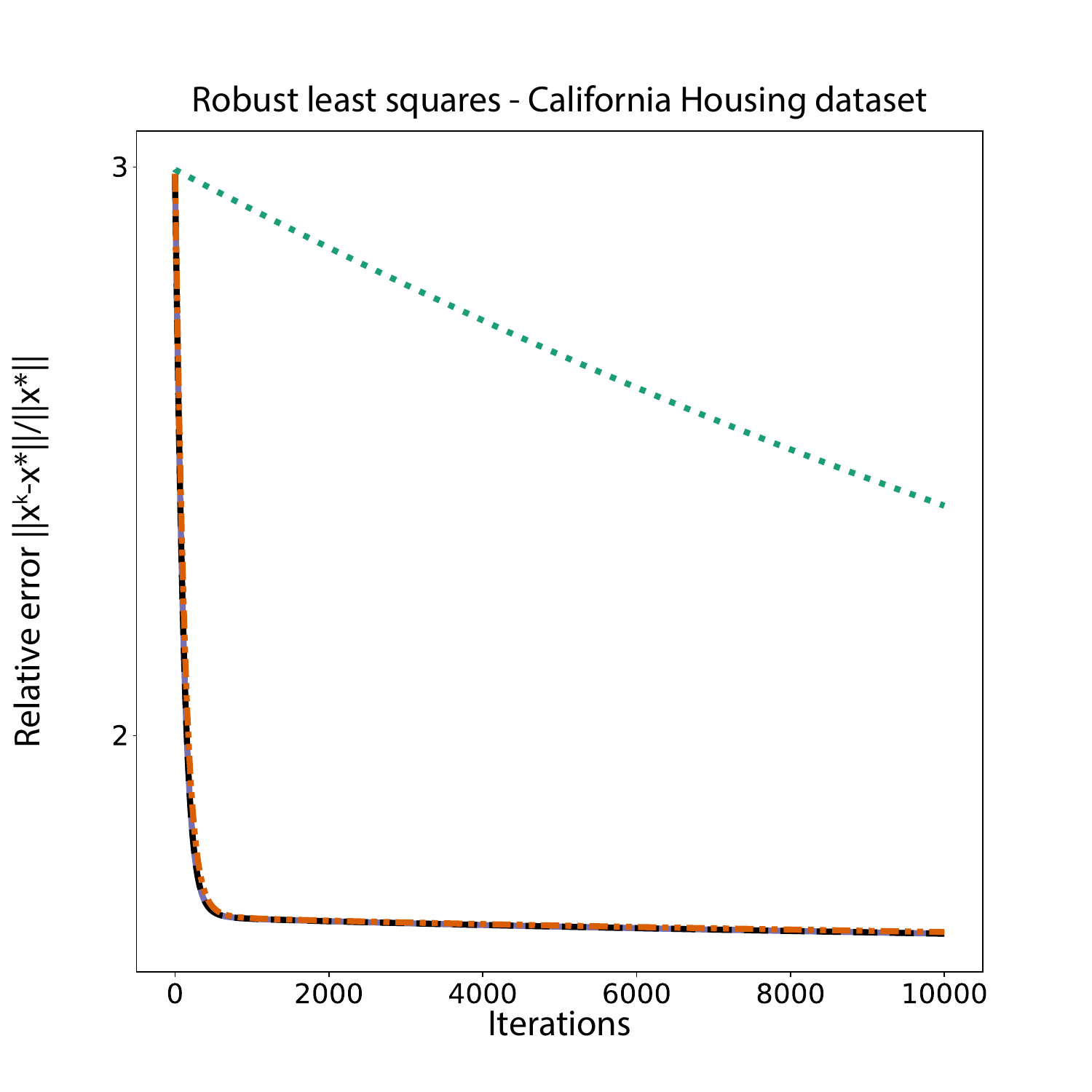}
    \caption{Relative errors vs. iterations.}
    \label{fig:rls_california_relerror}
    \end{subfigure}
    \hfill
    \begin{subfigure}{0.5\textwidth}
    \centering
    \includegraphics[width=\textwidth]{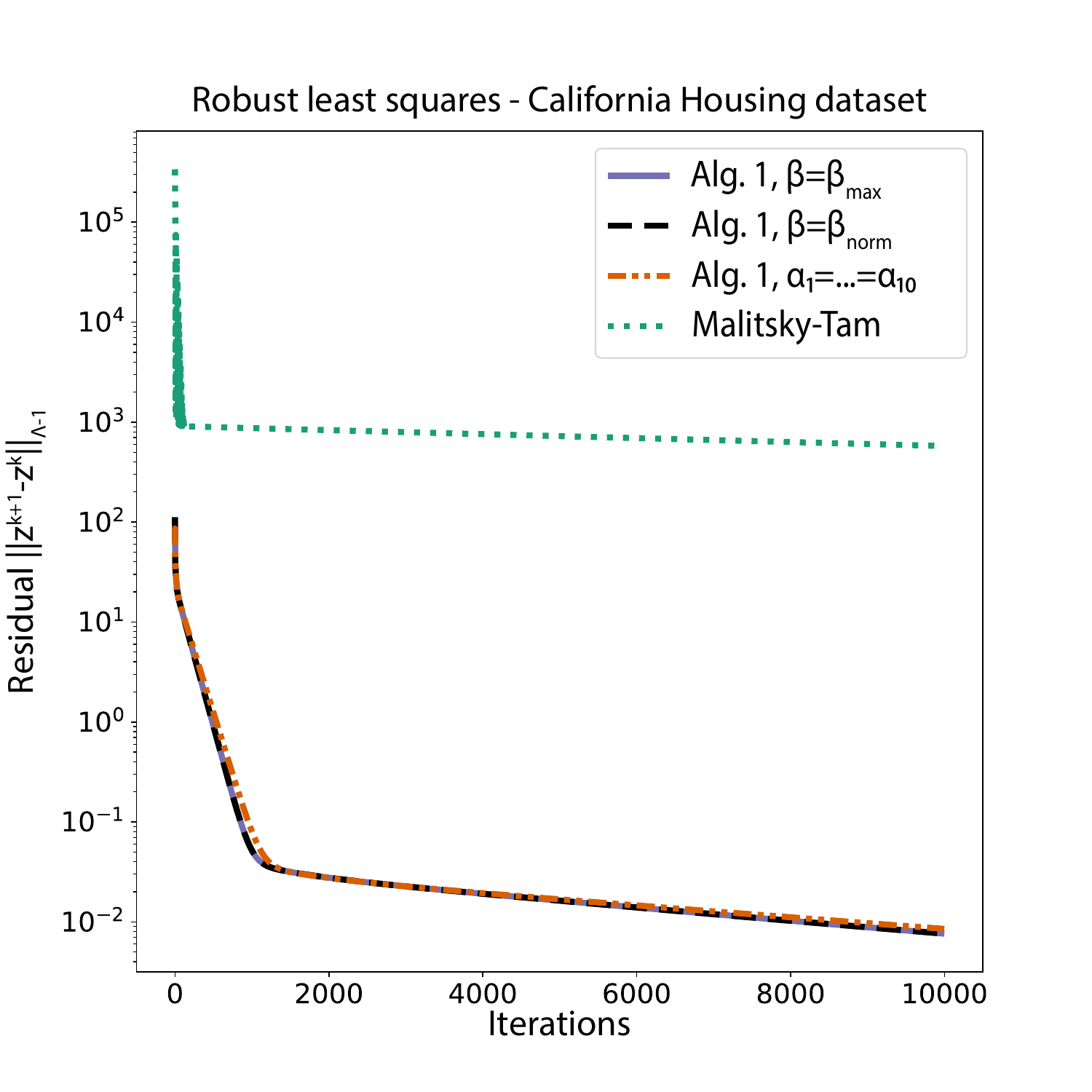}
    \caption{Residuals vs. iterations.}
    \label{fig:rls_california_resid}
    \end{subfigure}
    
    \caption{Robust least squares results using the California Housing dataset and a cycle communication graph, averaged over five runs with different initial points. Step sizes are given at the start of Section \ref{sec:numerics_part1}.}
    \label{fig:rls_california}
\end{figure}

Figures \ref{fig:rls_synthetic} and \ref{fig:rls_california} show that Algorithm \ref{alg:main_algorithm} outperformed~\eqref{alg:pdtr} across both datasets and metrics, regardless of the choice of step size or parameter $\beta$. This is potentially because of the choice of $\tau$ and the fact that the step size bound for~\eqref{alg:pdtr} depends on $W$ (and thus on $\tau$), whereas the step size bound for Algorithm \ref{alg:main_algorithm} does not. Consequently, the step size for Algorithm~\eqref{alg:pdtr} can be much smaller than that of Algorithm \ref{alg:main_algorithm} (in this experiment, approximately 100 times smaller). Using heterogenous step sizes versus homogeneous step sizes for Algorithm \ref{alg:main_algorithm} did not significantly affect performance, nor did the choice of $\beta$. The choice of communication graphs did not change these results (see Figures \ref{supp:fig:rls_synthetic}-\ref{supp:fig:rls_california}).

\subsection{Zero-sum matrix game} \label{sec:numerics:matrix_game}
In the second numerical experiment, we consider a zero-sum matrix game between two teams of $N$ players, where team one has the payoff matrix $M=\sum_{i=1}^N M_i\in\mathbb{R}^{d\times p}$ and team two has payoff matrix $-M$. Each player in team one is paired with a player from team two, and pair $i$ has the payoff matrices $(M_i,-M_i)$. \emph{Nash equilibria} of this game can be found as saddle-points of the min-max problem
\begin{equation}    \label{eqn:matrix_game}
    \min_{u\in\Delta^p} \max_{v\in\Delta^d}\, \sum_{i=1}^N \langle M_i u, v \rangle,
\end{equation}
where $\Delta^p$ is the unit simplex, $\Delta^p = \{u = (u_1, \dots, u_p)^\top\in\mathbb{R}^p_{\geq0} \,|\, \textstyle\sum_{j=1}^p u_j = 1\}$.
Following from~\eqref{eqn:example_minmax_operators}, solutions of~\eqref{eqn:matrix_game} can be characterised as solutions of~\eqref{eqn:finite_sum_inclusion} with
\[
z=\begin{pmatrix}u\\v\end{pmatrix} \in\mathcal{H}=\mathbb{R}^p\times\mathbb{R}^d,\quad A_i(z) = \begin{pmatrix} N_{\Delta^p}(u) \\ N_{\Delta^d}(v) \end{pmatrix},\quad  B_i(z) = \begin{bmatrix} 0 & M_i^* \\ -M_i & 0 \end{bmatrix} z.
\]
In the numerical experiments, we set $p=d=8$ and choose the matrices $M_i\in\mathbb{R}^{8\times 8}$ according to
\begin{equation}    \label{eqn:matrix_game_Mmatrix}
    M_i = s_i I - K_i,\quad s_i>\|K_i\|_2,\quad K_i>0,
\end{equation}
where $K>0$ denotes a matrix with positive entries. With this choice of $M_1,\dots,M_N$, the game~\eqref{eqn:matrix_game} is \emph{completely mixed} \citep[\S2]{paper:cohen_completely_mixed} and thus has a unique Nash equilibrium \citep[Thm.~B]{paper:raghavan_Mmatrix}. This equilibrium is computed using \texttt{NashPy} v0.041 to calculate the relative error.

Both algorithms were initialised with $\mathbf{y}^0=\mathbf{z}^0=\mathbf{0}$. The results are averaged over five realisations of $M_1,\dots,M_N$, which we generate according to~\eqref{eqn:matrix_game_Mmatrix} with $s_j = 1.1\|K_j\|_2$ and $K_j=jL_j$, where $L_j\in\mathbb{R}^{p\times p}$ has entries sampled from $U(0,1)$.

Figures \ref{fig:matrix_relerror} and \ref{fig:matrix_resid} show that Algorithm~\ref{alg:main_algorithm} outperforms~\eqref{alg:pdtr} in both metrics. This holds for both choices of $\beta$ and using heterogeneous or homogeneous step sizes. Using heterogeneous agent step sizes in Algorithm~\ref{alg:main_algorithm} significantly improved performance compared to using homogeneous step sizes, potentially because the Lipschitz constants of the operators $B_i$ were significantly different across agents ($(\max_i L_i - \min_i L_i)/\max_i L_i \approx 0.89$). Using homogeneous step sizes for Algorithm~\ref{alg:main_algorithm} provided better performance than~\eqref{alg:pdtr}, although the performance increase was smaller than using heterogeneous step sizes. The choice of communication graphs did not change these results (see Figure \ref{supp:fig:matrix}).

\begin{figure}[ht!]
    \begin{subfigure}{0.5\textwidth}
        \centering
        \includegraphics[width=\textwidth]{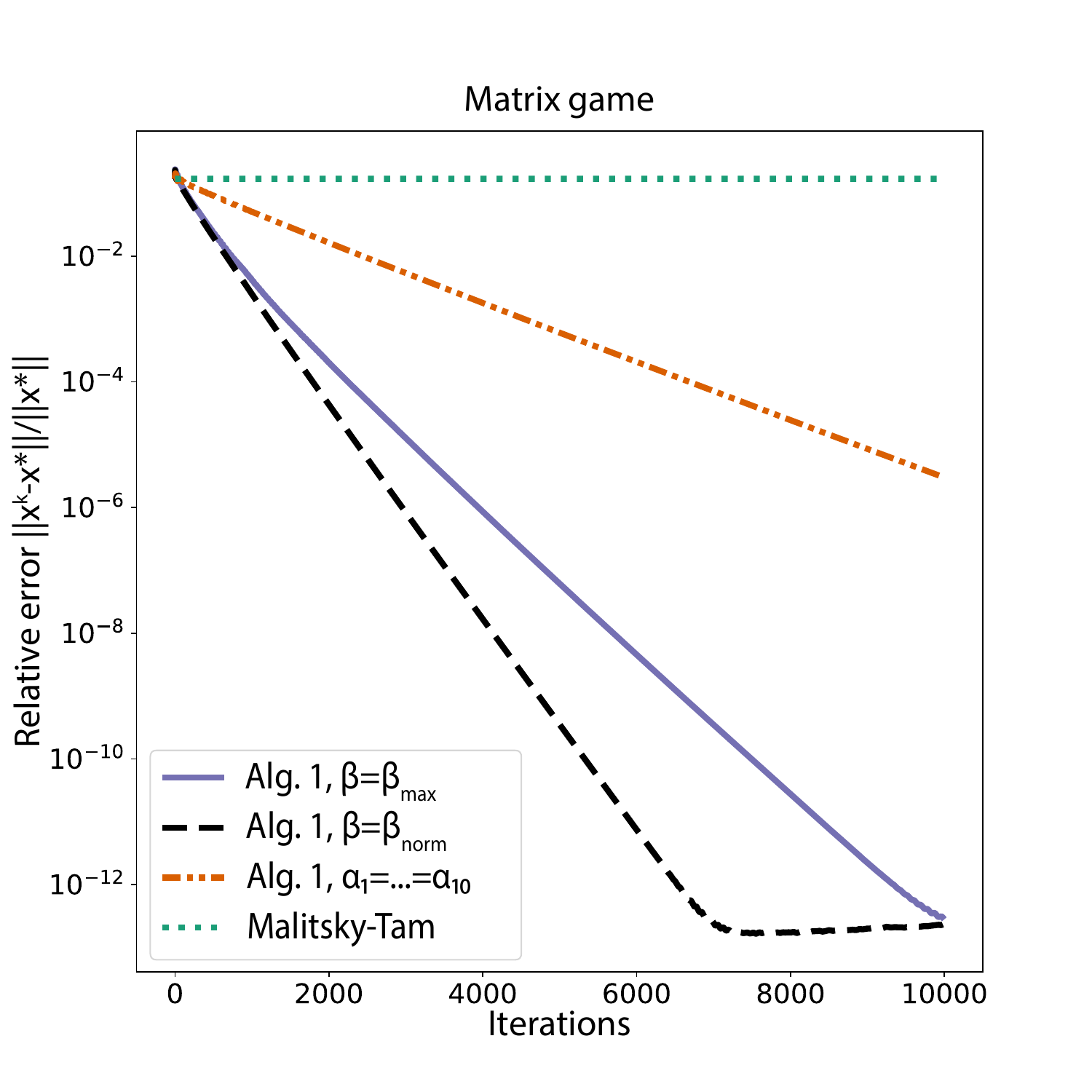}
        \caption{Relative errors vs. iterations.}
        \label{fig:matrix_relerror}
    \end{subfigure}
    \hfill
    \begin{subfigure}{0.5\textwidth}
        \centering
        \includegraphics[width=\textwidth]{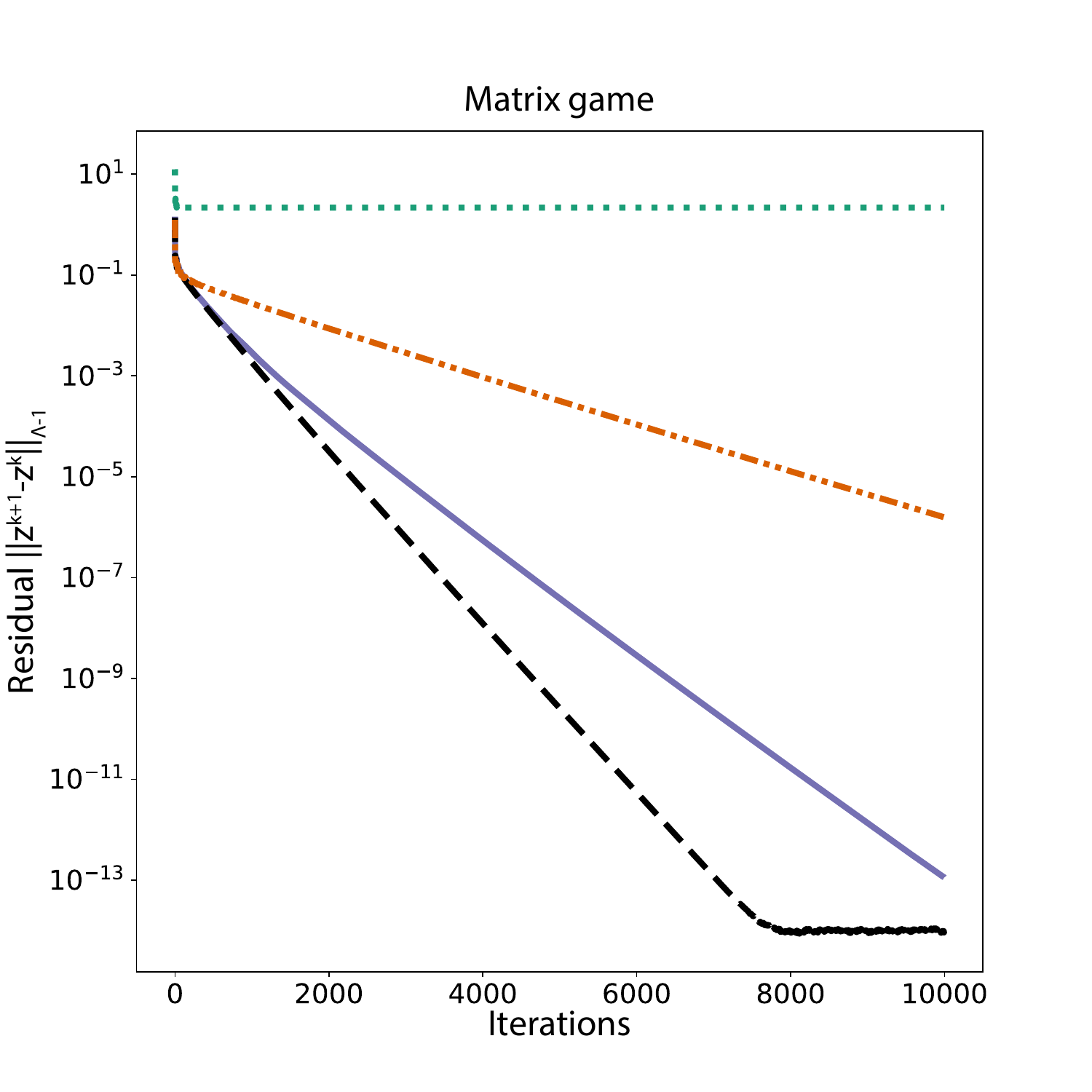}
        \caption{Residuals vs. iterations.}
        \label{fig:matrix_resid}
    \end{subfigure}

    \caption{Results for matrix game using a cycle communication graph averaged over five realisations of $M_1,\dots,M_N$. Step sizes are given at the start of Section \ref{sec:numerics_part1}.}
\end{figure}

\section{A memory-efficient approach for aggregative games}   \label{sec:boosted_algorithm}
In this section, we propose an alternative implementation of Algorithm~\ref{alg:main_algorithm} for \emph{aggregative games} \citep{paper:corchon_agg_game}. This implementation exploits the problem's aggreative structure to reduce the number of variables needed as compared to a direct application of Algorithm~\ref{alg:main_algorithm}. Precisely, we consider an aggregative game between $N$ noncooperative players where player $i$ seeks to
\begin{subequations}    \label{eqn:aggregate_game}
\begin{empheq}[left=\empheqlbrace, right={\quad\forall i=1,\dots,N.}]{align}
    \displaystyle\min_{u_i\in\mathbb{R}^{p}}\, & g_i(u_i) + f_i\big(u_i, \textstyle\sum_{j=1}^N u_j\big),
    \\
    \text{s.t.} \quad & u_i \in \Omega_i, 
    \\
    & M \sum_{j=1}^N u_j \leq b, \label{eqn:aggregate_game_coupling}
\end{empheq}
\end{subequations}
Here the function $g_i:\mathbb{R}^p\to\mathbb{R}\cup\{+\infty\}$ is proper lsc convex and $f_i:\mathbb{R}^p\times\mathbb{R}^p\to\mathbb{R}$ is convex with $L_i$-Lipschitz continuous gradient. The constraint set $\Omega_i\subseteq\mathbb{R}^p$ is nonempty closed convex, $M\in\mathbb{R}^{d\times p}$, and $b\in\mathbb{R}^d$. We assume \emph{Slater's condition} \citep[3.2.7]{book:borwein_lewis} holds for \eqref{eqn:aggregate_game} as detailed in the following assumption.
\begin{assum}[Slater's condition]   \label{assum:slaters}
    There exists $(u_1,\dots,u_N)\in\mathbb{R}^{p}\times\dots\times\mathbb{R}^p$ such that ${u_i\in\dom (g_i + \iota_{\Omega_i})}$ for all $i=1,\dots, N$, and
    $M \sum_{j=1}^N u_j < b$.
\end{assum}

In this work, solutions to~\eqref{eqn:aggregate_game} are understood in the sense of \emph{generalised Nash equilibria} \citep{paper:debreu_gne}. Precisely, a feasible point $u = (u_1, \dots, u_N)$ is said to be a \emph{generalised Nash equilibrium} of~\eqref{eqn:aggregate_game} if
\begin{equation} \label{eqn:GNE}
    g_i(u_i) + f_i(u_i, \sum_{j=1}^N u_j) \leq \inf\left\{g_i(\bar{u}_i) + f_i(\bar{u}_i, \bar{u}_i+\displaystyle\sum_{j\neq i} u_j)\,\big|\, \bar{u}_i\in\Omega_i,\,M\bar{u}_i\leq b-M\sum_{j\neq i}u_j\right\}\hfill\forall i=1,\dots,N.
\end{equation}
In other words, if a point is a Nash equilibrium then players have no incentive to unilaterally change strategy. The following proposition characterised Nash equilibrium in terms as zeros of an inclusion. 

\begin{prop}
    Suppose Assumption \ref{assum:slaters} holds. Then $u=(u_1,\dots,u_N) \in \mathbb{R}^{p}\times\dots\times\mathbb{R}^{p}$ is a generalised Nash equilibrium of~\eqref{eqn:aggregate_game} if and only if there exists a dual variable $v=(v_1,\dots,v_N) \in \mathbb{R}^{d}\times\dots\times\mathbb{R}^{d}$ such that $(u,v)$ satisfies
    \begin{equation} \label{eqn:aggregate_game_inclusion}
    \begin{pmatrix} 0 \\ 0 \end{pmatrix} \in \begin{pmatrix} \partial (g_i + \iota_{\Omega_i})(u_i) \\ N_{\mathbb{R}^{p}_{\geq0}}(v_i) \end{pmatrix} + \begin{pmatrix} \grad_{u_i} f_i(u_i, \textstyle\sum_{j=1}^N u_j ) + M^\top v_i \\ b - M\sum_{j=1}^N u_j \end{pmatrix}, \quad \forall i=1,\dots,N.
    \end{equation}
\end{prop}

\begin{proof}
Note that $u_i\mapsto f_i(u_i,\, u_i+u') + g_i(u_i) + \iota_{\Omega_i}(u_i)$ is convex for all $u'\in\mathbb{R}^p$, and that $f_i$ has full domain. Hence, as Assumption~\ref{assum:slaters} holds, the \emph{Karush-Kuhn-Tucker conditions} for~\eqref{eqn:aggregate_game} are necessary and sufficient for~\eqref{eqn:GNE} \citep[Thm.~3.2.8, Prop.~3.2.3]{book:borwein_lewis}, and are given by
\begin{subequations}    \label{eqn:proof_kkt}
    \begin{empheq}[left=\empheqlbrace, right={\quad\forall i=1,\dots,N.}]{align}
    0 &\in \partial_{u_i} \big( g_i(u_i) + \iota_{\Omega_i}(u_i) + f_i(u_i, \textstyle\sum_{j=1}^N u_j)  + \langle v_i, M \sum_{j=1}^N u_j - b \rangle \big),   \label{eqn:proof_kkt_a}
    \\
    0 &\geq M \textstyle\sum_{j=1}^N u_j - b, \label{eqn:proof_kkt_b}
    \\
    0 &\leq v_i,    \label{eqn:proof_kkt_c}
    \\
    0 &= \langle v_i, M \textstyle\sum_{j=1}^N u_j - b\rangle,  \label{eqn:proof_kkt_d}
    \end{empheq}
\end{subequations}
Since $(u,v)\mapsto \langle v_i, M \sum_{j=1}^N u_j - b \rangle$ has full domain, the subdifferential sum rule~\citep[Cor.~16.48]{book:camo_bauschke_combettes} applied to~\eqref{eqn:proof_kkt_a} yields
\[
0 \in \partial (g_i + \iota_{\Omega_i})(u_i) + \grad_{u_i} f_i(u_i, \textstyle\sum_{j=1}^N u_j)  + M^\top v_i,\quad \forall i=1,\dots,N,
\]
which is the first component of~\eqref{eqn:aggregate_game_inclusion}. Using the definition of $N_{\mathbb{R}^p_{\geq 0}}$ \citep[Eqn.~3.10]{paper:auslender_lagrange_multiply}, \eqref{eqn:proof_kkt_b}-\eqref{eqn:proof_kkt_d} can be expressed compactly as
\[
0 \in N_{\mathbb{R}^{p}_{\geq0}}(v_i) - \big(b - M \sum_{j=1}^N u_j\big),\quad \forall i=1,\dots,N,
\]
which is the second component of~\eqref{eqn:aggregate_game_inclusion}. This completes the proof.
\end{proof}

To state the proposed algorithm, let $\mathcal{H} = (\mathbb{R}^{p}\times\mathbb{R}^d)^N$ and
\begin{subequations} \label{eqn:aggregate_operator}
\begin{equation} \label{eqn:aggregate_operator_i}
z_i = \binom{u_i}{v_i},\quad z=\begin{pmatrix} z_1 \\ \vdots \\ z_N \end{pmatrix} \in \mathcal{H},
\quad A_i(z) = \begin{pmatrix} A_{i1}(z_1) \\ \vdots \\ A_{iN}(z_N) \end{pmatrix},
\quad B_i(z) = \begin{pmatrix} B_{i1}(z_1, \sum_{j=1}^N z_j) \\ \vdots \\ B_{iN}(z_N, \sum_{j=1}^N z_j) \end{pmatrix},
\end{equation}
where the operators $A_{ij}:\mathbb{R}^{p}\times\mathbb{R}^d \rightrightarrows \mathbb{R}^{p}\times\mathbb{R}^d$ and $B_{ij}:(\mathbb{R}^{p}\times\mathbb{R}^d)^2 \to \mathbb{R}^{p}\times\mathbb{R}^d$ are defined by
\begin{align} \label{eqn:aggregate_operator_ii}
    A_{ii}(z_i) &= \begin{pmatrix} \partial (g_i + \iota_{\Omega_i})(u_i) \\ N_{\mathbb{R}^{d}_{\geq0}}(v_i) \end{pmatrix},
    & A_{ij}(z_i) &= 0\quad \forall j\neq i,
    \\
    B_{ii}(z_i, \sum_{j=1}^N u_j) &= \left.\begin{pmatrix} \grad_{u_i} f_i(u_i, s) + \grad_s f_i(u_i, s) + M^\top v_i \\ b- M s \end{pmatrix}\right|_{s=\sum_{j=1}^N u_j},& B_{ij}(z_i, \sum_{j=1}^N z_j) &= 0\quad \forall j\neq i.
\end{align}
\end{subequations}
Using this notation, we then observe that the system of inclusions~\eqref{eqn:aggregate_game_inclusion} can be compactly expressed as~\eqref{eqn:finite_sum_inclusion}. Moreover, using the chain rule, we note that $\nabla_{u_i}f_i$ in \eqref{eqn:aggregate_game_inclusion} can be expressed as
\[\grad_{u_i} f_i(u_i, \textstyle\sum_{j=1}^N u_j ) = \big(\grad_{u_i} f_i(u_i, s ) + \grad_{s} f_i(u_i, s) \big)\big|_{s=\sum_{j=1}^N u_j}.\]
Under this notation, the proposed memory-efficient version of Algorithm \ref{alg:main_algorithm} for the aggregative game~\eqref{alg:main_algorithm_boosted} is given in Algorithm~\ref{alg:main_algorithm_boosted}.
\begin{algorithm}
\caption{Memory-efficient variant of Algorithm~\ref{alg:main_algorithm} for aggregative games.} \label{alg:main_algorithm_boosted}
    Choose $\mathbf{y}^{0}, \mathbf{z}^0 \in (\mathbb{R}^p\times\mathbb{R}^d)^{N^2}$ and set $\mathbf{v}^0=B(\mathbf{y}^0)$;
    
    Compute $\mu^0_i = \sum_{j\notin\mathcal{N}(i)} y^0_{ij}$ and $\bar{y}^0_i = \mu^0_i + \sum_{j\in\mathcal{N}(i)} y^0_{ij}$ $\forall i=1,\dots,N$;
    
    Choose step sizes $\alpha_1,\dots,\alpha_N$ and set $\Lambda = \mathrm{diag}(\alpha_1,\dots,\alpha_N)$;
    
    Choose a mixing matrix $W$ and parameter $\beta < \|\Lambda^{1/2} ((I-W)/2) \Lambda^{1/2}\|^{-1}$;

    Compute $W = (\widetilde{w}_{ij}) = I-(\beta/2)\Lambda(I-W)$;
    
    Initialise variables:
        \begin{subequations}
        \begin{empheq}[left=\empheqlbrace]{align}
            \mathbf{x}^0 & = J_{\Lambda A} (\mathbf{z}^0) \label{eqn:alg2_x0}
            \\
            \mathbf{y}^1 &= 2\mathbf{x}^0 - \mathbf{z}^0 - \Lambda \mathbf{v}^0, \label{eqn:alg2_y1}
            \\
            \mathbf{z}^1 &= \mathbf{z}^0 + \mathbf{y}^1 - \mathbf{x}^0, \label{eqn:alg2_z1}
            \\
            \mathbf{x}^1 & = J_{\Lambda A} (\mathbf{z}^1) \label{eqn:alg2_x1}
            \\
            \mu^1_i & = \sum_{j\notin\mathcal{N}(i)} y^1_{ij}, &\forall i&=1,\dots, N, \label{eqn:alg2_mu_01}
            \\
            \bar{y}^1_i & = \mu^1_i  + \sum_{j\in\mathcal{N}(i)} y^1_{ij}, &\forall i&=1,\dots, N. \label{eqn:alg2_ybar_k}
        \end{empheq} 
        \end{subequations}

    \For{$k\geq1$}{
        Denote $v^k = (v^k_{ii})_{i=1}^N,z^k = (z^k_{ii})_{i=1}^N,x^k = (x^k_{ii})_{i=1}^N, \mu^k = (\mu^k_{i})_{i=1}^N\in (\mathbb{R}^p\times\mathbb{R}^d)^N.$
        
        \For{$i=1,\dots,N$}{ 
        Denote $y_i^k = (y^k_{ij})_{j\in\mathcal{N}^2(i)}\in (\mathbb{R}^p\times\mathbb{R}^d)^{|\mathcal{N}^2(i)|}$.
            \begin{subequations}
            \begin{empheq}[left=\empheqlbrace]{align}
                v_{ii}^k & = 2B_{ii}(y_{ii}^k, \bar{y}^k_i) - B_{ii}(y_{ii}^{k-1}, \bar{y}^{k-1}_i), \label{eqn:alg2_vk}
                \\
                z_{ii}^{k+1} & = z^k_{ii} - x^k_{ii} + \widetilde{w}_{ii} \big(2x^k_{ii} - x^{k-1}_{ii} - \alpha_i(v^k_{ii} - v^{k-1}_{ii})\big)+ \sum_{l\in\mathcal{N}(i)\setminus\{i\}} \widetilde{w}_{il} \big(2y^k_{li} - y^{k-1}_{li}\big), \label{eqn:alg2_zk_i}
                \\
                x_{ii}^{k+1} & = J_{\alpha_i A_{ii}} (z_{ii}^{k+1}), \label{eqn:alg2_xk}
                \\
                y^{k+1}_{ij} & = \begin{cases}
                    x^k_{ii} + z^{k+1}_{ii} - z^k_{ii} & j=i
                    \\
                    \sum_{l=1}^N \widetilde{w}_{il} \big( 2x^k_{lj} - x^{k-1}_{lj} - \alpha_l (v^k_{lj} - v^{k-1}_{lj}) \big) & j\in\mathcal{N}(i)\setminus\{i\}
                    \\
                    \sum_{l\in\mathcal{N}(i)} \widetilde{w}_{il} \big(2y^k_{lj} - y^{k-1}_{lj}\big) & j\in \mathcal{N}^2(i) \setminus\mathcal{N}(i)
                \end{cases} , \label{eqn:alg2_yk}
                \\
                \mu^{k+1}_i & =  \sum_{l\in\mathcal{N}(i)} \widetilde{w}_{il} \Big(2\mu^k_l - \mu^{k-1}_l + \sum_{\substack{j\in\mathcal{N}(l) \\j\notin\mathcal{N}(i)}} (2y^k_{lj} - y^{k-1}_{lj}) - \sum_{\substack{j\notin\mathcal{N}(l) \\j\in\mathcal{N}(i)}} (2y^k_{lj} - y^{k-1}_{lj})\Big), \label{eqn:alg2_muk}
                \\
                \bar{y}^{k+1}_i & = \mu^{k+1}_i + \sum_{j\in\mathcal{N}(i)} y^{k+1}_{ij}. \label{eqn:alg2_ybar}
            \end{empheq}
            \end{subequations}
        }
    }
\end{algorithm}

The remainder of this section is dedicated to showing that Algorithm~\ref{alg:main_algorithm_boosted} is an implementation of Algorithm~\ref{alg:main_algorithm}. We first apply Algorithm~\ref{alg:main_algorithm} to inclusion~\eqref{eqn:finite_sum_inclusion} with $A_i$ and $B_i$ defined in~\eqref{eqn:aggregate_operator}. From the initialisation~\eqref{eqn:alg1_init} we obtain~\eqref{eqn:alg2_x0}-\eqref{eqn:alg2_x1}. From the main iteration, we have the local update for agent $i=1,\dots,N$, for all $k\geq1$,
\begin{subequations}    \label{eqn:proof_alg2_naive}
        \begin{empheq}[left=\empheqlbrace, right={\quad\forall j=1,\dots,N,\,k\geq 1}]{align}
            v_{ij}^k & = 2B_{ij}(y_i^k, \textstyle\sum_{l=1}^N y_{il}^{k}) - B_{ij}(y_{ii}^{k-1}, \sum_{l=1}^N y_{il}^{k-1}),   \label{eqn:proof_alg2_v_naive}
            \\
            z_{ij}^{k+1} & = z^k_{ij} - x^k_{ij} + \sum_{l=1}^N \widetilde{w}_{il} \big( 2x^k_{lj} - x^{k-1}_{lj} - \alpha_l (v^k_{lj} - v^{k-1}_{lj}) \big),   \label{eqn:proof_alg2_z_naive}
            \\
            x_{ij}^{k+1} & = J_{\alpha_i A_{ij}} (z_{ij}^{k+1}),   \label{eqn:proof_alg2_x_naive}
            \\
            y^{k+1}_{ij} & = x^k_{ij} + z^{k+1}_{ij} - z^k_{ij}.\label{eqn:proof_alg2_y_naive}
        \end{empheq}
\end{subequations}
From~\eqref{eqn:proof_alg2_x_naive} and the definition of $A_i$~\eqref{eqn:aggregate_operator}, we obtain~\eqref{eqn:alg2_xk} \citep[Ex.~23.3,~23.4]{book:camo_bauschke_combettes} and deduce that $x^k_{ij} = z^k_{ij}$ for $j\neq i$. From~\eqref{eqn:proof_alg2_y_naive}, we obtain the first equation in~\eqref{eqn:alg2_yk}, $y^k_{ii} = x^k_{ii} + z^{k+1}_{ii} - z^k_{ii}$, and deduce that $y^k_{ij}=z^k_{ij}$ for $j \neq i$. Combining~\eqref{eqn:proof_alg2_z_naive} with Remark \ref{rem:mixing_matrix_local} and the fact that, for $j \neq i$, $x^k_{ij} = z^k_{ij} = y^k_{ij}$ and $v^k_{ij}=0$, we obtain~\eqref{eqn:alg2_zk_i}:
\begin{align*}
z_{ii}^{k+1} & = z^k_{ii} - x^k_{ii} + \sum_{l\in\mathcal{N}(i)} \widetilde{w}_{il} \big( 2x^k_{lj} - x^{k-1}_{lj} - \alpha_l (v^k_{lj} - v^{k-1}_{lj}) \big)
\\
&= z^k_{ii} - x^k_{ii} + \widetilde{w}_{ii} \big(2x^k_{ii} - x^{k-1}_{ii} - \alpha_l (v^k_{ii} - v^{k-1}_{ii}) \big) + \sum_{l\in\mathcal{N}(i)\setminus\{i\}} \widetilde{w}_{il} \big( 2y^k_{li} - y^{k-1}_{li}\big).
\end{align*}
Using a similar logic, we obtain the second and third equations in~\eqref{eqn:alg2_yk}:
\begin{equation} \label{eqn:proof_alg2_yk_ij}
y^{k+1}_{ij} = z^{k+1}_{ij} = \begin{cases}
    \displaystyle\sum_{l\in\mathcal{N}(i)} \widetilde{w}_{il}(2x^k_{lj} - x^{k-1}_{lj} - \alpha_j(v^k_{lj}-v^{k-1}_{lj})), & j\in\mathcal{N}(i)\setminus\{i\},
    \\
    \displaystyle\sum_{l\in\mathcal{N}(i)} \widetilde{w}_{il} \big(2y^k_{lj} - y^{k-1}_{lj}\big), & j\in\{1,\dots,N\}\setminus\mathcal{N}(i).
\end{cases}
\end{equation}
In the statement of Algorithm~\ref{alg:main_algorithm_boosted}, we write~\eqref{eqn:alg2_yk} in the form of~\eqref{eqn:proof_alg2_yk_ij} for compactness. However, note that for $j\in\mathcal{N}(i)$,~\eqref{eqn:proof_alg2_yk_ij} can be expressed as
\[y^k_{ij} = \widetilde{w}_{ij}(2x^k_{jj} - x^{k-1}_{jj} - \alpha_j(v^k_{jj}-v^{k-1}_{jj}) + \displaystyle\sum_{\substack{l\in\mathcal{N}(i)\\l\neq j}} \widetilde{w}_{il} \big(2y^k_{lj} - y^{k-1}_{lj}\big),\quad j\in\mathcal{N}(i)\setminus\{i\}.\]

Furthermore, observe that the variables $y^k_{ij}$ for $j\notin\mathcal{N}(i)$ are only used to evaluate $B_i$, and then only as an aggregate $\sum_{j=1}^N y^k_{ij}$, following from the definition of $B_i$~\eqref{eqn:aggregate_operator}. To exploit this observation, we define the variables $\mu^k_i\in\mathbb{R}^p$ and $\bar{y}^k_i \in\mathbb{R}^p$ by
\begin{equation} \label{eqn:proof_aggregate_yk}
    \mu_i^k = \sum_{j\notin\mathcal{N}(i)} y_{ij}^k, \quad \bar{y}^k_i = \sum_{j=1}^N y^k_{ij} = \mu_i^k + \sum_{j\in\mathcal{N}(i)} y^k_{ij},
\end{equation}
from which we obtain~\eqref{eqn:alg2_mu_01} and~\eqref{eqn:alg2_ybar}. To derive an update for $\mu_i^k$, first observe that
\begin{equation}    \label{eqn:summation_trick}
    \sum_{j\notin\mathcal{N}(i)} y_{lj}^k = \sum_{\substack{j\in\mathcal{N}(l) \\j\notin\mathcal{N}(i)}} y_{lj} + \sum_{\substack{j\notin\mathcal{N}(l)\\j\notin\mathcal{N}(i)}} y_{lj} = \sum_{\substack{j\in\mathcal{N}(l) \\j\notin\mathcal{N}(i)}} y_{lj}^k + \mu^k_l - \sum_{\substack{j\notin\mathcal{N}(l) \\j\in\mathcal{N}(i)}} y_{lj}^k,\quad \forall l=1,\dots,N.
\end{equation}
Then combining~\eqref{eqn:summation_trick} with the definition of $\mu^k_i$ and~\eqref{eqn:proof_alg2_yk_ij} summed over $j\notin\mathcal{N}(i)$, we obtain~\eqref{eqn:alg2_muk}:
\begin{align}
    \mu^{k+1}_i &= \sum_{l\in\mathcal{N}(i)} \widetilde{w}_{il} \sum_{j\notin\mathcal{N}(i)} \big(2y^k_{lj} - y^{k-1}_{lj}\big), \nonumber
    \\
    & = \sum_{l\in\mathcal{N}(i)} \widetilde{w}_{il} \big(2\mu^k_l - \mu^{k-1}_l + \sum_{\substack{j\in\mathcal{N}(l) \\j\notin\mathcal{N}(i)}} (2y^k_{lj} - y^{k-1}_{lj}) - \sum_{\substack{j\notin\mathcal{N}(l) \\j\in\mathcal{N}(i)}} (2y^k_{lj} - y^{k-1}_{lj})\big). \label{eqn:proof_muk}
\end{align}
Note that~\eqref{eqn:proof_muk} can be computed using only the variables $y^k_{lj}$ for $l\in\mathcal{N}(i)$ and $j\in\mathcal{N}^2(l)$. To see this, observe that $l\in\mathcal{N}(i)$ implies $i\in\mathcal{N}(l)$, and thus $j\in\mathcal{N}(i)$ implies that $j \in \mathcal{N}^2(l)$.

Finally, by combining~\eqref{eqn:proof_alg2_v_naive},~\eqref{eqn:proof_aggregate_yk}, and the definition of $B_i$~\eqref{eqn:aggregate_operator}, we obtain~\eqref{eqn:alg2_vk}, completing the derivation.

\begin{remark}[Memory efficiency of Algorithm~\ref{alg:main_algorithm_boosted}]
Some comments are in order regarding the memory required to implement Algorithm~\ref{alg:main_algorithm_boosted}, as compared to Algorithm~\eqref{alg:main_algorithm} applied to~\eqref{eqn:aggregate_game}.
\begin{enumerate}[(i)]
    \item The na\"ive implementation of Algorithm~\ref{alg:main_algorithm} in~\eqref{eqn:proof_alg2_naive} updates local variables $v^k_{ij}$, $x^k_{ij}$, $y^k_{ij}$, and $z^k_{ij}$ for all $i, j=1,\dots,N$. Thus, each agent updates four variables in $(\mathbb{R}^p \times \mathbb{R}^d)^{N}$ in each iteration. \label{implementation_i}

    \item The memory requirement described in \ref{implementation_i} can be reduced slightly with a more careful implementation of Algorithm~\ref{alg:main_algorithm}. Indeed, following~\eqref{eqn:proof_alg2_naive}, we deduced that $x^k_{ij}=z^k_{ij}=y^k_{ij}$ and $v^k_{ij}=0$ for all $i=1,\dots,N$ and $j\neq i$. Using this observation, it suffices to only update $v^k_{ii}$, $x^k_{ii}$, and $z^k_{ii}$ for all $i=1,\dots,N$, and $y^k_{ij}$ for all $i,j=1,\dots,N$. Thus, each agent updates three variables in $\mathbb{R}^p \times \mathbb{R}^d$ and one variable in $(\mathbb{R}^p \times \mathbb{R}^d)^{N}$ in each iteration. 
    \label{implementation_ii}

    \item Algorithm~\ref{alg:main_algorithm_boosted} updates local variables $v^k_{ii}$, $x^k_{ii}$, $z^k_{ii}$, and $\mu^k_i$ for all $i=1,\dots,N$, and $y^k_{ij}$ for all $j\in\mathcal{N}^2(i)$ and $i=1,\dots,N$. Thus, each agent updates four variables in $\mathbb{R}^p \times \mathbb{R}^d$ and one variable in $(\mathbb{R}^p \times \mathbb{R}^d)^{|\mathcal{N}^2(i)|}$.    
\end{enumerate}

Since $N+3 \leq 4N$ for $N\geq 1$, implementation~\ref{implementation_ii} requires fewer variables in $\mathbb{R}^p \times \mathbb{R}^d$ than~\ref{implementation_i}, and thus can be more memory efficient. Algorithm~\ref{alg:main_algorithm_boosted} requires fewer variables in $\mathbb{R}^p \times \mathbb{R}^d$ than both~\ref{implementation_i} and~\ref{implementation_ii} when the communication graph satisfies
\begin{equation}    \label{eqn:boosted_condition}
    \sum_{i=1}^N \big|\mathcal{N}^2(i)\big| < N - 1.
\end{equation}
\end{remark}

\section{Numerical experiment: virtual power plant} \label{sec:vpp_model}
In this section, we numerically compare \citep[Alg.~1]{paper:malitsky_tam_minmax}, as presented in~\eqref{alg:pdtr}, and our proposed Algorithms~\ref{alg:main_algorithm} and \ref{alg:main_algorithm_boosted} in coordinating a \emph{virtual power plant (VPP)} as an aggregative game. Specifically, we consider a VPP consisting of $N$ noncooperative players, each operating a power bank, which are connected via the local power grid. For further details on VPPs, see \citep[\S3E]{paper:wang_vpp_review}. In each time period $t=1,\dots,p$, players must decide how much electricity to buy (charging their power banks) and how much to sell (discharging their power banks) to satisfy their personal electricity usage while minimising their total cost over the time horizon. This model extends the electric vehicle charging game presented by \citep[\S6]{paper:paccagnan_nash_wardrop}, which allows charging but not discharging. We seek a vector that determines when and how much players charge and discharge their power banks such that players have no incentive to individually deviate from this schedule. 
\subsection{Problem formulation}
We represent a charging schedule for the player population using $u=(u_1,\dots,u_N)^\top\in\mathbb{R}^{2p}\times\dots\times\mathbb{R}^{2p}$, where the charging ($u_i^+$) and discharging ($u_i^-$) decisions for player $i$ are denoted
\[u_i = \begin{pmatrix} u_i^+ \\ u_i^- \end{pmatrix}\in\mathbb{R}^{2p},\quad u_i^+ = \begin{pmatrix} u_i^+(1)\\ \vdots \\ u_i^+(p) \end{pmatrix}\in\mathbb{R}^{p},\quad u_i^- = \begin{pmatrix} u_i^-(1)\\ \vdots \\ u_i^-(p) \end{pmatrix}\in\mathbb{R}^{p}\]
where $u_i^+(t)\in\mathbb{R}$ (resp.\ $u_i^-(t)$) denotes the amount of electricity drawn from the grid (resp.\ discharged to the grid) by player $i=1,\dots,N$ in time period $t=1,\dots,p$. Under this notation, we formulate the VPP coordination problem as an aggregative game, where the problem for player $i$ is given by
\begin{subequations}    \label{eqn:vpp_game}
\begin{empheq}[left=\empheqlbrace]{align}
    \displaystyle\min_{u_i\in\mathbb{R}^{2p}}\, & u_i^\top Q_i u_i + p_i^\top u_i + \left\langle \sum_{j=1}^N u_j^+ - u_j^- + m,\, u_i^+ - u_i^- \right\rangle, \label{eqn:vpp_objective}
    \\
    \text{s.t.} \quad & 0 \leq u^+_i(t) \leq \overline{u}^+_{i}(t),& \forall t=1,\dots,p,& \label{eqn:vpp_charge_in}
    \\
    & 0 \leq u^-_i(t) \leq \overline{u}^-_{i}(t),& \forall t=1,\dots,p,& \label{eqn:vpp_charge_out}
    \\
    & l^\text{low}_{i}(t) \leq \sum_{s=1}^t \left( e_i^+ u^+_i(s) - \frac{1}{e_i^-} u^-_i(s) \right) \leq l^\text{up}_{i}(t),& \forall t=1,\dots,p, &\label{eqn:vpp_charge_state_constr}
    \\
    & -m(t) \leq \sum_{j=1}^N \left( u^+_j(t) - u^-_j(t) \right) \leq K(t) - m(t). & \forall t=1,\dots,p. &\label{eqn:vpp_constr}
\end{empheq}
\end{subequations}
The objective~\eqref{eqn:vpp_objective} is the sum of two functions. The first is a quadratic function modelling the monetary cost associated with power bank capacity degradation as a function of charging and discharging, where $Q\in\mathbb{R}^{2p\times 2p}$ is diagonal and $p\in\mathbb{R}^{2p}_{\geq0}$. For further details, see \citep[\S~2.3]{paper:ma_battery_degradation}. The second function is the objective is a bilinear function modelling the net cost of electricity for a given schedule, where $m=(m(t))\in \mathbb{R}^p$ is the non-VPP electricity demand in each time period. For further details, see \citep[\S VB(2)]{paper:belgioioso_ev_charge}. Constraints~\eqref{eqn:vpp_charge_in} and~\eqref{eqn:vpp_charge_out} enforce maximum throughputs on electricity drawn from the grid in each time period, $\overline{u}^+_i(t)$, and discharged to the grid in each time period, $\overline{u}^-_i(t)$. Constraints~\eqref{eqn:vpp_charge_state_constr} enforce power banks' minimum state of charge, $l^\text{low}_i(t)\geq0$, and maximum state of charge, $ l^\text{up}_i(t)\geq0$ in each time period, where the initial state of charge is 0. Each power bank has a charging efficiency $e_i^+ \in (0,1)$ and discharging efficiency $e_i^- \in (0,1)$. Specifically, if power bank $i$ desires $e_i^+ u_i^+$ then the power grid must export $u_i^+$, and if the power bank exports $(1/e_i^-)u_i^-$ then the grid receives $u_i^-$. This formulation is necessary to represent~\eqref{eqn:vpp_charge_state_constr} as an aggregate of $u_i$. Constraints~\eqref{eqn:vpp_constr} ensure that the total amount of power drawn from the grid by all players in each time period does not exceed the grid capacity, $K(t)$, and that the net power discharged by players in each time period does not exceed non-VPP demand.

To apply Algorithms \ref{alg:main_algorithm} and \ref{alg:main_algorithm_boosted} to the VPP game, observe that~\eqref{eqn:vpp_game} can be expressed as~\eqref{eqn:aggregate_game} by setting
\begin{subequations} \label{eqn:vpp_game_notation}
\begin{equation}
    g_i(u_i) = u_i^\top Q_i u_i + p_i^\top u_i, \quad f_i\left(u_i, \textstyle\sum_{j=1}^N u_j\right) = \left\langle \begin{bmatrix} I & -I \end{bmatrix} u_i,\, \begin{bmatrix} I & -I \end{bmatrix} \sum_{j=1}^N u_j + m  \right\rangle,
\end{equation}
\begin{equation}
    \Omega_i = \left\{u_i \in\mathbb{R}^{2p}\,\Big|\, 0 \leq u_i \leq \begin{pmatrix} \bar{u}^+_i \\ \bar{u}^-_i \end{pmatrix},\,  l^\text{low}_i \leq \begin{bmatrix} e^+_i R & - \frac{1}{e_i^-} R \end{bmatrix} u_i \leq l^\text{up}_i \right\},
\end{equation}
\begin{equation}
M = \begin{bmatrix} I & -I \\ -I & I \end{bmatrix} \in \mathbb{R}^{2p\times 2p},\quad b = \begin{pmatrix} K-m \\ m\end{pmatrix} \in \mathbb{R}^{2p},
\end{equation}
\end{subequations}
where $R=(r_{ij})\in\mathbb{R}^{p\times p}$ is a lower triangular matrix, with $r_{ij} = 1$ for $i \geq j$ and $0$ otherwise. Note that $\grad_{u_i} f_i$ is $2\sqrt2$-Lipschitz continuous, and that Assumption~\ref{assum:slaters} holds for~\eqref{eqn:aggregate_game} for an appropriate choice of parameters (see also \citep[\S VA]{paper:belgioioso_ev_charge}).

\subsection{Results}
The numerical experiments use $N=20$, $40$, $60$, $80$, and $100$, and several different communication graphs (barbell, cycle, and 2D grid graphs). Note that the barbell, cycle, and grid graphs satisfy condition~\eqref{eqn:boosted_condition} for $N\geq3$, and thus Algorithm~\ref{alg:main_algorithm_boosted} uses fewer variables than Algorithm~\ref{alg:main_algorithm}. For brevity, we display the results using a cycle graph in this section and include the remainder in Appendix~\ref{sec:appendix}. For the mixing matrix $W$, we used the Laplacian-based constant edge-weight matrix~\eqref{defn:laplacian_mixing} with $\tau=0.505\lambda_\text{max}(\mathcal{L})$, so as to satisfy~\eqref{eqn:mixing_matrix_condition}.

All experiments were implemented in Python 3.12.7 using \texttt{NumPy}~v1.26.4 on a virtual machine running Windows 10 with 16GB RAM and using one core of an AMD EPYC 7763 CPU @ 2.45 GHz. Algorithm~\ref{alg:main_algorithm} and~\eqref{alg:pdtr} were implemented using NumPy~v1.24.3 and used the default QP solver from CVXOPT~v1.3.2 to compute the resolvents $J_{A_i}$. Algorithm~\ref{alg:main_algorithm_boosted} was implemented using SciPy \texttt{sparse}~v1.10.1, and Numba~v0.57.0 for updating $\mu^k_i$~\eqref{eqn:alg2_muk}. 

We run all algorithms for 1000 iterations and compare them across two metrics at the final iteration:
\begin{enumerate}[(i)] 
    \item The \emph{normalised residuals} (\emph{i.e.,} normalised by step size) given by
 \[\text{normalised residual} = \|\mathbf{z}^k - \mathbf{z}^{k-1}\|_{\Lambda^{-1}},\]
 where $\mathbf{z}^k$ is defined in~\eqref{eqn:alg1_zk} for Algorithm \ref{alg:main_algorithm} and~\eqref{alg:pdtr_zk} for~\eqref{alg:pdtr}. The stepsize matrix $\Lambda$ is taken as $\Lambda=\mathrm{diag}(\alpha_1, \dots, \alpha_N)$ for Algorithm \ref{alg:main_algorithm} and $\Lambda=\alpha I$ for~\eqref{alg:pdtr} with $\alpha=\max_i\{\alpha_i\}$,

    \item The total \emph{time} taken in seconds.
\end{enumerate}
Recall that $L_i$ denotes the Lipschitz constant of the operator $B_i$. Algorithms~\ref{alg:main_algorithm} and~\ref{alg:main_algorithm_boosted} are tested using heterogeneous step sizes $\alpha_i = 0.9/(8L_i)$ and parameter $\beta = 0.9\|\Lambda^{1/2} ((I-W)/2) \Lambda^{1/2}\|^{-1}$, as this combination gave the best performance in Section \ref{sec:numerics_part1}. Algorithm~\eqref{alg:pdtr} is tested using the step size $\alpha~=~0.9(1+\lambda_{\text{min}}(W))/(4\max_i\{L_i\})$.

\begin{figure}[ht!]
    \centering
    \begin{subfigure}[t]{0.45\textwidth}
        \centering
        \includegraphics[width=\linewidth]{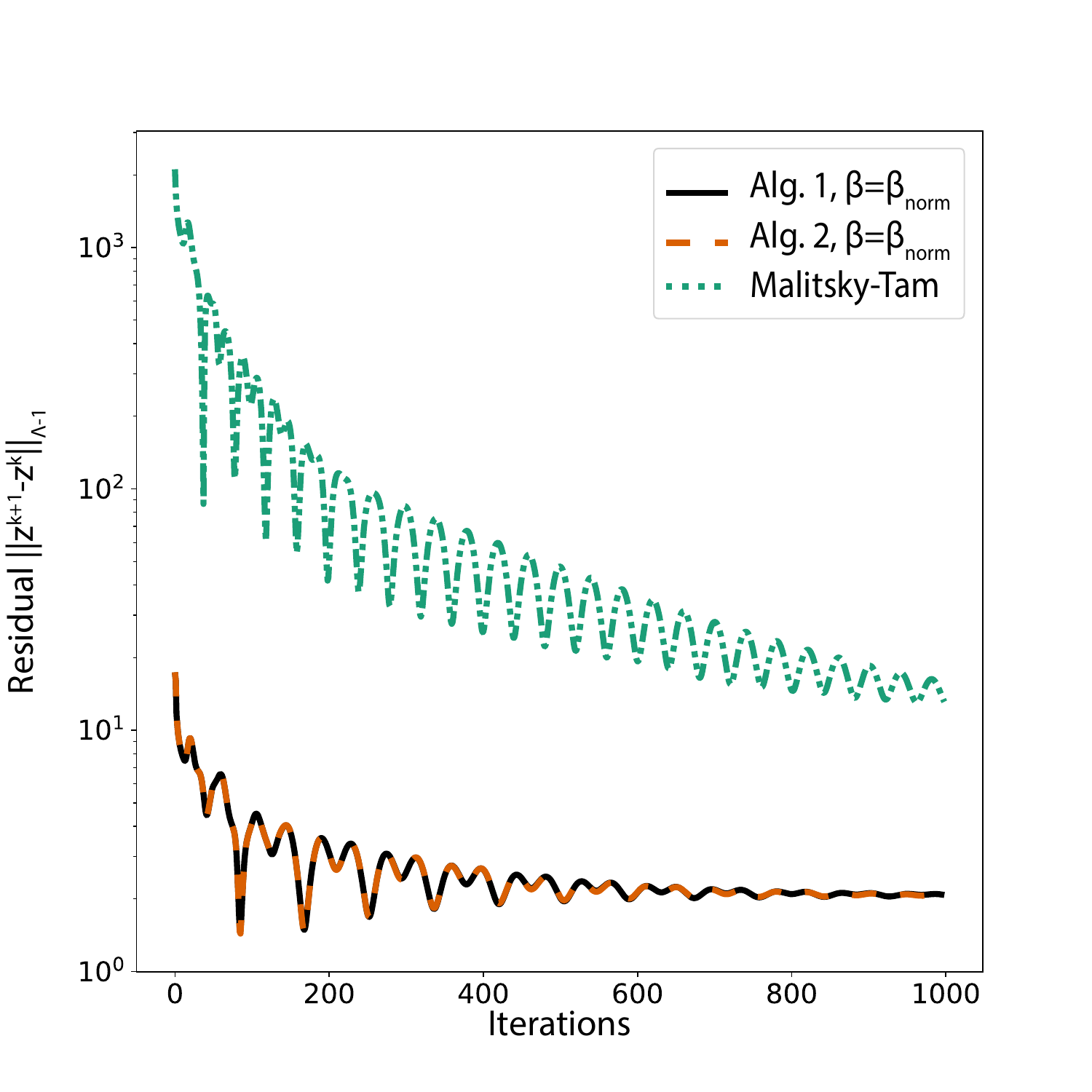}
        \caption{Normalised residuals vs. iterations.}
        \label{fig:vpp_resid}
    \end{subfigure}
    \begin{subfigure}[t]{0.45\textwidth}
        \centering
        \includegraphics[width=\textwidth]{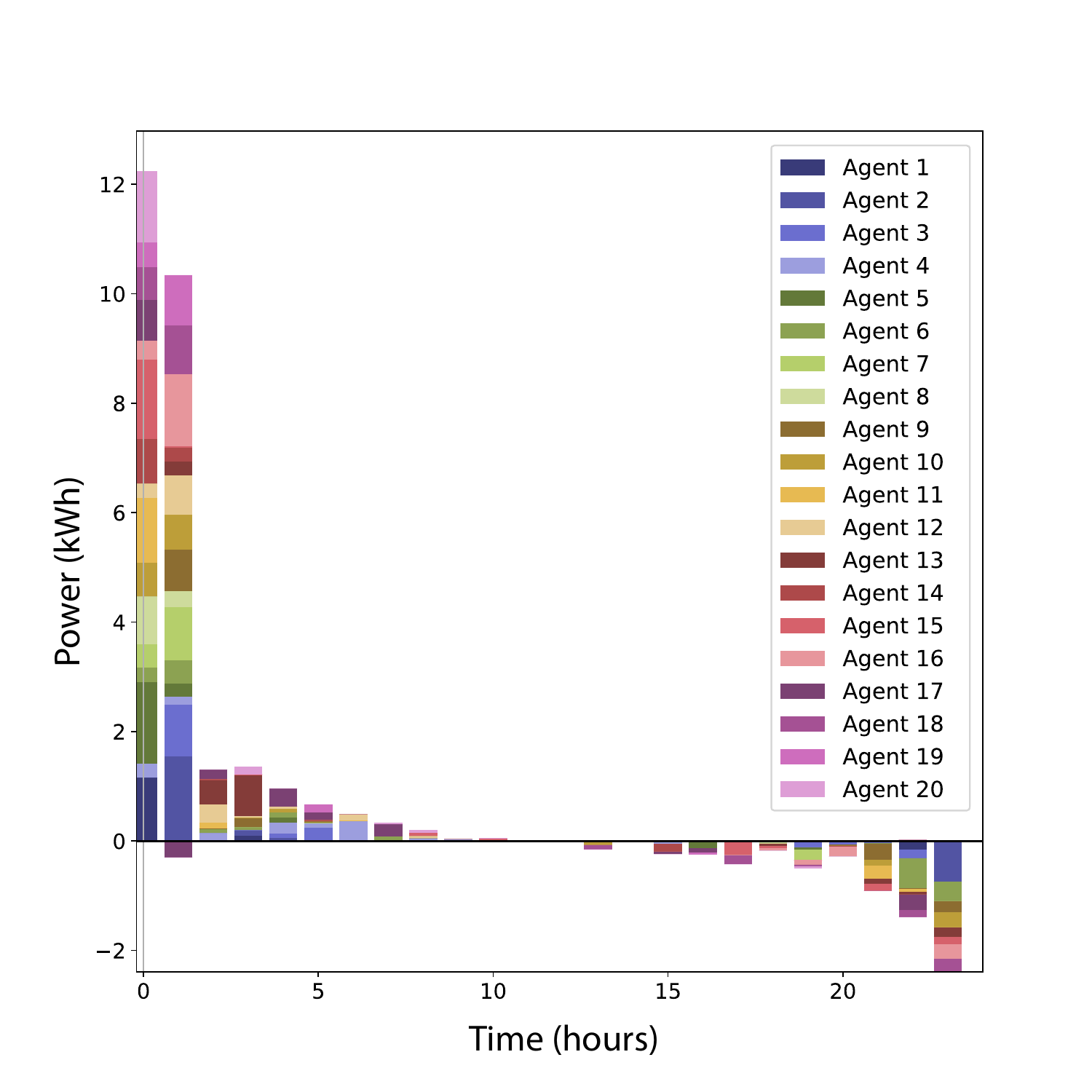}
        \caption{VPP schedule from Algorithm \ref{alg:main_algorithm}. Positive values indicate charging, negative values indicate discharging.}
        \label{fig:vpp_schedule}
    \end{subfigure}
    \caption{Sample results for the VPP problem using $N=20$ players and a cycle communication graph. Residuals are averaged over five runs with different initial points. Step sizes are given at the start of Section~\ref{sec:vpp_model}.}
    \label{fig:vpp}
\end{figure}

\begin{table}[ht!]
\centering
\setlength{\tabcolsep}{0.25em}
\begin{tabular}{c|rrrr|rrrr|rrrr}
 & \multicolumn{4}{c}{Alg.~\ref{alg:main_algorithm}} & \multicolumn{4}{c}{Alg.~\ref{alg:main_algorithm_boosted}} & \multicolumn{4}{c}{ \citep[Alg.~1]{paper:malitsky_tam_minmax}} \\
\hline
$N$ & \multicolumn{2}{c}{Residual} & \multicolumn{2}{c|}{Time (s)} & \multicolumn{2}{c}{Residual} & \multicolumn{2}{c|}{Time (s)} & \multicolumn{2}{c}{Residual} & \multicolumn{2}{c}{Time (s)} \\
20 &
2.08 & (2.08) & 251.03 & (252.68) &
\textbf{2.07} & (2.07) & 292.18 & (295.30) &
13.12 & (13.19) & \textbf{198.41} & (199.57)
\\
40 & 
\textbf{3.46} & (3.46) & 506.45 & (508.85) & 
3.54 & (3.54) & 599.51 & (605.39) &
66.50 & (66.61) & \textbf{399.72} & (412.35)
\\
60 &
\textbf{4.48} & (4.48) & 783.31 & (786.02) &
4.63 & (4.63) & 957.28 & (965.16) &
119.04 & (119.14) & \textbf{630.89} & (634.75)
\\
80 &
\textbf{5.49} & (5.49) & 1073.97 & (1078.89) &
5.72 & (5.72) & 1340.29 & (1343.13) &
175.92 & (175.97) & \textbf{881.37} & (882.53)
\\
100 &
15.74 & (15.74) & 1391.35 & (1410.60) &
\textbf{15.39} & (15.39) & 1767.72 & (1773.58) &
130.14 & (130.15) & \textbf{1182.00} & (1193.32)
\end{tabular}
\caption{VPP results using a cycle graph. Results are the average (and worst-case) over five repetitions, rounded to 2 decimal places. Bolded results indicate best performance amongst algorithms for a given $N$ and metric.}
\label{tab:vpp_results_cycle}
\end{table}
The parameters in \eqref{eqn:vpp_game} are generated analogously to \citet*{paper:belgioioso_ev_charge}. The diagonal entries of $Q_i$ are sampled from $U(0.1, 4)$ and are $0$ otherwise, and the entries of $p_i$ are sampled from $U(0.2,2)$. The entries of $c_i^+$ (resp. $c_i^-$) are sampled from $U(0,2)$ (resp. $U(-2,0)$) The entries of $\overline{u}_{i}^+$ and $\overline{u}_{i}^-$ are sampled from $U(1,5)$ with probability $0.8$ and are 0 otherwise. The efficiencies $e_i^\pm$ are sampled from $U(0.5,1)$, and the entries of $l_{i}^\text{low}$ (resp. $l_{i}^\text{up}$) are sampled from $U(0,1)$ (resp. $U(1,3)$). The non-VPP demand $m(t)$ is as in~\citep{paper:ma_nonev_demand}, and the grid capacity is $K(t)=0.55+\max_t m(t)$. The results are averaged over five runs using $\mathbf{y}^0 = \mathbf{0}$ and a random initial point $\mathbf{z}^0$, with entries sampled from $U(0,1)$ and scaled such that $\|\mathbf{z}^0\|=10$.

Table~\ref{tab:vpp_results_cycle} shows that Algorithms~\ref{alg:main_algorithm} and~\ref{alg:main_algorithm_boosted} outperform \eqref{alg:pdtr} in terms of residuals, with a sample convergence plot shown in Figure~\ref{fig:vpp_resid}. Algorithms~\ref{alg:main_algorithm} and~\ref{alg:main_algorithm_boosted} both take more time to complete 1000 iterations than~\eqref{alg:pdtr}, approximately 20\% and 50\% longer, respectively. These observations hold across all instances and graph types tested.

\section{Conclusion}    \label{sec:conclusion}
We proposed a decentralised forward-backward type algorithm (Algorithm~\ref{alg:main_algorithm}) for finding a zero of a finite sum of (potentially) set-valued operators single-valued operators. This algorithm improves on several aspects of the PG-EXTRA algorithm: only requiring agents' single-valued operators to be Lipschitz continuous, allowing agents to use heterogeneous step sizes, and only requiring agents to use local information in determining these step sizes. We also proposed a memory-efficient variant of this algorithm for aggregative games, which uses fewer variables than previous approaches (see, for example, \citep[Alg.~1]{paper:tatarenko_nedic}). Numerical experiments on robust least squares, zero-sum games, and coordinating a virtual power plan illustrated the performance of our proposed algorithms and validated the convergence results. Future directions include improving the step size bound for Algorithm~\ref{alg:main_algorithm}, which arises from the backward-forward-reflected-backward algorithm \citep{paper:rieger_tam_noncocoercive}, allowing variable stepsizes to permit locally Lipschitz continuous operators $B_i$ rather than globally Lipschitz continuous, and extending Algorithm~\ref{alg:main_algorithm} to time-varying mixing matrices for communication graphs that evolve through time.

\section{Acknowledgements}
MKT is supported in part by Australian Research Council grant DP23010174.
LT is supported by a Melbourne Research Scholarship. This research was supported by The University of Melbourne’s Research Computing Services and the Petascale Campus.

\bibliography{references}

\appendix
\setcounter{figure}{0}
\setcounter{table}{0}
\renewcommand{\thefigure}{A.\arabic{figure}}
\renewcommand{\thetable}{A.\arabic{table}}
\setlength{\tabcolsep}{0.25em}
\section{Additional numerical results} \label{sec:appendix}
% FIGURE RLS synthetic
\begin{figure}[ht!]
    \begin{subfigure}{0.45\textwidth}
    \centering
    \includegraphics[width=\textwidth]{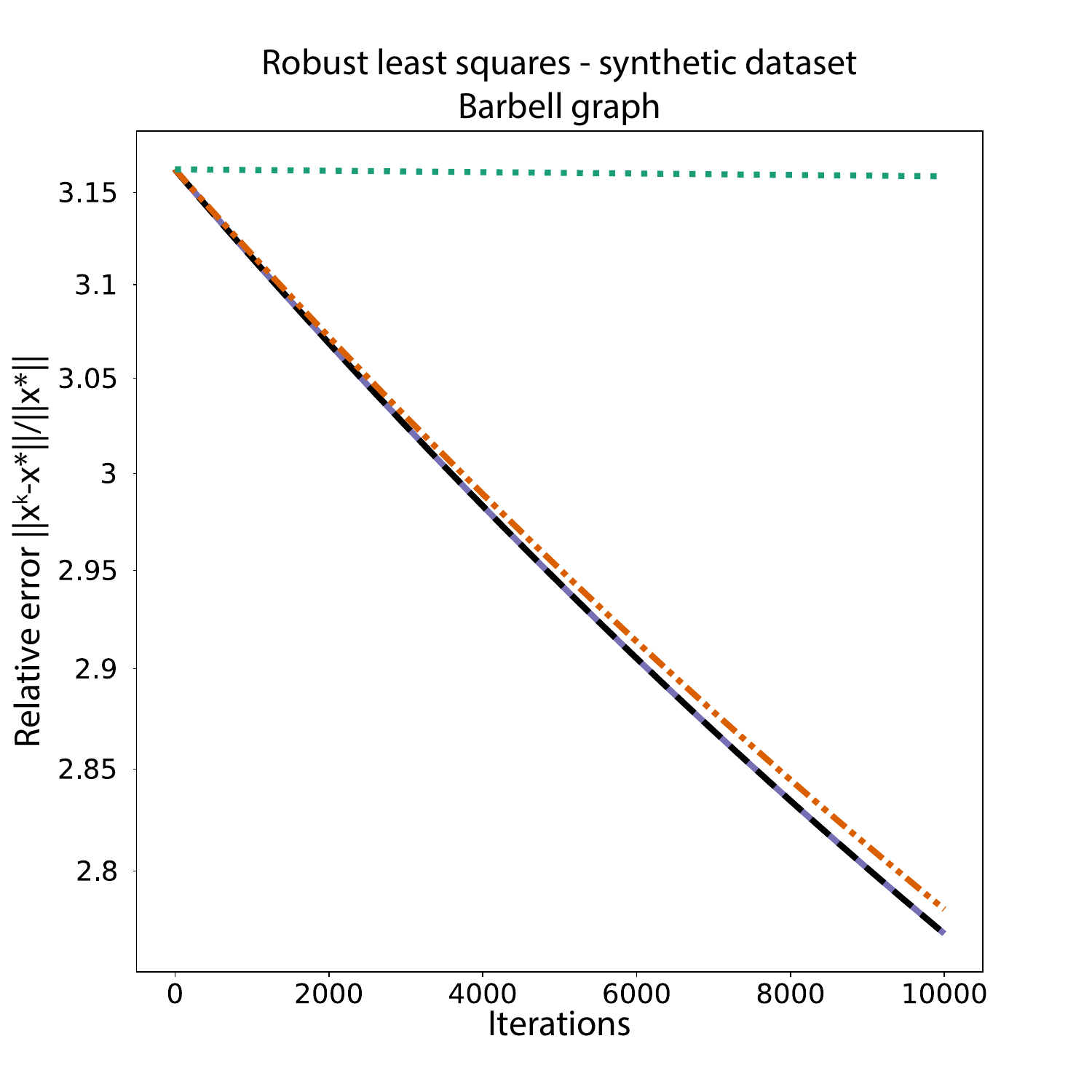}
    \caption{Barbell graph: relative error vs. iterations.}
    \label{fig:rls_synthetic_relerror_barbell}
    \end{subfigure}
    \hfill
    \begin{subfigure}{0.45\textwidth}
    \centering
    \includegraphics[width=\textwidth]{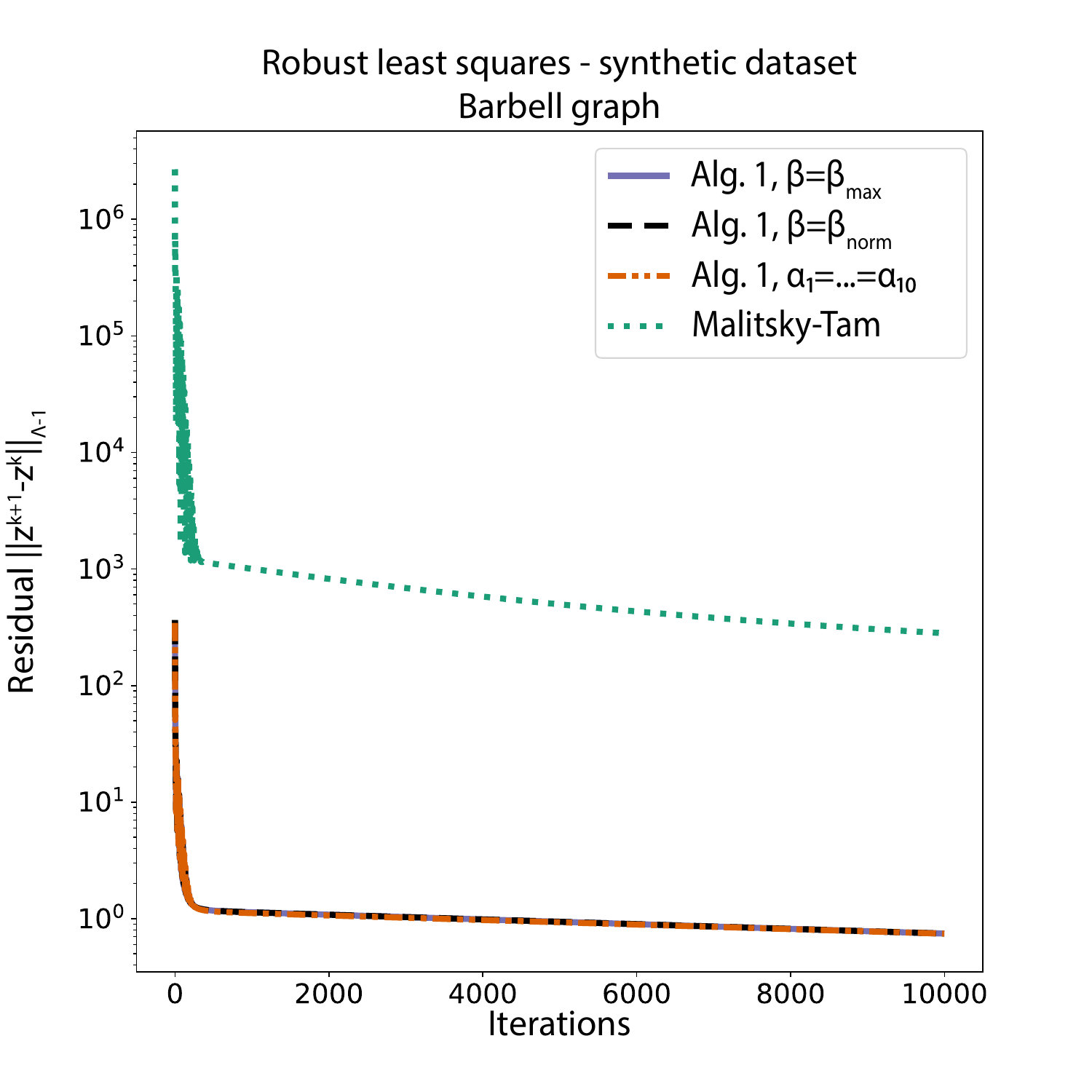}
    \caption{Barbell graph: residuals vs. iterations.}
    \label{fig:rls_synthetic_resid_barbell}
    \end{subfigure}

    \begin{subfigure}{0.45\textwidth}
    \centering
    \includegraphics[width=\textwidth]{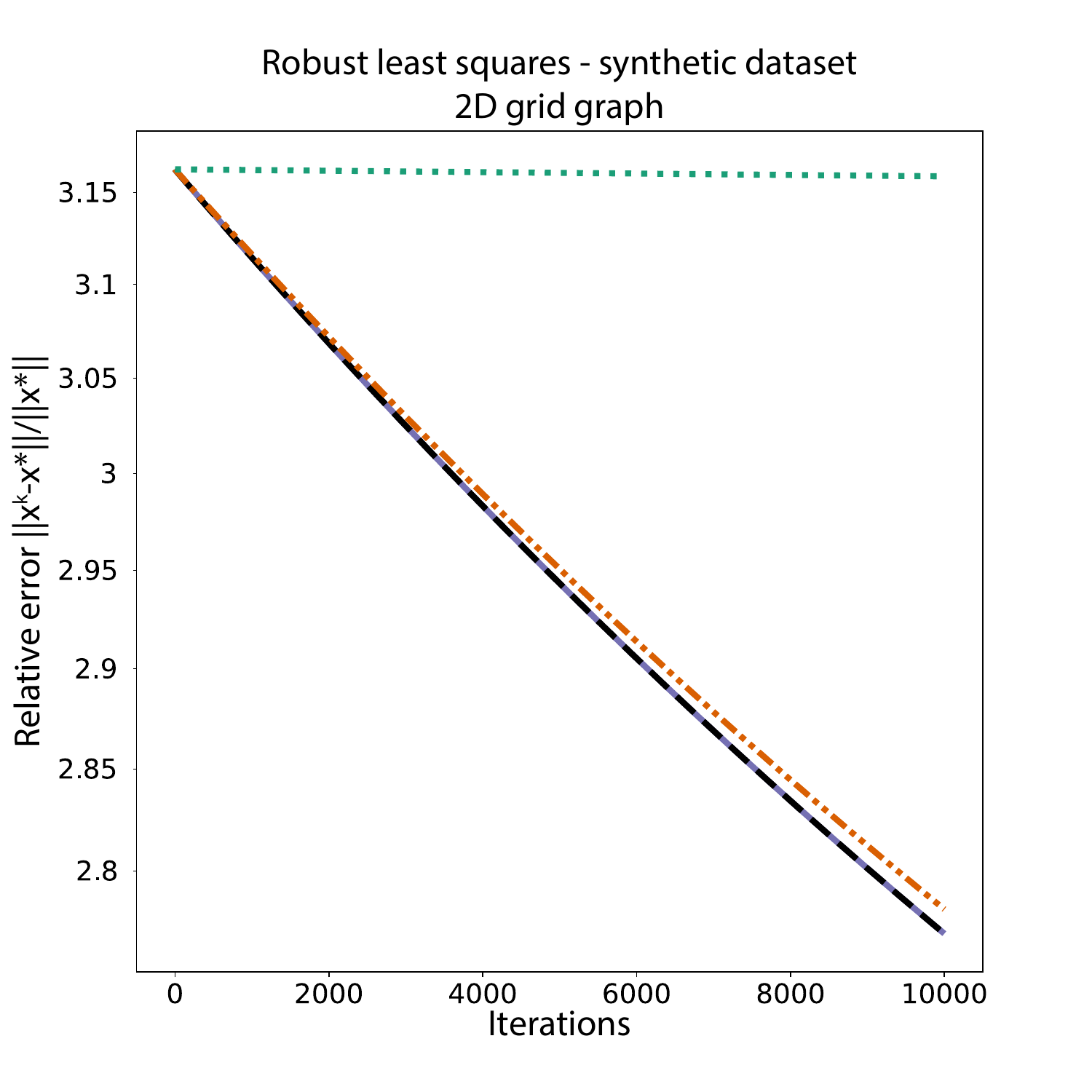}
    \caption{2D grid graph: relative error vs. iterations.}
    \label{fig:rls_synthetic_relerror_grid2D}
    \end{subfigure}
    \hfill
    \begin{subfigure}{0.45\textwidth}
    \centering
    \includegraphics[width=\textwidth]{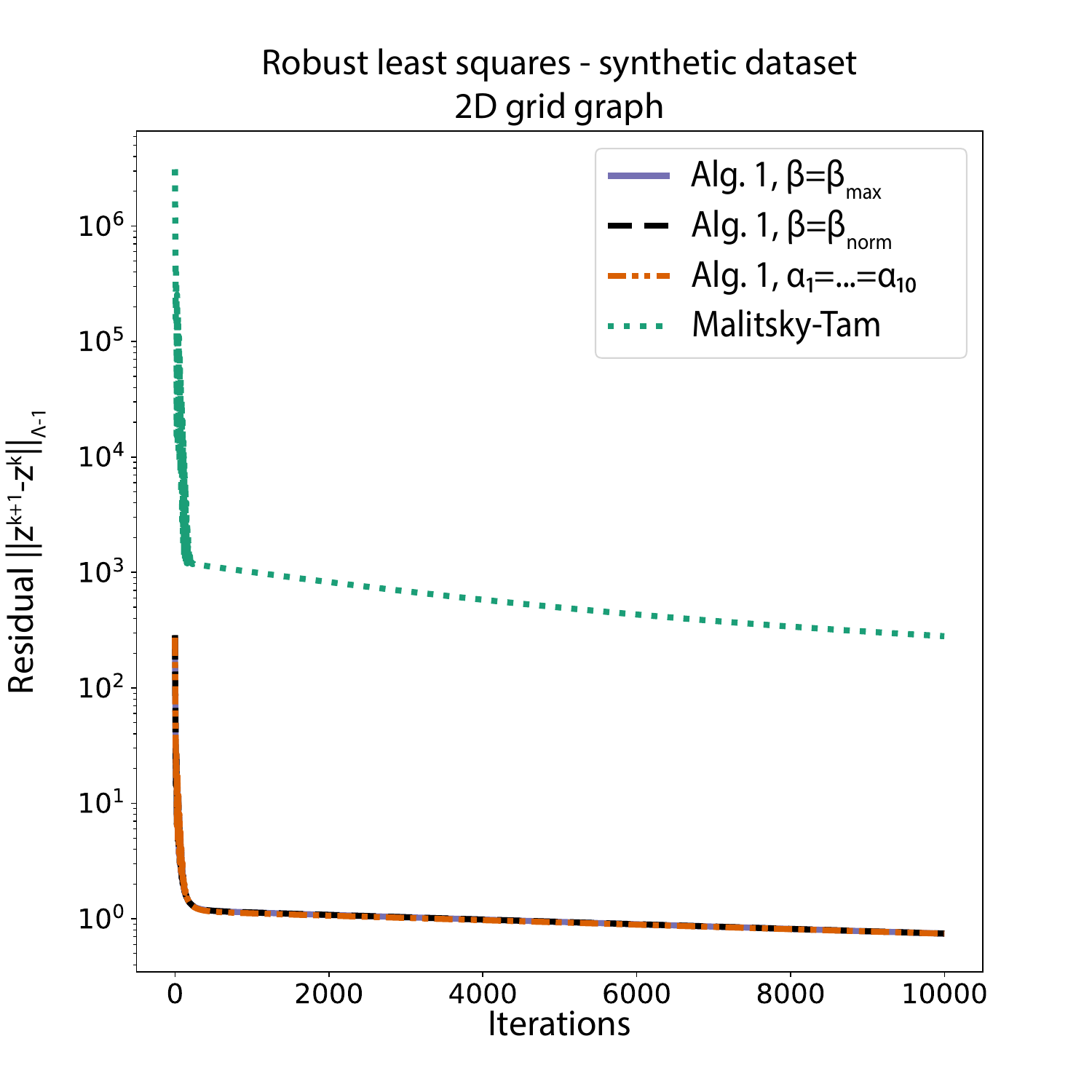}
    \caption{2D grid graph: residuals vs. iterations.}
    \label{fig:rls_synthetic_resid_grid2D}
    \end{subfigure}
    
    \caption{Robust least squares results using the synthetic dataset and a barbell (top row) and 2D grid (bottom row) communication graphs. Step sizes are given at the start of Section \ref{sec:numerics_part1}.}
    \label{supp:fig:rls_synthetic}
\end{figure}

% FIGURE RLS California Housing
\begin{figure}[ht!]
    \begin{subfigure}{0.45\textwidth}
    \centering
    \includegraphics[width=\textwidth]{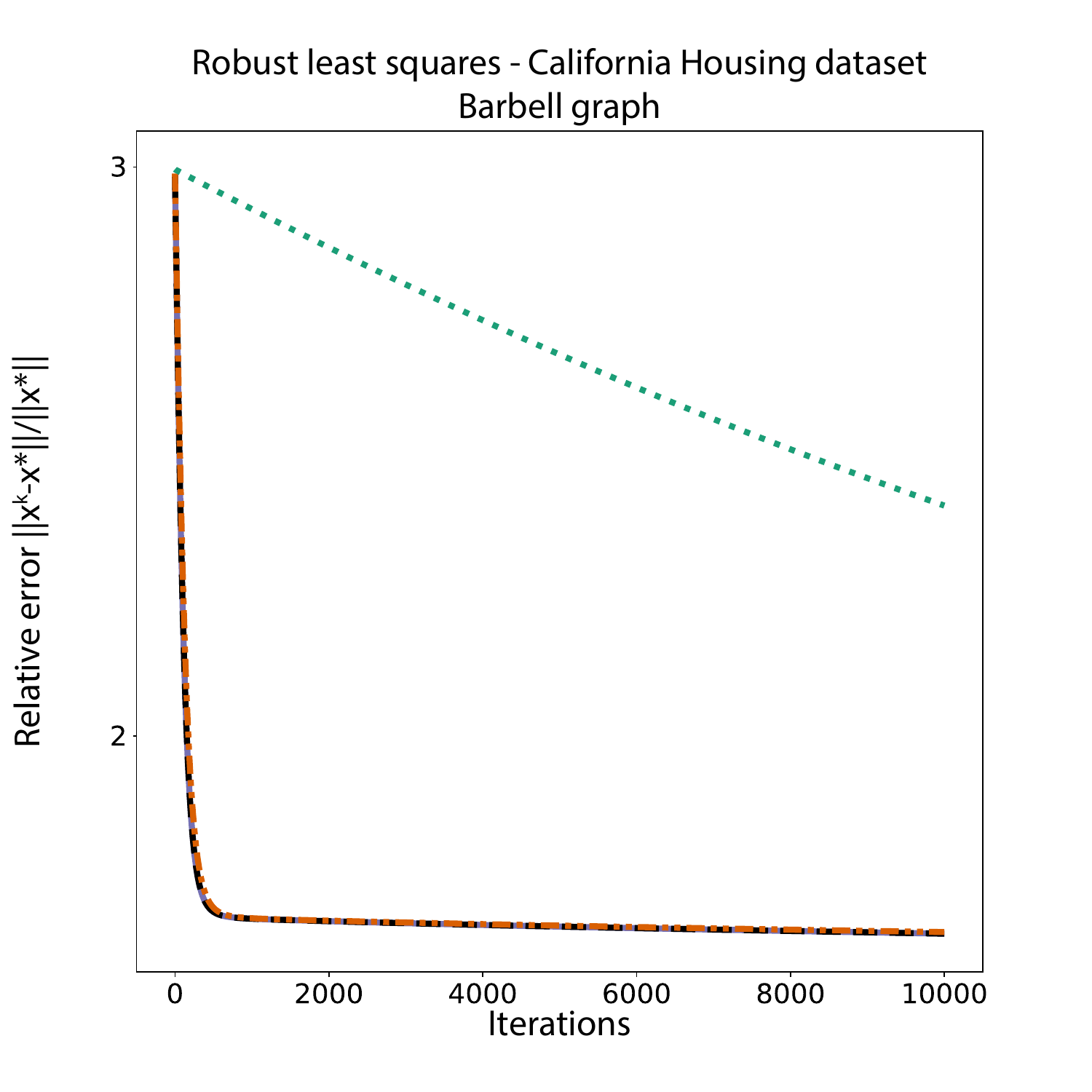}
    \caption{Barbell graph: relative error vs. iterations.}
    \label{fig:rls_california_relerror_barbell}
    \end{subfigure}
    \hfill
    \begin{subfigure}{0.45\textwidth}
    \centering
    \includegraphics[width=\textwidth]{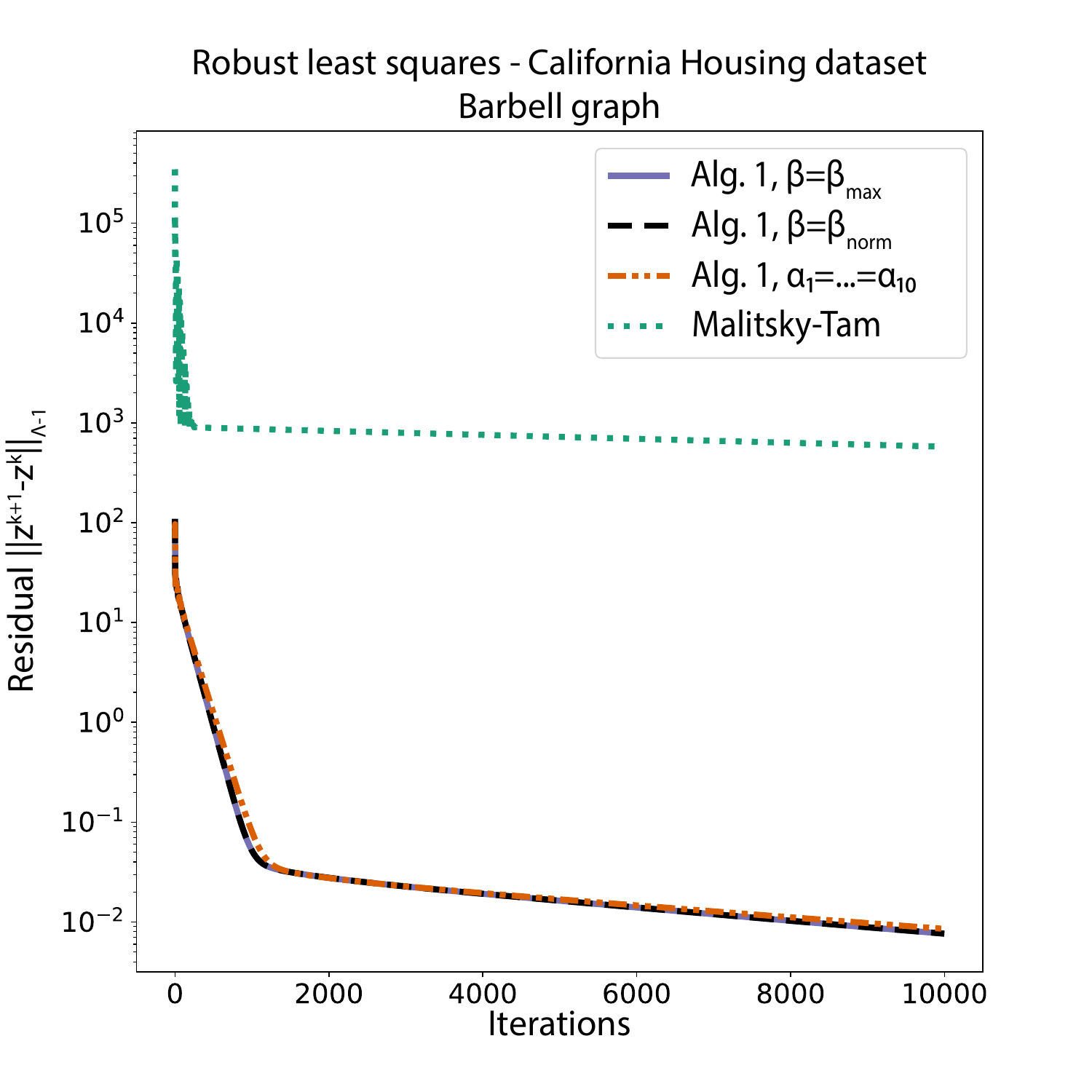}
    \caption{Barbell graph: residuals vs. iterations.}
    \label{fig:rls_california_resid_barbell}
    \end{subfigure}

    \begin{subfigure}{0.45\textwidth}
    \centering
    \includegraphics[width=\textwidth]{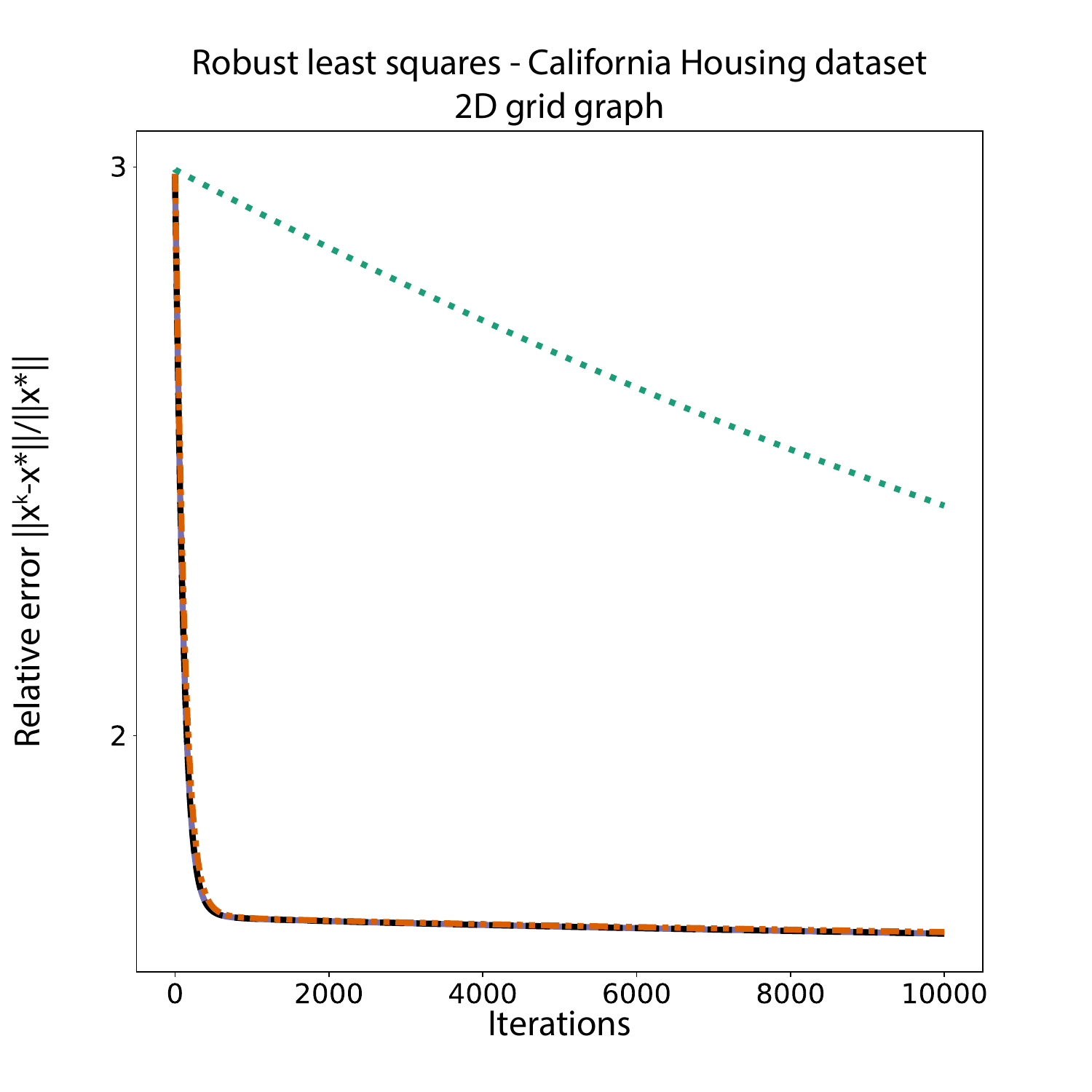}
    \caption{2D grid graph: relative error vs. iterations.}
    \label{fig:rls_california_relerror_grid2D}
    \end{subfigure}
    \hfill
    \begin{subfigure}{0.45\textwidth}
    \centering
    \includegraphics[width=\textwidth]{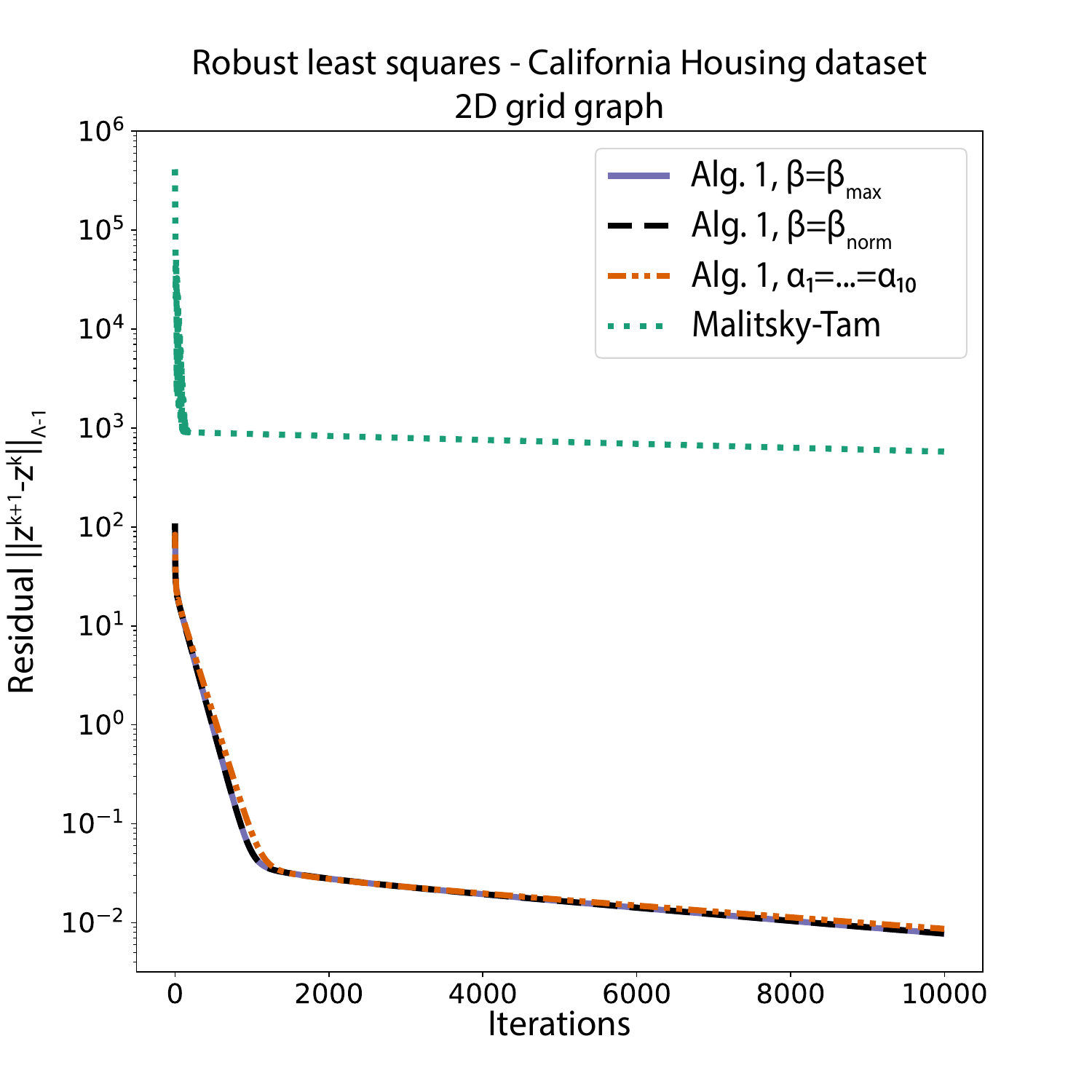}
    \caption{2D grid graph: residuals vs. iterations.}
    \label{fig:rls_california_resid_grid2D}
    \end{subfigure}
    
    \caption{Robust least squares results using the California Housing dataset and a barbell (top row) and 2D grid (bottom row) communication graphs. Step sizes are given at the start of Section \ref{sec:numerics_part1}.}
    \label{supp:fig:rls_california}
\end{figure}

% FIGURE Matrix game
\begin{figure}[ht!]
    \begin{subfigure}{0.45\textwidth}
    \centering
    \includegraphics[width=\textwidth]{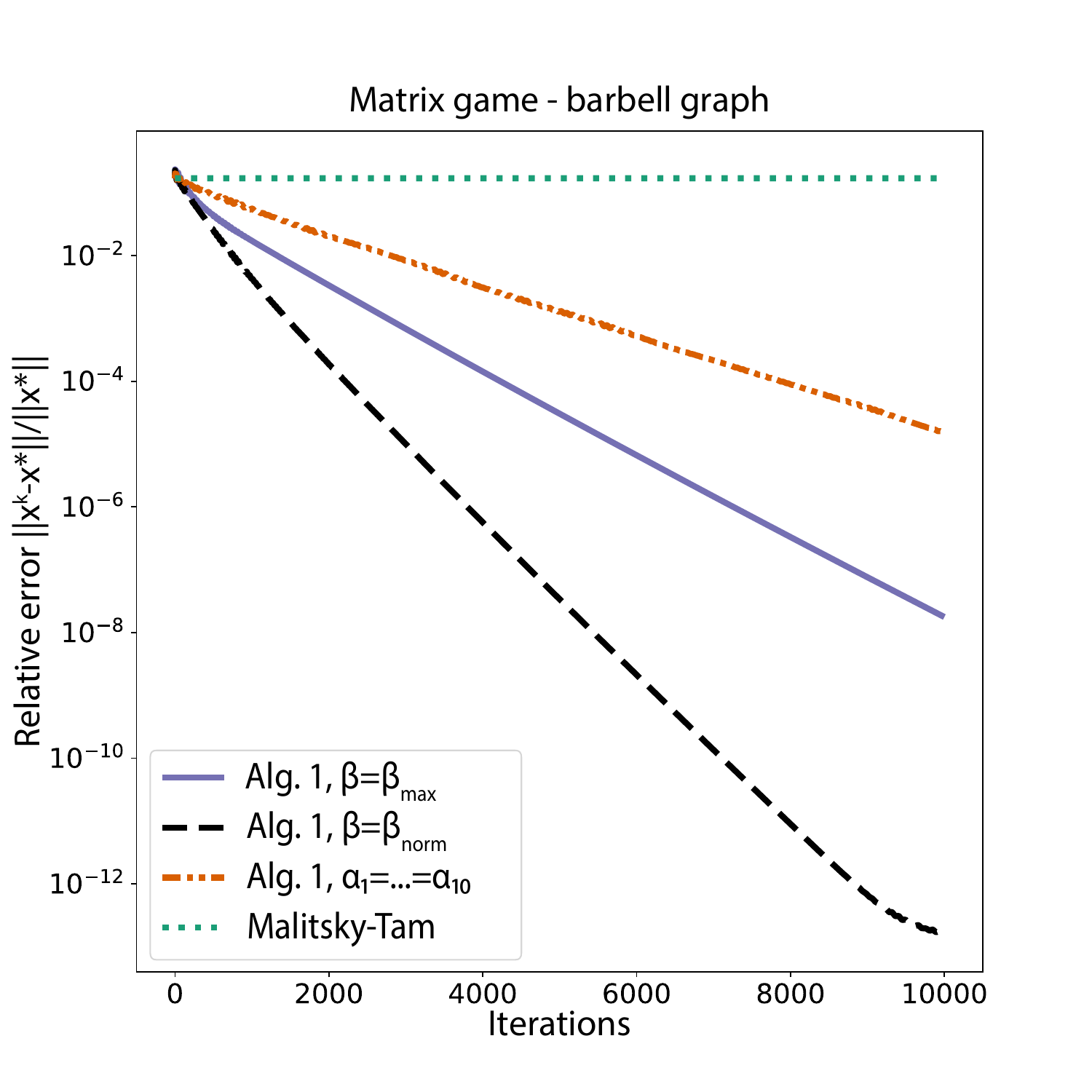}
    \caption{Barbell graph: relative error vs. iterations.}
    \label{fig:matrix_relerror_barbell}
    \end{subfigure}
    \hfill
    \begin{subfigure}{0.45\textwidth}
    \centering
    \includegraphics[width=\textwidth]{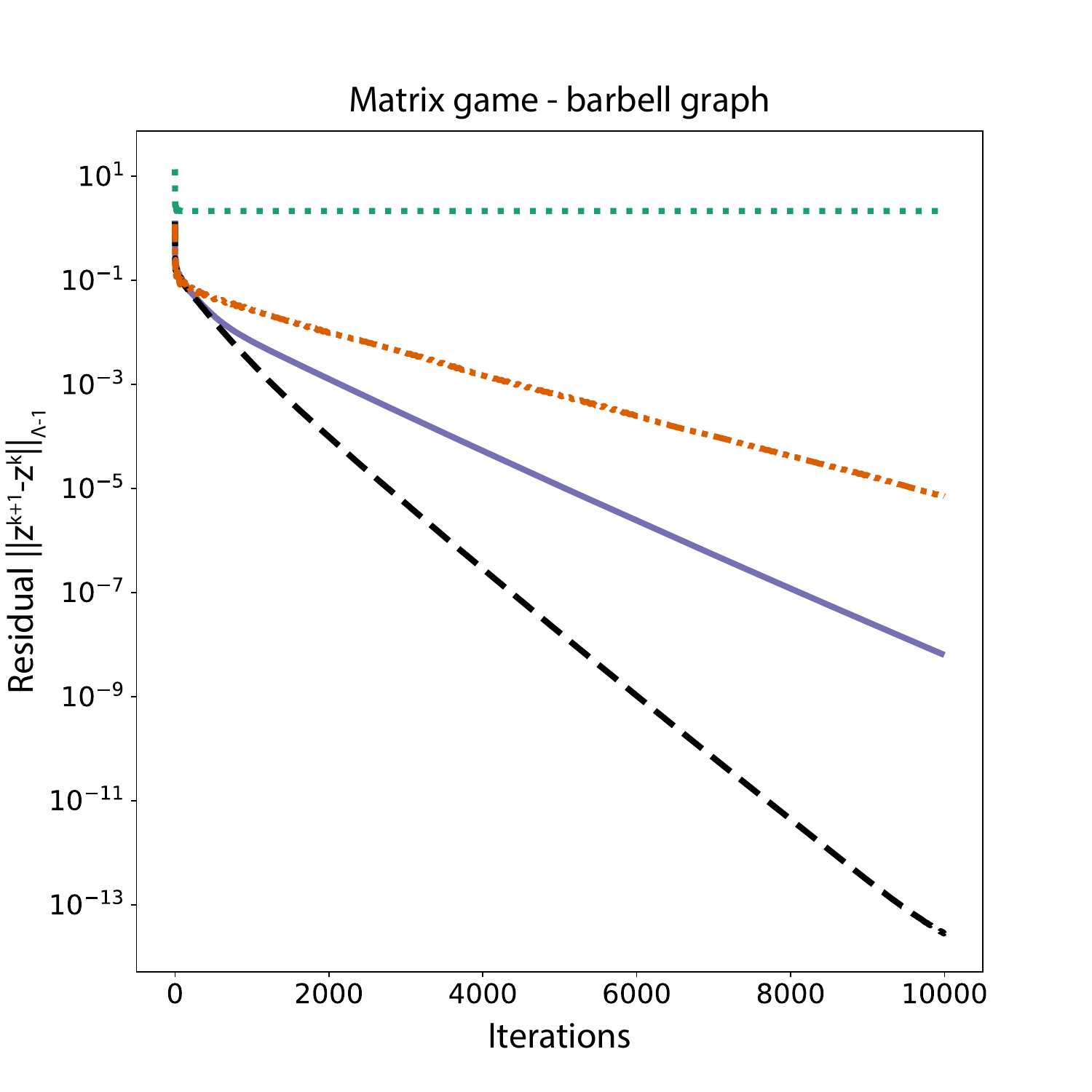}
    \caption{Barbell graph: residuals vs. iterations.}
    \label{fig:matrix_resid_barbell}
    \end{subfigure}

    \begin{subfigure}{0.45\textwidth}
    \centering
    \includegraphics[width=\textwidth]{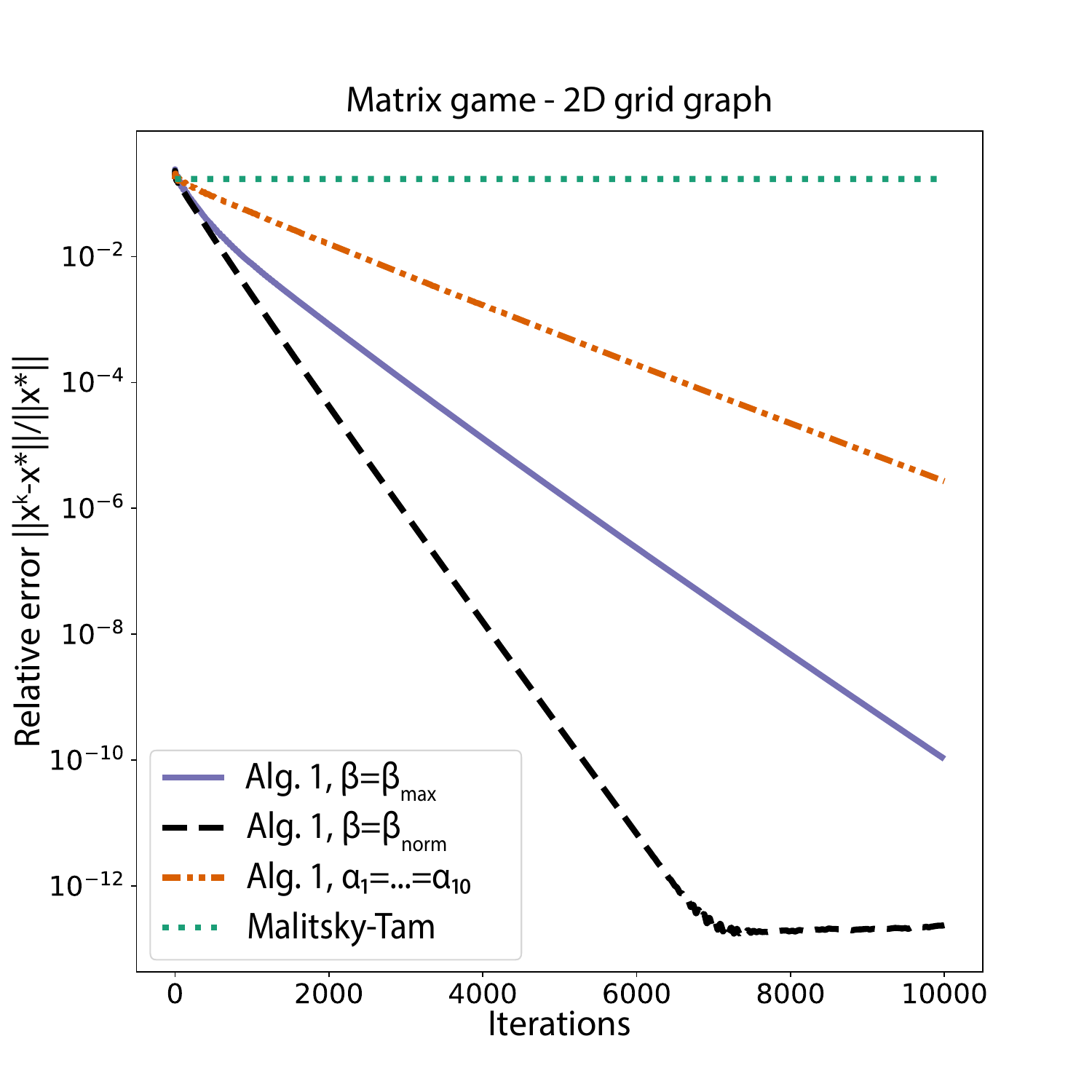}
    \caption{2D grid graph: relative error vs. iterations.}
    \label{fig:matrix_relerror_grid2D}
    \end{subfigure}
    \hfill
    \begin{subfigure}{0.45\textwidth}
    \centering
    \includegraphics[width=\textwidth]{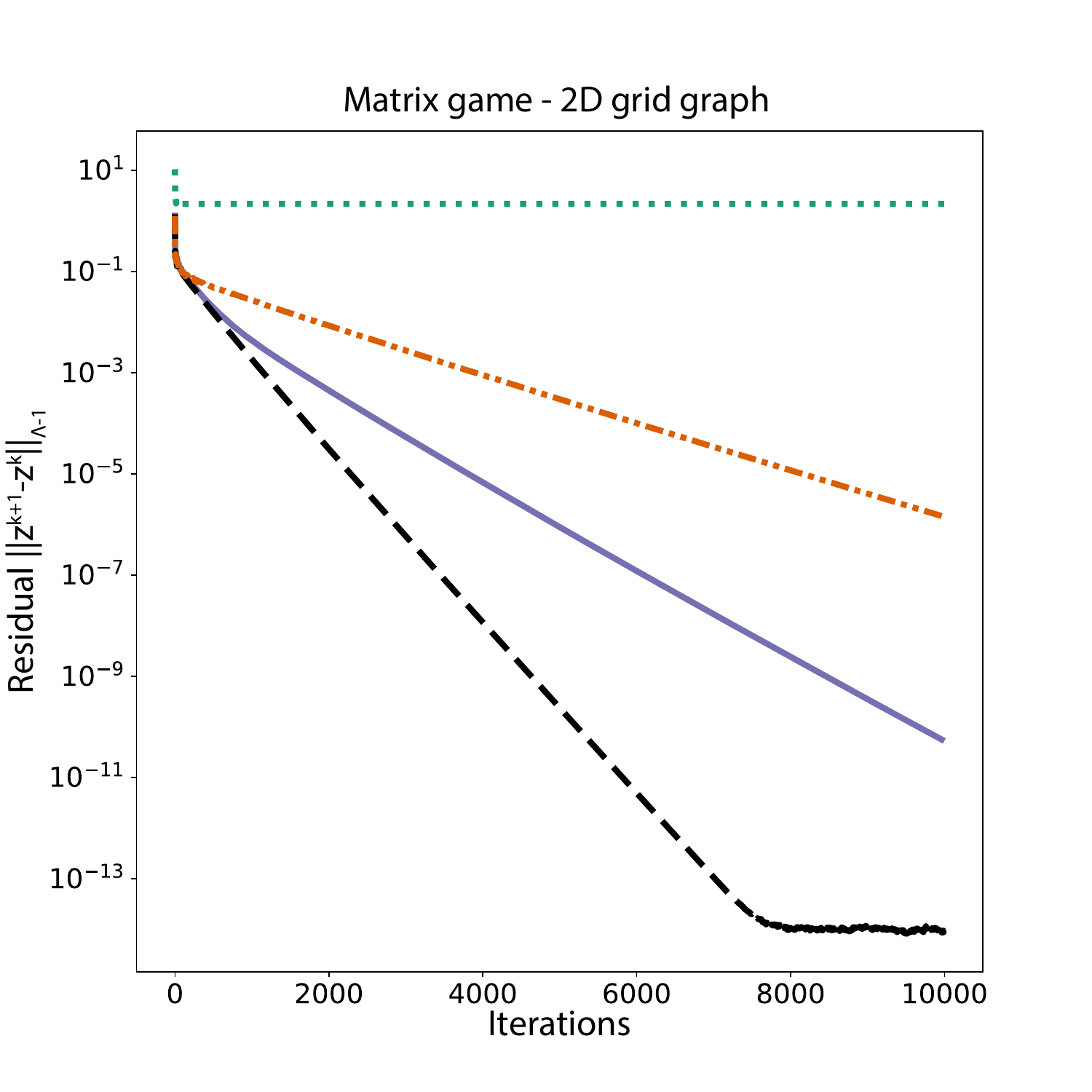}
    \caption{2D grid graph: residuals vs. iterations.}
    \label{fig:matrix_resid_grid2D}
    \end{subfigure}
    
    \caption{Matrix game results using a barbell (top row) and 2D grid (bottom row) communication graphs, averaged over five realisations of $M_1,\dots,M_N$. Step sizes are given at the start of Section \ref{sec:numerics_part1}.}
    \label{supp:fig:matrix}
\end{figure}

% TABLE VPP results
\begin{table}[!hb]
    \centering
    % Barbell results
    \begin{tabular}{c|rrrr|rrrr|rrrr}
     & \multicolumn{4}{c}{Alg.~\ref{alg:main_algorithm}} & \multicolumn{4}{c}{Alg.~\ref{alg:main_algorithm_boosted}} & \multicolumn{4}{c}{ \citep[Alg.~1]{paper:malitsky_tam_minmax}} \\
    \hline
    $N$ & \multicolumn{2}{c}{Residual} & \multicolumn{2}{c|}{Time (s)} & \multicolumn{2}{c}{Residual} & \multicolumn{2}{c|}{Time (s)} & \multicolumn{2}{c}{Residual} & \multicolumn{2}{c}{Time (s)} \\
    \hline
    20 &
    2.08 & (2.08) & 250.09 & (251.38) &
    \textbf{2.07} & (2.07) & 294.58 & (299.15) &
    14.25 & (14.28) & \textbf{198.43} & (199.71) \\
    40 & 
    4.33 & (4.33) & 511.30 & (531.91) &
    \textbf{4.29} & (4.29) & 621.28 & (630.48) &
    42.09 & (42.18) & \textbf{419.35} & (442.24) \\
    60 &
    7.11 & (7.11) & 781.39 & (785.44) &
    \textbf{7.03} & (7.03) & 1008.22 & (1016.38) &
    221.92 & (222.01) & \textbf{634.25} & (639.86) \\
    80 &
    11.04 & (11.04) & 1074.24 & (1081.22) &
    \textbf{10.79} & (10.79) & 1492.4 & (1500.59) &
    349.04 & (349.11) & \textbf{885.11} & (894.94) \\
    100 &
    \textbf{10.16} & (10.16) & 1387.55 & (1402.71) &
    10.30 & (10.30) & 2054.31 & (2064.07) &
    409.69 & (409.73) & \textbf{1170.32} & (1172.39)
    \end{tabular}
    \vspace{5mm}
    
    % 2D Grid results
    \begin{tabular}{c|rrrr|rrrr|rrrr}
     & \multicolumn{4}{c}{Alg.~\ref{alg:main_algorithm}} & \multicolumn{4}{c}{Alg.~\ref{alg:main_algorithm_boosted}} & \multicolumn{4}{c}{ \citep[Alg.~1]{paper:malitsky_tam_minmax}} \\
    \hline
    $N$ & \multicolumn{2}{c}{Residual} & \multicolumn{2}{c|}{Time (s)} & \multicolumn{2}{c}{Residual} & \multicolumn{2}{c|}{Time (s)} & \multicolumn{2}{c}{Residual} & \multicolumn{2}{c}{Time (s)} \\
    \hline
    20 &
    2.07 & (2.07) & 249.79 & (250.52) &
    \textbf{2.06} & (2.06) & 288.31 & (289.12) &
    12.41 & (12.48) & \textbf{196.46} & (197.05) \\
    40 & 
    3.82 & (3.82) & 504.26 & (506.97) &
    \textbf{3.78} & (3.78) & 599.75 & (603.02) &
    31.03 & (31.06) & \textbf{405.25} & (419.54) \\
    60 &
    5.87 & (5.87) & 780.35 & (787.57) &
    \textbf{5.82} & (5.82) & 951.92 & (953.46) &
    51.75 & (51.78) & \textbf{619.96} & (622.67) \\
    80 &
    8.56 & (8.56) & 1084.38 & (1105.01) &
    \textbf{8.48} & (8.48) & 1355.41 & (1371.66) &
    73.70 & (73.72) & \textbf{871.82} & (883.42) \\
    100 &
    12.19 & (12.19) & 1391.27 & (1400.70) &
    \textbf{12.16} & (12.16) & 1791.00 & (1804.76) &
    100.32 & (100.34) & \textbf{1154.46} & (1157.50) \\
    \end{tabular}
    \caption{VPP results using a barbell (top) and 2D grid (bottom) graph. Results are the average (and worst-case) over five repetitions, rounded to 2 decimal places. Bolded results indicate best performance amongst algorithms for a given value of $N$.}
    \label{tab:placeholder}
\end{table}

\end{document}